\documentclass[12pt]{article}

\DeclareUnicodeCharacter{0301}{{\color{red}\'{e}}}

\usepackage{cmap}  
\usepackage[T2A]{fontenc}
\usepackage[utf8]{inputenc}
\usepackage[english]{babel}
\usepackage[backend=biber,
style=alphabetic, sorting=nyt, isbn=false, date=year]{biblatex}
\AtEveryBibitem{\clearlist{language} \clearfield{urlyear}}

\usepackage[usenames]{color}
\usepackage{amsmath,amssymb,amsthm}
\usepackage[colorlinks=true,linkcolor=blue,urlcolor=red,unicode=true,hyperfootnotes=false,bookmarksnumbered]{hyperref}
\usepackage{mathtools}
\usepackage{xcolor}
\usepackage{csquotes}
\usepackage{wrapfig}
\usepackage{setspace}
\usepackage{enumitem}
\usepackage{appendix}
\usepackage{algorithm}
\usepackage{algpseudocode}
\usepackage{dsfont}
\usepackage{comment}
\usepackage{bbding}
\usepackage[nomessages]{fp}
\usepackage{tikz}

\newcommand{\ff}{\mathcal{F}}
\newcommand{\hh}{\mathcal{H}}
\newcommand{\cG}{\mathcal{G}}
\newcommand{\cS}{\mathcal{S}}
\newcommand{\cR}{\mathcal{R}}
\newcommand{\cA}{\mathcal{A}}
\newcommand{\cB}{\mathcal{B}}
\newcommand{\cQ}{\mathcal{Q}}
\newcommand{\cP}{\mathcal{P}}

\newcommand{\cT}{\mathcal{T}}
\newcommand{\cW}{\mathcal{W}}

\newcommand{\cC}{\mathcal{C}}

\newcommand{\cU}{\mathcal{U}}
\newcommand{\cI}{\mathcal{I}}
\newcommand{\bH}{\mathbf{H}}
\newcommand{\bX}{\mathbf{X}}
\newcommand{\bh}{\mathbf{h}}
\newcommand{\indicator}{\mathds{1}}

\renewcommand{\ge}{\geqslant}
\renewcommand{\le}{\leqslant}
\renewcommand{\mid}{\,:\,}

\newcommand{\coveringNumber}{\tau}

\floatname{algorithm}{Procedure}

\newcommand{\PP}{\mathbb{P}}

\newcommand{\EE}{\mathbb{E}}
\newcommand{\RR}{\mathbb{R}}
\newcommand{\bW}{\mathbf{W}}

\newcommand{\Stab}{\operatorname{Stab}}

\newcommand{\Sname}{core}

\DeclareMathOperator{\support}{supp}

\DeclareMathOperator{\poly}{poly}

\DeclareUnicodeCharacter{02BC}{\textcolor{red}{U}}

\newcounter{assumptionCounterNLarge}
\renewcommand{\theassumptionCounterNLarge}{\arabic{assumptionCounterNLarge}$'$}
{}

\newtheorem{theorem}{Theorem}
\newtheorem{proposition}[theorem]{Proposition}
\newtheorem{lemma}[theorem]{Lemma}
\newtheorem{corollary}[theorem]{Corollary}
\newtheorem{definition}{Definition}
\newtheorem{claim}[theorem]{Claim}
\newtheorem{assumption}{Assumption}
\newtheorem{example}{Example}
\newtheorem{problem}{Probem}
\newtheorem{observation}[theorem]{Observation}

\title{Linear dependencies, polynomial factors in the Duke--Erd\H os forbidden sunflower problem}
\author{Andrey Kupavskii and Fedor Noskov}
\date{September 2024}

\begin{document}

\maketitle
\begin{abstract}
    We call a family of $s$ sets $\{F_1, \ldots, F_s\}$  a \textit{sunflower with $s$ petals} if, for any distinct $i, j \in [s]$, one has $F_i \cap F_j = \cap_{u = 1}^s F_u$. The set $C = \cap_{u = 1}^s F_u$ is called the {\it core} of the sunflower. It is a classical result of Erd\H os and Rado that there is a function $\phi(s,k)$ such that any family of $k$-element sets contains a sunflower with $s$ petals. In 1977, Duke and Erd\H os asked for the size of the largest family $\ff\subset{[n]\choose k}$ that contains no sunflower with $s$ petals and core of size $t-1$. In 1987, Frankl and F\" uredi asymptotically solved this problem for $k\ge 2t+1$ and $n>n_0(s,k)$. Their paper is one of the pinnacles of the so-called Delta-system method.

    In this paper, we extend the result of Frankl and F\"uredi to a much broader range of parameters:  $n>poly(s,t) k$, when additionally $k>poly(s,t)\log n$ or $n>n_0(s,t)$. We also extend this result to other domains, such as $[n]^k$ and ${n\choose k/w}^w$ and obtain even stronger and more general results for forbidden sunflowers with core at most $t-1$ (including results for families of permutations and subfamilies of the $k$-th layer in a simplicial complex). 

    The methods of the paper, among other things, combine the spread approximation approach, introduced by Zakharov and the first author, with the Delta-system approach of Frankl and F\"uredi and the hypercontractivity  for global functions approach, developed by Keller, Lifshitz and coauthors. Previous works in extremal set theory relied on at most one of these approaches. Creating such  a unified view was one of the goals for the paper. 
\end{abstract}\newpage
\tableofcontents
\newpage
\section{Introduction}
We employ notation, standard for Extremal Set Theory: $[n]$ denotes the set $\{1,\ldots, n\}$, and, given a set $X$ and an integer $k$, we denote by $2^X$ ($\binom{X}{k}$) the family of all subsets of $X$ (all subsets of $X$ of size $k$). If a family $\ff\subset {[n]\choose k}$ then we say that $\ff$ is {\it $k$-uniform}.
We call a family of $s$ sets $\{F_1, \ldots, F_s\}$ a {\it $\Delta(s)$-system} or a \textit{sunflower with $s$ petals} if, for any distinct $i, j \in [s]$, one has $F_i \cap F_j = \cap_{u = 1}^s F_u$. The set $C = \cap_{u = 1}^s F_u$ is called the {\it core} or {\it kernel} of the sunflower. 

One of the founding papers for extremal set theory (and probably the second paper written on the subject) studies $\Delta$-systems in  families of sets. In 1960, Erd\H os and Rado \cite{erdos1960intersection} proved a result of Ramsey-theoretic flavor that states that a $k$-uniform family  of size bigger than $k!(s-1)^k$ contains a $\Delta(s)$-system. Let $\phi(s, k)$ be the maximum size of a $k$-uniform family $\cT$ that does not contain a $\Delta(s)$-system. Erd\H os and Rado also showed that $\phi(s,k)> (s-1)^{k+1}$ and stated that ``it is not improbable that'' $\phi(s, k) \le (C s)^{k+1}$, where $C$ is some absolute constant. This is one of the most important unsolved problems in Extremal Combinatorics and one of the  Paul Erd{\H{o}}s' 1000\$ problems. 

Over a long period of time, the best known upper bound on $\phi(s,t)$ was of the form $t^{t (1 + o(1))}$, where the $o(1)$ term depends on $s$, until Alweiss, Lovett, Wu and Zhang in a breakthrough paper~\cite{alweiss2021improved} have proved that $\phi(s, t) \le (C s^3 (\ln t) \ln \ln t)^t$ for some absolute constant $C$. 
This result was  pushed a bit further  in several papers~\cite{rao_coding_2020, tao_sunflower_2020, bell_note_2021}. The current record $\phi(s, t) \le (C s \ln t)^t$ with an absolute constant $C$ is   due to Bell, Chueluecha and Warnke~\cite{bell_note_2021}. 


This paper is devoted to the study of the  following closely related problem, which dates back to Duke and Erd{\H{o}}s~\cite{duke1977systems}.

\begin{problem}[Duke and Erd\H os, 1977]
\label{problem: sunflower problem}
Suppose that a family $\ff \subset \binom{[n]}{k}$ does not contain a sunflower with $s$ petals and  core of size $t - 1$. What is the maximum size of $\ff$?
\end{problem}

In 1970s, Erd\H os and coauthors posed several very general problems that became defining for the field of Extremal Set Theory. Problem~\ref{problem: sunflower problem} generalizes  the famous forbidden one intersection problem of Erd{\H{o}}s and S\'os. 
In 1971, Erd{\H{o}}s and S\'os (see~\cite{erdos1975problems}) conjectured\footnote{The original conjecture of Erd\H os and S\'os is for families avoiding intersection size $1$, the general case is due to Erd\H os.} that for $k \ge 2t$ and $n>n_0(k)$, the largest family $\ff\subset {[n]\choose k}$ with no $A,B\in \ff$ such that $|A\cap B|=t-1$  has the form
\begin{align}
\label{eq: star construction}
    \ff_T = \left \{ F \in \binom{[n]}{k} \mid T \subset F \right \}
\end{align}
for some $T$ of size $t$. The conjecture was resolved by Frankl and F\"uredi in a highly influential paper~\cite{frankl_forbidding_1985}. 
Great progress was made by Keller and Lifshitz~\cite{keller2021junta} who solved the Erd\H os-S\'os problem for $n$  \textit{linear} in $k$, but with unspecified dependence on $t$. Their result was further refined  by Ellis, Keller and Lifshitz~\cite{ellis2024stability}, who managed to extend it to the regime $k \le (1/2 - \varepsilon)n$ for arbitrary small $\varepsilon$ and $n>n_0(t,\varepsilon)$. Recently, Kupavskii and Zakharov~\cite{kupavskii_spread_2022} showed that any extremal example in the Erd\H{o}s-S\'os problem is isomorphic to~\eqref{eq: star construction} under \textit{polynomial} dependencies between $n, k$ and $t$: i.e.,  for $n =k^\alpha$ and $t=k^\beta$ with $0<\beta<1/2$ and $\alpha>1+2\beta$, provided $k>k_0$.
A part of the reason the problem of Erd{\H{o}}s and S\'os received a lot of attention is because of its numerous applications: in discrete geometry~\cite{frankl1990partition}, communication complexity~\cite{sgall1999bounds} and quantum computing~\cite{buhrman1998quantum}. We remark that, in the applications, it is important to allow $t$ to grow with $n$. 

In 1978, Deza,  Erd\H os and Frankl~\cite{deza_intersection_1978} studied  the following  generalization of the Erd\H os--Ko--Rado theorem \cite{erdos1961intersection} and the Erd\H os--S\'os problem. For positive integers $n,k$ and a set $L\subset \{0,\ldots,k-1\}$, an {\it $(n,k,L)$-system} is a family 
$\ff\subset {[n]\choose k}$ such that $|A\cap B|\in L$ for any $A,B\in L$. Deza,  Erd\H os and Frankl proved a general bound on the size of the largest $(n,k,L)$-system with $|L|=r$, provided $n>n_0(k).$ In this paper, they introduced the Delta-system method, that makes use of large sunflowers in the family in order to discern its lower-uniformity structure.  The method was later developed in the works by Frankl and F\"uredi and  had a huge impact on Extremal Set Theory, see~\cite{Frankl1987, frankl_forbidding_1985, Fur1991, mubayi2016survey}. We discuss some aspects of the Delta-system method in Section~\ref{section: delta-system method}. We note here that  `forbidden intersections' or `allowed intersections' problems often have algebraic flavour. In particular, there are numerous results that use  linear-algebraic methods~\cite{babai_linear_1992}, as well as many of these questions are directly related to  the existence of designs of certain type (cf., e.g., \cite{frankl1983extremal}).  

Another particular case of Problem~\ref{problem: sunflower problem} of great importance is as follows. We say that sets $F_1, \ldots, F_s$ form a \textit{matching} of size $s$ if they are pairwise disjoint. A matching is a sunflower with $s$ petals and empty core. That is, if we substitute $t = 1$ in Problem~\ref{problem: sunflower problem}, we get the following famous problem. The Erd{\H{o}}s Matching Conjecture \cite{erdos1965problem} states that any family $\ff\subset {[n]\choose k}$ with no matching of size $s$ satisfies
\begin{align*}
    |\ff| \le \max \left \{ \binom{k s - 1}{k}, \binom{n}{k} - \binom{n - (s - 1)}{k} \right \},
\end{align*}
provided $n \ge k s - 1$. Erd\H{o}s~\cite{erdos1965problem} proved the conjecture for $n \ge f(k) s$. After numerous improvements~\cite{bollob1976sets, huang2012size, frankl2012matchings,frankl_improved_2013}, it was shown~\cite{frankl2022erdHos} that the conjecture holds for $n \ge \frac{5}{3} sk$ and $s>s_0$. In \cite{kolupaev2023erdHos} the conjecture was proved in the regime when $n$ is only slightly bigger than $sk$.


Let us return to the question of Duke and Erd\H os. In their seminal paper~\cite{Frankl1987}, Frankl and F\"uredi derived the following asymptotic bound on a sunflower-free  family $|\ff|$ when $s, k, t$ are fixed and $n$ tends to infinity. 

\begin{theorem}
\label{theorem: asymptotic bound due to Frankl}
Let $s, k, t$ be integers, $k \ge 2t + 1$. Suppose that $s, k, t$ are fixed and $n\to \infty$. Then
\begin{align*}
    |\ff| \le (\phi(s, t) + o(1)) \binom{n}{k - t},
\end{align*}
where $\phi(s, t)$ is the maximal size of a $t$-uniform family $\cT$ that does not contain a sunflower (with arbitrary core).
\end{theorem}
Frankl and F\"uredi also showed that this bound is asymptotically the best possible. The extremal example is given in Section~\ref{section: main result}.
Theorem~\ref{theorem: asymptotic bound due to Frankl} holds even though we know little about the behavior of the function $\phi(s, t)$, which actually is a major complication for the study of the problem. The proof of the theorem is one of the pinnacles of the Delta-system method. (The authors use sunflowers in order to find sunflowers!)  Previous to \cite{Frankl1987}, partial cases  $t \in \{2, 3\}$, $k = 3$ and arbitrary $s$  were studied by Duke and Erd\H{o}s~\cite{duke1977systems}, Frankl~\cite{frankl1978extremal}, Chung~\cite{chung1983unavoidable} and Frankl and Chung~\cite{chung1987maximum}. Recently, the case $k = 4$ was studied by Buci{\'c} et al. in~\cite{bucic2021unavoidable}. Later, Brada{\v{c}}, Buci{\'c} and Sudakov~\cite{bradavc2023turan} proved that an extremal family in Problem~\ref{problem: sunflower problem} has size $O_k(n^{k - t} s^t)$ when $n$ tends to infinity, $s$ grows arbitrarily with $n$, and $k$ is fixed.


In this paper, we deal with the case when $n$ is potentially  {\it linear} in $k$, while $s$ is small. The following is a coarse version of our main result, stated and discussed in  Section~\ref{section: main result}.

\begin{theorem}
\label{theorem: simplified version of the main result}
    Fix integers $n, k, t$, $t \ge 2$. There exists a function $f_0(s, t)$, polynomial in both $s$ and $t$, such that if $n \ge f_0(s, t) \cdot k$, $k \ge 2t + 1$ and $n \ge n_0(s, t)$, then the following holds. Let $\ff \subset \binom{[n]}{k}$ be a family without a sunflower with $s$ petals and the core of size $t - 1$. Then, we have
    \begin{align*}
        |\ff| \le \left ( 1 + \sqrt[3]{\frac{f_0(s, t) \cdot k}{n}}\right )\phi(s, t) \binom{n - t}{k - t},
    \end{align*}
\end{theorem}

Moreover, we prove similar results for subfamilies of permutations and subfamilies of $\binom{[n]}{k/w}^w$, see Section~\ref{section: erdos-duke for spread domains}.

The following question has been raised as one of the initial problems in the proposal by Kalai for the Polymath project \cite{Poly} on the Erd\H os--Rado Delta-system  conjecture. In their paper~\cite{bradavc2023turan}, Brada{\v{c}} et al. reiterated it. How big could a family $\ff$ be if we forbid sunflowers with the core \textit{at most} $t - 1$? While Problem~\ref{problem: sunflower problem} generalizes the Erd\H{o}s Matching Conjecture and the Erd\H{o}s-S\'{o}s forbiden one intersection problem, this question generalizes the Erd\H{o}s Matching Conjecture and the Erd\H{o}s--Ko--Rado Theorem~\cite{erdos1961intersection}. In this paper, we answer this question.

\begin{theorem}
\label{theorem: no sunflower with down-closed core}
Suppose that $n/k \ge 4 s^2 (t + 1)^3 \log_2^3 (n/k)$ and $n \ge 2^{10} s k \log_2 k$. Suppose that a family $\ff \subset \binom{[n]}{k}$ does not contain a sunflower with $s$ petals and  core of size at most $t - 1$. Then, the following bound holds:
\begin{align*}
    |\ff| \le \left ( 1 +  \frac{2^{3} s^2 (t + 1)^3 \log_2^3 (n/k)}{n/k} \right ) \phi(s, t ) \cdot \binom{n - t}{k - t}.
\end{align*}
\end{theorem}

In Section~\ref{section: proof of down-closed core theorem} we show that the theorem is tight up to second-order terms. Moreover, we prove that similar results hold if $\ff$ is a subfamily of any of the following domains:
\begin{itemize}
    \item $[n]^k$;
    \item $\prod_{i = 1}^w \binom{[n]}{k_i}$;
    \item the set of permutations $\Sigma_n$ on $n$ elements;
    \item the $k$-th layer of a simplicial complex with rank at least $n$.
\end{itemize}
All these cases can be included in the general framework developed in~\cite{kupavskii_spread_2022, kupavskii2023erd, kupavskii2023intersection}, which we further develop in this paper. See Section~\ref{section: proof of down-closed core theorem} for the general statement and proofs. In the remaining part of the introduction, we discuss the methods of this and other works.


The starting point for us is the result of Alweiss, Lovett, Wu and Zhang~\cite{alweiss2021improved}. The core of their proof is the so-called spread lemma. Very roughly, the lemma states that a ``pseudorandom'' family $\ff \subset \binom{[n]}{k}$ is likely to contain a matching of size $s$ that with good probability respects a random partition, for an appropriate notion of pseudorandomness. The authors show that for any large enough family $\ff$ there is a set $A$, $|A| < k$, such that the restriction
\begin{align*}
    \ff(A) = \{F \setminus A \mid F \in \ff \text{ and } A \subset F\}
\end{align*}
is ``pseudorandom'', and, therefore, contains a matching $F_1, \ldots, F_s$. Thus, $A \cup F_1, \ldots, A \cup F_s \in \ff$ is a sunflower in $\ff$ of size $s$. We discuss this argument and its applications in our setting in Section~\ref{subsection: spread lemma}. Based on the spread lemma and the notion of {\it $\tau$-homogeneity} that they introduced, Kupavskii and Zakharov~\cite{kupavskii_spread_2022} developed the {\it spread approximation} technique and used it to obtain progress in several hypergraph Tur\'an-type  
problems. The technique was refined in the follow-up papers~\cite{kupavskii2023erd, kupavskii2023intersection, kupavskii2024almost}. That was the first method giving polynomial bounds in the Erd{\H{o}}s--S\'os problem. The method works when $n \ge \poly(t) k \log k$ and $k \ge \poly(t) \log n$. 
In this paper,  we add several new structural ingredients and merge their approach with the Delta-system technique and the method of Brada\v{c} et al~\cite{bradavc2023turan} to extend the spread approximation results beyond $k \ge \poly(t) \log n$, see Section~\ref{section: proof of theorem for large uniformity}. 
We should say that for $k<\sqrt n$, we no longer need to rely on the spread lemma, and develop other combinatorial techniques instead, which, nevertheless, are in line with the general spread approximations framework.

One fruitful point of view on the spread lemma is as on a \textit{sharp threshold} result, see Section~\ref{subsection: boolean analysis} for the definition of sharp thresholds and its applications in our setting. This connection was used to settle Talagrand's conjecture on fractional-expectation thresholds in random graphs~\cite{frankston2021thresholds} and the well-known Kahn---Kalai conjecture~\cite{park2024proof}. Spread lemma, however, is not the only way to establish sharp thresholds in pseudorandom families. In~\cite{keller2021junta}, Keller and Lifshitz showed that a type of pseudorandom families called {\it uncapturable} admits sharp threshold phenomenon. Based on this fact, they extended the  \textit{junta method} of Dinur and Friedgut~\cite{dinur_intersecting_2009}, and using it they obtained  significant progress in the famous Chv\'atal Simplex conjecture. Later, their method was refined by Keevash et al.~\cite{keevash2021global} and applied to a number of hypergraph Turan-type problems. Finally, Keller et al.~\cite{keller_sharp_2023} using a variant of the notion of $\tau$-homogeneity of Zakharov and the author as the notion of pseudorandomness, showed that such pseudorandom Boolean functions admit \textit{hypercontractivity} property (see Section~\ref{subsection: boolean analysis} for definitions). This lead to new results in sharp thresholds for pseudorandom families and  has found several consequences in Extremal Combinatorics and beyond. We refer the reader to the follow-up papers of Keller et al.~\cite{keller_t-intersecting_2024, keller2023improved}. We apply their results to study Problem~\ref{problem: sunflower problem} when $n = \poly(s, t) k$.

 At this point of development of these methods, it seems that none of the Delta-system, spread approximation  and hypercontractivity approaches can be fully replaced by the others. One of the goals of this paper is to take the best of each of these methods and combine them to investigate Problem~\ref{problem: sunflower problem}.

The paper is organized as follows. In Section~\ref{section: notation}, we introduce the notation used in our paper. In Section~\ref{section: main result}, we discuss Theorem~\ref{theorem: simplified version of the main result} and present its version with the dependencies stated explicitly. In Section~\ref{section: proof of down-closed core theorem}, we give a proof of Theorem~\ref{theorem: no sunflower with down-closed core} and present a key tool for the proof of  Theorem~\ref{theorem: simplified version of the main result}. In Section~\ref{section: preliminaris-ii}, we introduce certain preliminary results necessary for the proof of Theorem~\ref{theorem: simplified version of the main result}. In Section~\ref{section: proof of the main result}, we present the reduction of Theorem~\ref{theorem: simplified version of the main result} to several cases, which are treated in Section~\ref{section: proof of theorem for large uniformity} and Section~\ref{section: proof for small k theorem}. In Section~\ref{section: erdos-duke for spread domains}, we generalize our results for subfamilies of permutations and $[n]^k, \binom{[n]}{k/w}^w$. Section~\ref{section: proof sketch for product case -- small k uniformity} contains some proofs omitted in Section~\ref{section: erdos-duke for spread domains}.

\section{Notation}
\label{section: notation}

Given a set $X$ and an integer $h$, we denote by $\binom{X}{h}$ the family of subsets of $X$ of size $h$. Families of tuples $\binom{[n]}{k/w}^w$ and $[n]^k$ are understood as subfamilies of $\binom{[wn]}{k}$ and $\binom{[kn]}{k}$ respectively.

For a family $\ff$ and sets $A, B$, we define
\begin{align*}
    \ff(A, B) = \{F \setminus B \mid F \in \ff \text{ and } F \cap B = A\}. 
\end{align*}
If $A = B$, we denote $\ff(B) = \ff(B, B)$. If $B$ is a singleton $\{x\}$, we omit braces and use $\ff(x) = \ff(\{x\})$. 
Next, if $\cB$ is a family of sets, then we define
\begin{align*}
    \ff[\cB] = \{F \in \ff \mid \exists B \in \cB \text{ such that } B \subset F\}.
\end{align*}
If $\cB = \{B\}$ is a singleton of a set $B$, we omit braces and use $\ff[B] = \ff[\{B\}]$. We define the support of a family $\ff$ as $\support(\ff) = \bigcup_{F\in \ff}F$. For families $\ff$ and $\cB$, we define
\begin{align*}
    \ff \vee \cB = \{F \cup B\mid F \in \ff \text{ and } B \in \cB\}.
\end{align*}

A family $\ff$ is  {\it $k$-uniform} ({\it $\le k$-uniform}), if all sets in $\ff$ have size exactly $k$ (at most $k$).
Given a family $\ff$ and a number $h$, we define its $h$-th layer $\ff^{(h)}$ as follows:
\begin{align*}
    \ff^{(h)} = \{F \in \ff \mid |F| = h \}.
\end{align*}
The {\it shadow} of a family $\ff$ on the layer $h$ is defined as follows:
\begin{align*}
    \partial_h \ff = \bigcup_{F \in \ff} \binom{F}{h}.
\end{align*}
We put $\partial_{\le h} \ff = \bigcup_{h' \le h} \partial_{h'} \ff$. For a family $\ff \subset 2^{[n]}$, we define the {\it upper-closure} of a family $\ff$ as follows:
\begin{align*}
    \ff^\uparrow = \left \{S \subset [n] \mid \exists F \text{ such that } F \subset S \right \}.
\end{align*}

We say that a set $T$ is a {\it transversal} of a family $\ff$, if each $F \in \ff$ intersects $T$. The {\it covering number} $\coveringNumber(\ff)$ of a family $\ff$ is the minimal cardinality of its transversal. This notation should not be confused with $\tau$-homogeneity, defined below.

\section{Main result}
\label{section: main result}

Here we present our main result with the parameter dependencies stated explicitly. Theorem~\ref{theorem: simplified version of the main result} is its obvious corollary.

\begin{theorem}
\label{theorem: main theorem}
    For integers $n, k, t$, $t \ge 2$, suppose that 
    \begin{align*}
        n \ge 2^{11} s k \cdot \min \left \{  \log_2 k  + 2^{22} s^9 t^{11}, 2^{379} s^{42} t^{82} + 2^{23000}\right \},
    \end{align*}
    $n \ge n_0(s, t)$ and $k \ge 2t + 1$, where $n_0(s, t)$ is some function depending on $s, t$ only. Let $\ff \subset \binom{[n]}{k}$ be a family without a sunflower with $s$ petals and the core of size $t - 1$. Then it holds  that
    \begin{align}
    \label{eq: upper bound on F in main thm}
        |\ff| \le \left (1 + \sqrt[3]{\frac{2^{54} s^{15} t^{18}}{n/k}} \right ) \phi(s, t) \binom{n - t}{k - t}.
    \end{align}
    Moreover, if $k \ge 5t$, then $n_0(s, t)$ can be chosen to have the form $s^{O(t)} \cdot g_0(t)$, where $g_0(t)$ is some function depending on $t$ only.
\end{theorem}

The minimum in the lower bound on $n$ corresponds to two different strategies of obtaining bound~\eqref{eq: upper bound on F in main thm}. If $n = \Omega(sk \log k)$, then one can prove the result using the spread lemma. Otherwise, if $n \ge \poly(s, t) \cdot k$, the result can be proved using Boolean analysis. Note the latter leads to worse dependencies on $s, t$.

The proof is moved to Section~\ref{section: proof of the main result}. Intriguingly, the analysis is significantly simpler when  $k \ge 33 s t^2 \ln n$. A similar effect was observed in~\cite{keevash2021global} and~\cite{keller2021junta}. In this regime, the result holds without the assumption $n \ge n_0(s, t)$, and the lower bound of the form $n \ge \poly(s, t) \cdot k$ is enough. In the follow-up of this paper~\cite{kupavskii_noskov_followup}, we will show that one can prove Theorem~\ref{theorem: main theorem} in the regime when $n \ge \poly(s, t) \cdot k$ and $k \ge \poly(s, t)$ without the assumption that $n \ge n_0(s, t)$, i.e. the dependencies between all parameters $n, k, s, t$ are polynomial and there is no lower bound on $n$.

Compared to our theorem, Theorem~\ref{theorem: asymptotic bound due to Frankl} requires $n \ge n_0(s, k)= s^{2^{k}} \cdot \hat g_0(k)$ for some function $\hat g_0(k)$ depending on $k$ only. We use  Theorem~\ref{theorem: asymptotic bound due to Frankl} only if $k < 5t$, which requires $n \ge s^{2^{5t}} \cdot \hat g_0(t)$ in this case. We mentioned the case $k \ge 5t$ separately, since then we obtain a significant improvement in the lower bound on $n$. Unfortunately, both functions $g_0(t), \hat g_0(t)$ are doubly-exponential in $t$. 

 Note that Theorem~\ref{theorem: main theorem} is tight up to second-order terms, as the following example demonstrates.

\begin{example}[Example 2.3~\cite{Frankl1987}]
\label{example: basic example}
Suppose that $n > t \phi(s, t)$. Choose a family $\cT \subset \binom{[n]}{t}$ that does not contain a sunflower with $s$ petals and $|\cT| = \phi(s, t)$. Then, define a family $\ff$ as follows:
\begin{align*}
    \ff = \left \{F \in \binom{[n]}{k} \mid F \cap \support(\cT) \in \cT \right \}.
\end{align*}
The family $\ff$ does not contain a sunflower with $s$ petals and the core of size $t - 1$.
\end{example}

The family $\cT$ has support of size at most $t \phi(s, t)$.  Therefore, for the family $\ff$ from Example~\ref{example: basic example}, we have
\begin{align*}
    |\ff| & \ge \phi(s, t) \binom{n - t \phi(s, t)}{k - t} \ge \phi(s, t) \binom{n - t}{k - t} - t \phi^2(s, t) \binom{n - t - 1}{k - t - 1} \\
    & \ge \left (1 - \frac{t \phi(s, t) (k - t)}{n - t} \right )\phi(s, t) \binom{n - t}{k - t}.
\end{align*}

Theorem~\ref{theorem: no sunflower with down-closed core}  is closely related to Theorem~\ref{theorem: main theorem}, and even can be considered as an intermediate step in the proof of Theorem~\ref{theorem: main theorem}. Therefore, before we establish Theorem~\ref{theorem: main theorem}, we present the proof of Theorem~\ref{theorem: no sunflower with down-closed core}.

\section{Theorem~\ref{theorem: no sunflower with down-closed core} for spread domains}
\label{section: proof of down-closed core theorem}

In order to we state the promised generalizations of Theorem~\ref{theorem: no sunflower with down-closed core}, we need some preliminaries.

\subsection{Preliminaries I}

\subsubsection{Spread lemma}
\label{subsection: spread lemma}

We start this section with the following definitions:
\begin{definition}
    A probability measure $\mu$ on $2^{[n]}$ is said to be {\em $R$-spread}, if for any $X \in 2^{[n]}$ one has
    \begin{align*}
        \mu \left ( X \right ) \le R^{-|X|}.
    \end{align*}
\end{definition}

\begin{definition}
    A random set $\bW$ is called a {\em $p$-random subset of $[n]$} if elements $j \in [n]$ belong to $\bW$ independently with probability $p$.
\end{definition}

One of the main reasons why the definition of $R$-spreadness is so important for us is the following theorem called ``the Spread lemma'':
\begin{theorem}[Spread lemma]
\label{theorem: r-spread theorem}
Say $\bW$ is $m \delta$-random subset of $[n]$ if elements $i \in [n]$ belong to $\bW$ independently with probability $m \delta$. If for some $n, R \ge 1$ a measure $\mu: 2^{[n]} \to [0, 1]$ is $R$-spread and $\mathbf{W}$ is an $m \delta$-random subset of $[n]$, where $m$ is integer, then
\begin{align*}
    \PP \left [\exists F \in \support(\mu) \mid F \subset \mathbf{W} \right ] \ge 1 - \left ( \frac{2}{\log_2 (R \delta)} \right )^m |\mu|.
\end{align*}
\end{theorem} 

The spread lemma was initially established in the seminal paper by Alweis et al.~\cite{alweiss2021improved} with non-optimal dependices. The proof relies on the coding techniques in the spirit of Razborov's proof of the H{\aa}stad switching lemma (see Chapter 14.1 of~\cite{arora_computational_2009}). It was refined by Rao in~\cite{rao_coding_2020}. Later, Tao~\cite{tao_sunflower_2020} provided a new proof based on the entropy,  but his derivations contain a mistake. Later, the proof was fixed and simplified by Hu~\cite{hu_entropy_2021}. Stoeckl~\cite{stoeckl_lecture_nodate} established Theorem~\ref{theorem: r-spread theorem} with $1 + h_2(\delta) \le 2$ instead of constant $2$, where $h_2(\delta) = - \delta \log_2 \delta - (1 - \delta) \log_2 (1 - \delta)$ is the binary entropy of $\delta$.

Note that Theorem~\ref{theorem: r-spread theorem} can be applied to families of sets. We say that a family $\ff$ is {\it $R$-spread} if the measure induced by uniform sampling from $\ff$ is $R$-spread. In other words, $\ff$ is $R$-spread if and only if for each $X \in 2^{[n]}$, one has
\begin{align*}
    |\ff(X)| \le R^{-|X|} |\ff|.
\end{align*}

Note that for a $k$-uniform family $\ff$, one may obtain a restriction $X$ such that $\ff(X)$ is $R$-spread (assuming that the family consisting only of the empty set is arbitrarily spread).
\begin{observation}[Lemma 2 from~\cite{tao_sunflower_2020}]
\label{observation: R-spread restriction}
    Let $X$ be the largest set such that $|\ff(X)| \ge R^{-|X|} |\ff|$ holds. Then, the family $\ff(X)$ is $R$-spread.
\end{observation}
The observation is non-trivial, provided $|\ff| > R^k$. In that case, we have $|\ff(X)| > R^{-|X|} R^k = R^{k - |X|}$. In particular, this implies $|X| < k$.

For families of sets, Theorem~\ref{theorem: r-spread theorem} can be efficiently combined with the following lemma by Keevash et al.~\cite{keevash2021global}. We say that a family $\cG$ is \textit{upward closed} if for any $A \in \cG$ and $B, A \subset B,$ one has $B \in \cG$.

\begin{lemma}
\label{lemma: matching lemma}
    Let $X_1, \ldots, X_i$ be $p_i$-random subsets of $[n]$. Let $\cG_1, \ldots, \cG_s$ be upward closed subfamilies of $2^{[n]}$ such that
    \begin{align*}
        \PP (X_i \in \cG_i) \ge 3 s p_i.
    \end{align*}
    Then there exist disjoint sets $G_1 \in \cG_1, \ldots, G_s \in \cG_s$.
\end{lemma}

Together, Theorem~\ref{theorem: r-spread theorem} and Lemma~\ref{lemma: matching lemma} imply the following proposition which is the main tool in our analysis.

\begin{proposition}
    \label{proposition: coloring trick}
    Let $\cG_1, \ldots, \cG_s$ be $R$-spread uniform families of uniformity at most $k$. Suppose that $R > 2^{7} s \lceil \log_2 k \rceil$. Then, there exist pairwise disjoint sets $F_1 \in \cG_1, \ldots, F_s \in \cG_s$.
\end{proposition}

\begin{proof}
    Put $p = 1/8s$ and consider a $p$-random set $\bW$ of $[n]$. Choose $m = \lceil \log_2 k \rceil $. Due to Theorem~\ref{theorem: r-spread theorem}, for each $i = 1, \ldots, s$, we have
    \begin{align*}
        \PP \left (\exists F \in \cG_i \text{ such that } F_i \subset \bW \right ) \ge 1 - \left ( \frac{2}{\log_2 \left (\frac {pR} {m} \right )} \right )^m k > 1 - k 2^{-m} > \frac{1}{2}.
    \end{align*}
    We get
    \begin{align*}
        \PP \left ( \bW \in \cG_i^\uparrow \right ) > \frac{1}{2} > 3 s p.
    \end{align*}
    By Lemma~\ref{lemma: matching lemma}, there exist pairwise disjoint sets $G_i \in \cG_i^{\uparrow}$, and so for some disjoint sets $F_i \subset G_i$, $i = 1, \ldots, s$, we have $F_1 \in \cG_1, \ldots, F_s \in \cG_s$.
\end{proof}

Proposition~\ref{proposition: coloring trick} yields the following bound on $\phi(s, t)$:
\begin{claim}
\label{claim: sunflower bound}
     We have $\phi(s, t) \le (2^{7} s \log_2 t)^t$.
\end{claim}
\begin{proof} 
To obtain it, one should consider a family $|\ff| > (2^{7} s \log_2 t)^t$ of uniformity $t$, choose $R = 2^{6} s \lceil \log_2 t \rceil$ and find a maximal $X$ such that $|\ff(X)| \ge R^{-|X|} |\ff|$. Note that $|X| < t$. Observation~\ref{observation: R-spread restriction} guarantees that $\ff(X)$ is $R$-spread, and Proposition~\ref{proposition: coloring trick} implies that $\ff(X)$ contains a matching $F_1, \ldots, F_s$ of size $s$. Thus, $\ff$ contains a sunflower $X \cup F_1, \ldots, X \cup F_s$ of size $s$.
\end{proof}

Note that constant $2^{7}$ in Claim~\ref{claim: sunflower bound} is not optimal. For example, one may obtain $\phi(s, t) \le (2^6 s \log_2 t)^t$ (see~\cite{stoeckl_lecture_nodate}). For a short self-contained proof of Claim~\ref{claim: sunflower bound}, we refer the reader to a recent paper of Rao~\cite{rao2023sunflowers}.

In what follows, we will also use the following property of $R$-spread families.
\begin{claim}
\label{claim: spread removing intersections}
    Let $\cG$ be an $R$-spread family. Consider an arbitrary $X$ of size less than $R$. Then, the family $\cG(\varnothing, X)$ is $\hat R$-spread, where
    \begin{align*}
        \hat R = R - |X|
    \end{align*}
    and has size at least $\left (1 - \frac{|X|}{R} \right ) |\cG|$.
\end{claim}

\begin{proof}
    We have
    \begin{align*}
        |\cG(\varnothing, X)| \ge |\cG| - \sum_{x \in X} |\cG(x)| \ge \left (1 - \sum_{x \in X} R^{-1} \right ) |\cG| \ge \left (1 - \frac{|X|}{R} \right) |\cG|.
    \end{align*}
    Let $A$ be an arbitrary non-empty set disjoint from $X$. Then, we have
    \begin{align*}
        |\cG(\varnothing, X) (A)| \le |\cG(A)| \le R^{-|A|} |\cG| \le R^{-|A|}\left(1 - \frac{|X|}{R}\right) |\cG(\varnothing, X)| \le \hat{R}^{-|A|} |\cG(\varnothing, X)|,
    \end{align*}
    so $\cG(\varnothing, X)$ is indeed $\hat{R}$-spread.
\end{proof}

Claim~\ref{claim: spread removing intersections} yields the following simple but important corollary.

\begin{claim}
\label{claim: covering number}
Let $\cG$ be an arbitrary $R$-spread family. Then the covering number of $\cG$ is at least $R$. 
\end{claim}

\begin{proof}
    Let $X = \{x_1, \ldots, x_\ell\}$ be a transversal of $\cG$. If $\ell < R$, then, by Claim~\ref{claim: spread removing intersections}, we have
    \begin{align*}
        |\cG(\varnothing, X)| \ge \left (1 - \frac{\ell}{R} \right ) |\cG| > 0,
    \end{align*}
    contradicting the definition of a transversal.
\end{proof}

\subsubsection{Spread domains}
\label{subsection: spread domains}

An important topic in Extremal Set Theory is to study families $\ff$, satisfying some restrictions in domains $\cA$ different from $\binom{[n]}{k}$. There is a wide range of such works, for example, generalizations of Erd\H{o}s--Ko--Rado theorem for vector fields~\cite{frankl1986erdos}, integer sequences~\cite{frankl_erdos-ko-rado_1999}, layers of simplicial complexes~\cite{borg2009extremal} and permutations~\cite{frankl1977maximum}, generalizations of Kruskal--Katona Theorem for $[n]^k$~\cite{frankl_shadows_1988}, generalizations of Erd\H{o}s-S\'os problem for codes~\cite{keevash2023forbidden} and linear maps~\cite{ellis2023forbidden}.

Zakharov and the first author~\cite{kupavskii_spread_2022} developed the machinery of spread approximations that can be used to deduce extremal set theory results for subfamilies of domains different from  $\binom{[n]}{k}$. To discuss these results, we shall need the following definition.

\begin{definition}
    We say that a family $\cA$ is $(r, t)$-spread, if for any $T \in \partial_{\le t} \cA$, the family $\cA(T)$ is $r$-spread.
\end{definition}

Here, we illustrate the definition of $(r, t)$-spreadness with several examples, including $\binom{[n]}{k}$.

\begin{proposition}
\label{proposition: spread families examples}
    Fix integers $n, k, t$ such that $k \ge t$ and $n \ge k$. Then the following holds:
    \begin{enumerate}
        \item \label{item: n choose k is (r,t) spread} The family $\binom{[n]}{k}$ is $\left (\frac{n}{k}, t \right )$-spread;
        \item \label{item: n to the pow k is (r, t) spread} The family $[n]^k$ is $(n, t)$-spread;
        \item \label{item: product of families are (r,t) spread} Fix $w\le k$ and $k_1, \ldots, k_w$ such that $\sum_{i} k_i = k$. The family $\binom{[n]}{k_1} \times \ldots \times \binom{[n]}{k_w}$ is $(r, t)$-spread for $r = \min_i \frac{n}{k_i}$.
    \end{enumerate}
\end{proposition}

\begin{proof}
    It is enough to prove \ref{item: product of families are (r,t) spread}, since~\ref{item: n choose k is (r,t) spread} and~\ref{item: n to the pow k is (r, t) spread} follow from it. Set
    \begin{align*}
        \cA = \binom{I_1}{k_1} \times \ldots \times \binom{I_w}{k_w},
    \end{align*}
    where $|I_i| = n$. 
    Fix some set $T$ of size at most $t$, and let $T_i = T \cap I_i$. Since $T \in \partial_t \cA$, we have $|T_i| \le k_i$. We aim to show that the family
    \begin{align*}
        \cA(T) = \prod_{i = 1}^w \binom{I_i \setminus T_i}{k - |T_i|}
    \end{align*}
        is $r$-spread. Consider an arbitrary $S$ disjoint from $T$, and let $S_i = S \cap (I_i \setminus T_i)$. We may assume that $|S_i| \le k_i - |T_i|$, since otherwise $|\cA(T)(S)| = 0$ and the spreadness definition is trivially satisfied.  We have
    \begin{align*}
        |\cA(T)(S)| = \prod_{i = 1}^w \left | \binom{I_i \setminus T_i}{k - |T_i|} (S_i) \right | = \prod_{i = 1}^w \left | \binom{I_i \setminus (T_i \cup S_i)}{k - |T_i| - |S_i|} \right |.
    \end{align*}
    Then, we have
    \begin{equation*}
        |\cA(T)(S)|  \le \prod_{i 
        } \left ( \frac{n - |T_i|}{k_i - |T_i|} \right )^{-|S_i|} \binom{|I_i \setminus T_i|}{k_i - |T_i|}
        \le \left ( \min_i \frac{n}{k_i} \right )^{-|S|} |\cA(T)|. \qedhere
    \end{equation*}
\end{proof}

Next, we give some more sophisticated examples of families that are $(r, t)$-spread.

Consider the family $\Sigma_n$ of permutations on $[n]$. One can consider $\Sigma_n$ as a subfamily of $\binom{[n]^2}{n}$. For each $\sigma \in \Sigma_n$, correspond the following set $F_{\sigma} \in \binom{[n]^2}{n}$:
\begin{align*}
    F_{\sigma} = \left \{(i, j) \mid \sigma(i) = j \right \}.
\end{align*}
We denote by $\mathfrak{S}_n$ the family of sets $\{F_{\sigma} \mid \sigma \in \Sigma\}$. Clearly, the family $\mathfrak{S}_n$ is $n$-uniform. The following proposition ensures that it is $(r, t)$-spread for a suitable $r$.
\begin{proposition}[Section 3.4 from~\cite{kupavskii_spread_2022}]
\label{proposition: r-t spreadness of permutations}
    Suppose that $n > 4t$. Then, the family $\mathfrak{S}_n$ is $(n/4, t)$-spread.
\end{proposition}

The restriction that a family $\ff \subset \Sigma_n$ obeys are translated into the restrictions on the corresponding family $\mathfrak{F} = \{F_\sigma \mid \sigma \in \ff\}$. For example, if for any $\sigma_1, \sigma_2 \in \ff$ there exists $x$ such that $\sigma_1(x) = \sigma_2(x)$, then the family $\mathfrak{F}$ is intersecting, i.e. for any $F_1, F_2 \in \mathfrak{F}$ we have $|F_1 \cap F_2| \ge 1$. To formulate the counterpart of Theorem~\ref{theorem: no sunflower with down-closed core} for permutations, we need to elaborate the corresponding sunflower restriction for them.

\begin{definition}
\label{definition: sunflower of permutation}
    Consider $s$ permutations $\sigma_1, \ldots, \sigma_s$. If
\begin{enumerate}
    \item there exists a set $X = \{x_1, \ldots, x_{m}\}$ such that $\sigma_1(x_i) = \ldots = \sigma_s(x_i)$ for each $i = 1, \ldots, m$,
    \item for any $x \in [n] \setminus X$, all values $\sigma_1(x), \ldots, \sigma_s(x)$ are distinct,
\end{enumerate}
then we say that $\sigma_1, \ldots, \sigma_s$ form a sunflower with $s$ petals and the core of size $m$.
\end{definition}

Note that $\ff \subset \Sigma_n$ does not contain a sunflower with $s$ petals and the core of size $t - 1$ if and only if the corresponding family of sets $\mathfrak{F} \subset \mathfrak{S}_n$ does not contain a sunflower of $s$ petals and the core of size $t - 1$. Using this correspondence, we can establish the counterpart of Theorem~\ref{theorem: no sunflower with down-closed core} for permutations.

Another non-obvious example of an $(r, t)$-spread family is the $k$-th layer of a simplicial complex of large rank. The {\it rank} of a simplicial complex is the size of its smallest inclusion-maximal set.
Consider a simplicial complex $\cC$ of rank $n$. Let $n \ge k \ge t$. Consider the $k$-th layer $\cC^{(k)}$ of $\cC$:
\begin{align*}
    \cC^{(k)} = \{F \in \cC \mid |F| = k\}.
\end{align*}
The following lemma states that the family $\cC^{(k)}$ is $(r, t)$-spread for a suitable $r$.
\begin{lemma}[Lemma 4 from~\cite{kupavskii2023intersection}]
\label{lemma: r-t spreadness of simplicial complex}
    The family $\cC^{(k)}$ is $(n/k, t)$-spread.
\end{lemma}

Thus, there is a wide range of $(r, t)$-spread families $\cA$. In Section~\ref{subsection: proof of down-closed theorem}, we show that for each such $\cA$ and $\ff \subset \cA$ some counterpart of Theorem~\ref{theorem: no sunflower with down-closed core} holds.

\subsection{Theorem~\ref{theorem: no sunflower with down-closed core} for spread domains}
\label{subsection: proof of down-closed theorem}

In order to formulate a generalization of Theorem~\ref{theorem: no sunflower with down-closed core} for $\ff$ being a subfamily of $(r, t)$-spread domain (or ambient family) $\cA$, we require some natural equivalents of the quantities used in the bound on $|\ff|$. One of such quantities replaces $\binom{n - t}{k - t}$. We denote by $A_t$ the following quantity:
\begin{align*}
    A_t = \max_{T \in \partial_t \cA} |\cA(T)|.
\end{align*}
Using this notation, we present the main result of this section.

\begin{theorem}
\label{theorem: general theorem for down-closed sunflowers}
    Let $\cA$ be a $k$-uniform and $(r, t)$-spread ambient family. Assume that $r \ge 4 s^2 (t + 1)^3 \log_2^3 r$ and $r \ge 2^{10} s \log_2 k$. Suppose that a family $\ff \subset \cA$ does not contain a sunflower with $s$ petals and the core of size at most $t - 1$. Then, there exists a family $\cT \subset \partial_{t} \cA$ that does not contain a sunflower with $s$ petals such that
    \begin{align*}
        |\ff \setminus \ff[\cT]| \le \phi(s, t) \cdot \frac{2^{3} s^2 (t + 1)^3 \log_2^3 r}{r} \cdot A_t.
    \end{align*}
\end{theorem}
We prove Theorem~\ref{theorem: general theorem for down-closed sunflowers} in Section~\ref{subsection: proof of the general spread theorem for down-closed sunflowers}.

Obviously, the theorem yields the following upper bound on $\ff$:
\begin{align}
\label{eq: upper bound on down-closed sunflower free families}
    |\ff| \le \max_{\cT} |\cA[\cT]| +\phi(s, t) \cdot \frac{2^{3} s^2 (t + 1)^3 \log_2^3 r}{r} \cdot A_t,
\end{align}
where the maximum is taken over all families $\cT \subset \partial_t \cA$ that do not contain a sunflower with $s$ petals. If $\cA = \binom{[n]}{k}$, then we have $r = n/k$ due to Proposition~\ref{proposition: spread families examples}, and the following bound holds:
\begin{align*}
    |\ff| \le  \left ( 1 +  \frac{2^{3} s^2 (t + 1)^3 \log_2^3 (n/k)}{n/k} \right ) \phi(s, t) \cdot \binom{n - t}{k - t}.
\end{align*}
Thus, Theorem~\ref{theorem: no sunflower with down-closed core} follows.

We claim that the bound $\eqref{eq: upper bound on down-closed sunflower free families}$ is tight up to remainder terms of order $\log_2^3r/r$. The corresponding example below generalizes that of~\cite{Frankl1987}. Consider a family $\cT^* \subset \partial_t \cA$ that maximizes $|\cA[\cT]|$ and does not contain a sunflower with $s$ petals. Denote
\begin{align*}
    \support \cT^* = \bigcup_{T \in \cT^*} T.
\end{align*}
Then, define
\begin{align*}
    \ff^* = \bigcup_{T \in \cT^*} \cA(T, \support \cT^*) \vee \{T\}.
\end{align*}
We claim that $\ff^*$ does not contain a sunflower with $s$ petals and the core of size at most $t - 1$. Indeed, if such a sunflower $F_1, \ldots, F_s$ exists, then $F_1 \cap \support \cT^*, \ldots, F_s \cap \support \cT^*$ form a sunflower with $s$ petals and the core of size at most $t - 1$. By construction, $F_1 \cap \support \cT^*, \ldots, F_s \cap \support \cT^*$ belongs to $\cT^*$, a contradiction.

Next, we have
\begin{align*}
    |\cA[\cT^*]| - |\ff^*| \le \sum_{T \in \cT^*} |\cA(T)| - |\cA(T, \support \cT^*)| \le \sum_{T \in \cT^*} \sum_{x \in (\support \cT^*) \setminus T} |\cA(T \sqcup \{x\})|.
\end{align*}
Since $\cA$ is $(r, t)$-spread, we can bound $|\cA(T \sqcup \{x\})| \le r^{-1} |\cA(T)|$. Therefore, we have
\begin{align*}
    |\cA[\cT^*]| - |\ff^*| \le \frac{|\cT^*| |\support \cT^*|}{r} \cdot A_t \le \frac{t \phi^2(s, t)}{r} A_t,
\end{align*}
so $|\ff^*| \ge \max_{\cT} |\cA[\cT]| - \frac{t \phi^2(s, t)}{r} A_t$, which matches bound~\eqref{eq: upper bound on down-closed sunflower free families} up to second-order terms.

Next, we present a corollary of Theorem~\ref{theorem: general theorem for down-closed sunflowers} for subfamilies of a simplicial complex. Using Lemma~\ref{lemma: r-t spreadness of simplicial complex}, we obtain the following.

\begin{corollary}
    Consider a simplicial complex $\cC$ with the rank at least $n$. Fix numbers $s, k, t$ such that $k \ge t$, $n/k \ge 4 s^2 (t + 1)^3 \log_2^3(n/k)$ and $n \ge 2^{10} s k \log_2 k$. Suppose that a family $\ff \subset \cC^{(k)}$ does not contain a sunflower with $s$ petals and the core of size at most $t - 1$. Then, there exists a family $\cT \subset \cC^{(t)}$ that does not contain a sunflower with $s$ petals such that
    \begin{align*}
        |\ff \setminus \ff[\cT]| \le \phi(s, t) \cdot \frac{2^{3} s^2 (t + 1)^3 \log_2^3 (n/k)}{n/k} \cdot \max_{T \in \cC^{(t)}} |\cC(T)|.
    \end{align*}
\end{corollary}

To state the corollary of Theorem~\ref{theorem: general theorem for down-closed sunflowers} for permutations, we need one more definition. We say that a function $\nu$ is a \textit{partial permutation} of size $t$, if $\nu$ is an injection from some set $T \subset [n]$ of size $t$ to $[n]$. We define the family of partial permutations of size $t$ by $\Sigma^{(t)}_n$. Definition~\ref{definition: sunflower of permutation} of a sunflower generalizes trivially for partial permutations. Since any such injection can be extended to bijection from $[n]$ to $[n]$, there is one-to-one correspondence between elements of $\Sigma^{(t)}_n$ and the shadow $\partial_t \mathfrak{S}_n$. Using it and Proposition~\ref{proposition: r-t spreadness of permutations}, we deduce the following corollary (note that in our case the uniformity $n$ and the spreadness $n/4$ are the same up to multiplicative constants, so we have only one inequality on $n$ instead of two inequalities on $r$ in Theorem~\ref{theorem: general theorem for down-closed sunflowers}).
\begin{corollary}
    Fix numbers $n, s, t$ such that $n \ge 2^{20} s (t + 1) \log_2 n$. Suppose that a family $\ff \subset \Sigma_n$ of permutations does not contain a sunflower with $s$ petals and the core of size at most $t - 1$. Then, there exists a family $\cT \subset \Sigma^{(t)}_n$ of partial permutations that does contain a sunflower with $s$ petals and the core of size at most $t - 1$, such that
    \begin{align*}
        |\ff \setminus \ff[\cT]| \le \phi(s, t) \cdot \frac{2^{5} s^2 (t + 1)^3 \log_2^3 (n/4)}{n} \cdot (n - t)!
    \end{align*}
\end{corollary}

Here $\ff[\cT]$ is a family of all permutations $\sigma$ in $\ff$ such that there exists a partial permutation $\nu \in \cT$ which coincides with $\sigma$ on the support of $\nu$.

\subsection{Proof of Theorem~\ref{theorem: general theorem for down-closed sunflowers}}
\label{subsection: proof of the general spread theorem for down-closed sunflowers}

Set $w = (t + 1) \log_2 r$. Consider two cases. 

\textbf{Case 1. We have $k \le w$}. In this case, we apply the following lemma with $\ff$ in place of $\cS$ with $q = k$.

\begin{lemma}
\label{lemma: sunflower simplification}
    Let $\cA$ be a $k$-uniform $(r, t)$-spread family. Fix some $q \le k$ and $\varepsilon \in (0, 1)$. Assume that $\varepsilon r \textcolor{red}{\ge} s^2 q^3$. Suppose that a family $\cS \subset \partial_{\le q} \cA$ 
    \begin{itemize}
        \item does not contain an $s$-sunflower with the core of size $\le t - 1$;
        \item consists of sets of cardinality at least $t$.
    \end{itemize} 
    Then, there exists a family $\cT \subset \partial_{t} \cA$ which does not contain any sunflower with $s$ petals and
    \begin{align*}
        |\cA[\cS] \setminus \cA[\cT]| \le \left |\cA\left [\cS \setminus \cS[\cT] \right ] \right | \le \phi(s, t) \cdot \frac{2\varepsilon}{1 - \varepsilon} A_t.
    \end{align*}
\end{lemma}
We prove Lemma~\ref{lemma: sunflower simplification} in Section~\ref{subsection: proof of simplification argument}. We apply it with $\varepsilon = \frac{s^2q^3}r\le \frac{s^2(t+1)^3\log_2^3 r}{r}$. Note that the condition on $r$ from the theorem implies that $\varepsilon<1.$

The lemma gives suboptimal bounds if $k \ge (t + 1) \log_2 r$. Indeed, if we apply it with $k$ in place of $q$, it gives the result for $r > s^2 k^3$, while we can establish the theorem provided $r \ge 2 s^2 (t + 1)^3 \log^3 r$ and $r \ge 2^{10} s \log_2 k$.

\textbf{Case 2. We have $k > (t + 1) \log_2 r$.} In this case, we use the following iterative process to obtain some structure within the family $\ff$, cf.~\cite{kupavskii2023erd, kupavskii2023intersection}.

Set $\ff_0 = \ff$. Given a family $\ff_{i}$, we obtain a family $\ff_{i + 1}$ by the following procedure.

Set $R = r/2$. Let $S_i$ be the maximal set such that $|\ff_{i}(S_i)| \ge R^{-|S_i|} |\ff_i|$. Note that $\ff_i(S_i)$ is $R$-spread due to Observation~\ref{observation: R-spread restriction}. If $|S_i| > w$, then stop. Otherwise, set $\ff_{i + 1} = \ff_i \setminus \ff_i[S_i].$

Suppose that the procedure above stops at step $m$. It means that $|S_m| \ge w + 1$ and
\begin{align*}
    |\ff_m| & \le R^{|S_m|} |\ff(S_m)| \le R^{|S_m|} |\cA(S_m)| \\&
    \overset{{\color{teal} (\cA \text{ is $(r, t)$-spread})}}{\le} R^{|S_m|} r^{t - |S_m|} |\cA(T)| \le  r^t \left ( \frac{R}{r} \right )^{w + 1} |\cA(T)| \le r^t2^{-w-1} |\cA(T)| \le \frac{1}{r} |\cA(T)|.
\end{align*}
for some $T \subset S_m$ of size $t$. Define $\cS = \{S_1, \ldots, S_{m - 1}\}$ and $\ff_{S_i} = \ff_i(S_i)$. Note that $\ff$ admits the following decomposition:
\begin{align*}
    \ff = \ff_{m} \sqcup \bigsqcup_{S \in \cS} \ff_S \vee \{S\},
\end{align*}
and each $\ff_S$ is $R$-spread.

We claim the following.
\begin{claim}
The family $\cS$ possesses two properties:
    \begin{enumerate}
    \item \label{item: core at most t - 1 for down-closed sunflowers}it does not contain a sunflower with $s$ petals and the core of size at most $t - 1$;
    \item \label{item: no sets less than t for down-closed sunflowers}it consists of sets of cardinality at least $t$.
\end{enumerate}
\end{claim}

\begin{proof}
    Arguing indirectly, assume that $\cS$ contains a sunflower $S_1, \ldots, S_s$ with $s$ petals and the core of size at most $t - 1$. Note that one may consider $s$ copies of a set $S$ of size at most $t - 1$ as a sunflower with $s$ empty petals and the core of size at most $t - 1$. The proof below does not require $S_1, \ldots, S_s$ being distinct, so it implies both~\ref{item: core at most t - 1 for down-closed sunflowers} and~\ref{item: no sets less than t for down-closed sunflowers}. 

    The families $\ff_{S_1}, \ldots, \ff_{S_s}$ are $R$-spread. Put $S = \cup_j S_j$. Consider the families $\cG_i = \ff_{S_i}(\varnothing, S\setminus S_i)$, $j\in[s]$. Due to Claim~\ref{claim: spread removing intersections},  $\cG_i$ are $\hat R$-spread for
    \begin{align*}
        \hat R = R-|S|\ge R-sw \ge R/2,
    \end{align*}
    since $R \ge 2 s w$. Next, $\hat R \ge R/2 \ge 2^{8} s \log_2 k \ge 2^{7} s \lceil \log_2 k \rceil$, so Proposition~\ref{proposition: coloring trick} implies that there are disjoint $G_1 \in \cG_1, \ldots, G_s \in \cG_s$. Therefore, $\ff$ contains a sunflower $S_1 \cup G_1, \ldots, S_s \cup G_s$ with $s$ petals and the core of size at most $t - 1$, a contradiction.
\end{proof}

Note that $\mathcal S\subset \partial_{\le w} \mathcal A$ by definition. Then, we apply Lemma~\ref{lemma: sunflower simplification} with $q = w$ and $\varepsilon = \frac{s^2w^3}r=\frac{s^2(t+1)^3\log_2^3 r}{r}$ to $\cS$, and obtain
\begin{align*}
    |\ff[\cS] \setminus \cA[\cT]| & \le |\cA[\cS] \setminus \cA[\cT]| \le \phi(s, t) \cdot \frac{2 \varepsilon}{1 - \varepsilon} A_t \\
    & \overset{{\color{teal}(\varepsilon \le 0.5)}}{\le} \phi(s, t) \cdot\frac{4 s^2 (t + 1)^3 \log_2^3 r}{r}A_t.
\end{align*}
Therefore, we have
\begin{align*}
    |\ff \setminus \ff[\cT]| \le |\ff_m| + |\ff[\cS] \setminus \cA[\cT]| \le \phi(s, t) \cdot\frac{2^{3} s^2 (t + 1)^3 \log_2^3 r}{r}A_t. & \qedhere
\end{align*}
This completes the proof of Theorem~\ref{theorem: general theorem for down-closed sunflowers} modulo Lemma~\ref{lemma: sunflower simplification}.

\subsection{Proof of Lemma~\ref{lemma: sunflower simplification}: a simplification argument}
\label{subsection: proof of simplification argument}

Lemma~\ref{lemma: sunflower simplification} is the key ingredient in the proofs of Theorems~\ref{theorem: main theorem} and~\ref{theorem: general theorem for down-closed sunflowers}. For a reminder, let us restate the lemma.

\vspace{0.5cm}
\noindent \textbf{Lemma~\ref{lemma: sunflower simplification}.} {\it
Let $\cA$ be a $k$-uniform $(r, t)$-spread family. Fix some $q \le k$ and $\varepsilon \in (0, 1)$. Assume that $\varepsilon r \ge s^2 q^3$. Suppose that a family $\cS \subset \partial_{\le q} \cA$ 
    \begin{itemize}
        \item does not contain an $s$-sunflower with the core of size $\le t - 1$;
        \item consists of sets of cardinality at least $t$.
    \end{itemize} 
    Then, there exists a family $\cT \subset \partial_{t} \cA$ which does not contain any sunflower with $s$ petals and
    \begin{align*}
        |\cA[\cS] \setminus \cA[\cT]| \le \left |\cA\left [\cS \setminus \cS[\cT] \right ] \right | \le \phi(s, t) \cdot \frac{2 \varepsilon}{1 - \varepsilon} A_t.
    \end{align*}
}

We will obtain $\cT$ at the end of the following process. Set $\cT_0 = \cS$. Given a family $\cT_{i}$, we will construct a family $\cT_{i + 1}$ by Procedure~\ref{algo: simplification procedure I} starting from $i = 0$.
\begin{algorithm}
    \begin{algorithmic}[1]
        \State Set $\cW_i \leftarrow \cT_i^{(q - i)}$, $\cT_{i + 1}' \leftarrow \varnothing$;
        \State Set \[
            \alpha \leftarrow 
                 sq;
        \]
        \State Find a maximal set $T$ such that $|\cW_i(T)| \ge \alpha^{-|T|} |\cW_i|$;
        \While{$|T| \le q - i - 1$}
            \State Define $\cT_{i, T} := \cW_i(T)$;
            \State Remove sets containing $T$ from $\cW_i$: $$\cW_i \leftarrow \cW_i \setminus \cW_i[T];$$
            \State Add $T$ to $\cT_{i + 1}'$: $$\cT_{i + 1}' \leftarrow \cT_{i + 1}' \cup \{T\};$$
            \State Find a maximal set $T'$ such that $|\cW_i(T')| \ge \alpha^{-|T'|} |\cW_i|$;
            \State Set $T \leftarrow T'$;
        \EndWhile
        \State Set $\cT_{i + 1} \leftarrow \cT_{i + 1}' \cup \cT_i^{(\le q - i - 1)}$.
    \end{algorithmic}
    \caption{Simplification procedure}
    \label{algo: simplification procedure I}
\end{algorithm}

In what follows, we define families $\cT_{i + 1}, \cW_i, \cT_{i, T}, \cT_{i + 1}'$ as obtained after this procedure. The desired family $\cT$ is $\cT_{q - t}$.
Note that each set in $\cT_i$ has size at most $q - i$. 
Finally, for any $T \in \cT_{i + 1}'$, there exists an $\alpha$-spread family $\cT_{i, T}$ such that
\begin{align*}
    \cT_{i}^{(q - i)} \setminus \cW_i = \bigsqcup_{T \in \cT_{i + 1}'} \{T\} \vee \cT_{i, T}.
\end{align*}

In order to establish Lemma~\ref{lemma: sunflower simplification}, we need the following three claims.

\begin{claim}
\label{claim: simlification consistency}
    For each $i = 0, \ldots, q - t$, $\cT_i$ admits two properties:
    \begin{itemize}
        \item does not contain a sunflower with the core of size $\le t - 1$ and $s$ petals;
        \item each set of $\cT_i$ has size at least $t$.
    \end{itemize}
\end{claim}

\begin{proof}
The proof is by induction on $i$. For $i = 0$ the statement holds by the assumptions of Lemma~\ref{lemma: sunflower simplification}. Assume that  $\cT_{i}$ satisfies both properties. We prove that $\cT_{i + 1}$ does also.

\textbf{Assume that $\cT_{i + 1}$ contains a sunflower with the core of size $\le t - 1$ and $s$ petals.} Consider such a sunflower $T_1, \ldots, T_s \in \cT_{i + 1}$. Recall that $\cT_{i + 1} = \cT_{i}^{(\le q - i - 1)}\cup \cT_{i + 1}'$. Thus, we may assume that $T_1, \ldots, T_\ell \in \cT_{i +1}'$ for some $\ell$, $0\le \ell \le s$, and the remaining sets $T_{\ell + 1}, \ldots, T_s \in \cT_{i}^{(\le q - i - 1)}$.

Families $\cT_{i, T_{w}}, w = 1, \ldots, \ell,$ defined by the procedure, are $\alpha$-spread. Claim~\ref{claim: covering number} implies that their covering number is at least $\alpha =  s q$. We claim that there are disjoint $F_1 \in \cT_{i, T_1}, \ldots, F_\ell \in \cT_{i, \ell}$ that do not intersect $\bigcup_{p = 1}^s T_p$. Assume that the opposite holds and let $j$ be the maximum integer such that there are disjoint $F_1 \in \cT_{i, T_1}, \ldots, F_j \in \cT_{i, T_j}$ that do not intersect $\bigcup_{p = 1}^s T_p$. It means that $\bigcup_{p = 1}^j (F_j \cup T_j) \cup \bigcup_{p = j + 2}^s T_p$ is a transversal for $\cT_{i, T_{j + 1}}$, so the covering number of $\cT_{i, T_j}$ is at most $j (q - i) + (s - 1 - j) (q - i) < s (q - i)$, a contradiction. Therefore, $\cT_{i}$ contains sets $T_1 \cup F_1, \ldots, T_\ell \cup F_\ell, T_{\ell + 1}, \ldots, T_s$, which form a sunflower with the core of size $\le t - 1$, a contradiction.

\textbf{Next, assume that $\cT_{i + 1}$ contains a set $T$ of size at most $t - 1$.} By the induction hypothesis, we may assume that $T \in \cT_{i + 1}'$. Then, the family $\cT_{i, T}$ is $\alpha$-spread, and, thus, its covering number is at least $\alpha_i$. Arguing as before, we find disjoint sets $F_1, \ldots, F_s \in \cT_{i, T}$. Thus, $\cT_{i}$ contains a sunflower $T \cup F_1, \ldots, T \cup F_s$ with the core $T$, $|T| \le t - 1$, a contradiction.
\end{proof}

\begin{claim}
\label{claim: low layers bound in simplification argument}
    We have $|\cW_i| \le 2 \binom{q - i}{t} \phi(s, q - i) \alpha^{q - i - t}$.
\end{claim}

\begin{proof}
    We define a graph $G = (\cW_i, E)$ such that two sets in $W_1, W_2$ are connected by an edge if they intersect in at least $t$ elements. We prove below that the maximal degree in this graph is at most $\binom{q - i}{t} \alpha^{q - i - t}$. Consider a set $W \in \cW_i$. Then all neighbors of $W$ in the graph $G$ belong to a set
    \begin{align*}
        \bigcup_{T \in \binom{W}{t}} \cW_i[T].
    \end{align*}
    We claim that $|\cW_i(T)| \le \alpha^{q - i - t}$. Consider two cases.

    \textbf{Case 1. We have $|\cW_i(T)| \le \alpha^{-|T|} |\cW_i|$.} By the construction of $\cW_i$, the bound $\alpha^{-|X|}|\cW_i| \le |\cW_i(X)|$ holds for some $X$ of size $q - i$. But $|\mathcal W_i(X)| = 1$, therefore, $|\cW_i| \le \alpha^{q - i}$, and so $|\cW_i(T)| \le \alpha^{q - i - t}$.

    \textbf{Case 2. We have $|\cW_i(T)| > \alpha^{-|T|} |\cW_i|$.} Consider the maximal $X$ disjoint from $T$ such that $|\cW_i(X \sqcup T)| \ge \alpha^{-|X|} |\cW_i(T)|$. Since we have $|\cW_i(X \sqcup T)| > \alpha^{-|X| - |T|} |\cW_i|$ by the assumption of the case, the size of $X$ should be $q - i - t$ by the construction of $\cW_i$. Hence, we have $1 = |\cW_i(X \sqcup T)| \ge \alpha^{-(q - i - t)} |\cW_i(T)|$, and so $|\cW_i(T)| \le \alpha^{q - i - t}$.

    Therefore, the set $W$ has at most $\sum_{T\in{W\choose t}}|\mathcal W_i(T)|\le \binom{q - i}{t} \alpha^{q - i - t}$ neighbors. It implies that the maximal degree of a graph $G$ is at most $\binom{q - i}{t} \alpha^{q - i - t}$. Using a standard greedy argument, one can show that the size of the largest independent set in $G$ is at least
    \begin{align}
    \label{eq: independent set size}
        \frac{|\cW_i|}{1 +\binom{q - i}{t} \alpha^{q - i - t}} \ge \frac{|\cW_i|}{2 \binom{q - i}{t} \alpha^{q - i - t}}.
    \end{align}

    Therefore, there exists a family $\cI \subset \cW_i$ of size at least the RHS of~\eqref{eq: independent set size}, such that any two distinct sets $F_1, F_2 \in \cI$ intersect in at most $t - 1$ elements. 
    
    We claim that $\cI$ does not contain a sunflower with $s$ petals. Suppose that such a sunflower $F_1, \ldots, F_s \in \cI$ exists. Then, the core of this sunflower must have size at most $t - 1$. Since $\cW_i$ is a subset of $\cT_i$, the sets $F_1, \ldots, F_s$ belong to $\cT_i$, contradicting Claim~\ref{claim: simlification consistency}. Therefore, $\eqref{eq: independent set size} \le |\cI| \le \phi(s, q - i)$, and so
    \begin{align*}
        |\cW_i| \le \phi(s, q - i) \cdot 2 \binom{q - i}{t} \alpha^{q - i - t}.
    \end{align*}
    The claim follows.
\end{proof}

Next, we adapt the proof of Erd\H{o}s--Rado Theorem~\cite{erdos1960intersection} to bound $\phi(s, q - i)$ via $\phi(s, t)$.

\begin{claim}
\label{claim: phi(s,q) bound}
    For any $a \le b$, we have $\phi(s, b) \le \phi(s, a) \cdot (sb)^{b - a}$.
\end{claim}
\begin{proof}
    Let $\cG$ be a $b$-uniform family without a sunflower with $s$ petals and such that $|\cG| = \phi(s, b)$. Consider the maximal matching $F_1, \ldots, F_m$ in $\cG$. Since $\cG$ does not contain a sunflower with $s$ petals, we have $m < s$. Since the matching is maximal, each set $G \in \cG$ intersects $M = \bigcup_{j = 1}^m F_m$, where $|M| \le (s - 1)b$. Therefore, the size of $|\cG|$ is at most $\sum_{x \in M} |\cG(x)|$. Each family $\cG(x)$ in the latter sum does not contain a sunflower with $s$ petals, so $|\cG(x)| \le \phi(s, b - 1)$. It yields the bound
    \begin{align*}
        \phi(s, b) = |\cG| \le (s - 1) b \cdot \phi(s, b - 1).
    \end{align*}
    Iterating the inequality above, we obtain 
    \begin{align*}
        \phi(s, b) \le (s - 1)^{b - a} \cdot \phi(s, a) \cdot \prod_{j = a + 1}^b j \le (sb)^{b - a} \cdot \phi(s, a). & \qedhere
    \end{align*}
\end{proof}

Finally, we are ready to prove Lemma~\ref{lemma: sunflower simplification}.

\begin{proof}[Proof of Lemma~\ref{lemma: sunflower simplification}]    
    Decomposing $\cA[\cS \setminus \cS[\cT]]$ into the union of $\cA[\cW_i], i = 0, \ldots, q - t - 1$, we obtain
    \allowdisplaybreaks
    \begin{align*}
        |\cA[\cS \setminus \cS[\cT]]| & \le \sum_{i = 0}^{q - t - 1} |\cW_i| \max_{W \in \cW_i} |\cA(W)|  \\
        & \le \sum_{i = 0}^{q - t- 1} |\cW_i| A_{q - i} \\
        & \overset{\text{Claim~\ref{claim: low layers bound in simplification argument}}}{\le} { 2 \sum_{i = 0}^{q - t - 1} \phi(s, q - i) \binom{q - i}{q - i - t}\left (\frac{\alpha}{r} \right )^{q - i - t} A_t} \\
        & \overset{\text{Claim~\ref{claim: phi(s,q) bound}}}{\le} 2 \sum_{i = 0}^{q - t - 1} \phi(s, t) (s \cdot (q - i))^{q - i - t} \cdot (q - i)^{q - i - t} \cdot \left (\frac{\alpha}{r} \right)^{q - i - t} A_t.
    \end{align*}
    Since $\varepsilon r \ge \alpha sq^2$, the last sum is at most
    \begin{align*}
        2 \phi(s, t) \sum_{i = 0}^{q - t - 1} \varepsilon^{q - i - t} A_t \le \frac{2 \varepsilon}{1 - \varepsilon} \cdot \phi(s, t) A_t. & \qedhere
    \end{align*}
\end{proof}

\section{Preliminaries II}
\label{section: preliminaris-ii}

\subsection{Delta-systems}
\label{section: delta-system method}

The original method of Frankl and F\"uredi heavily relies on the following lemma (see Lemma 7.4 from~\cite{Frankl1987}).

\begin{lemma}
\label{lemma: delta-system method}
    Let $\ff \subset \binom{n}{k}$. Fix numbers $p, t$, such that $p \ge k \ge 2 t + 1$. Suppose that $\ff$ does not contain a sunflower with $p$ petals and the core of size $t - 1$. Then, there exists a subfamily $\ff^* \subset \ff$ that admits the following properties:
    \begin{enumerate}
        \item for any $F \in \ff^*$, there exists a subset $T$ of size $t$ such that each set $E$, $T \subset E \subsetneq F$, is a core of a sunflower with $p$ petals in $\ff^*$;
        \item $|\ff \setminus \ff^*| \le c(p, k) \binom{n}{k - t - 1}$,
    \end{enumerate}
    where $c(p, k)$ is some function that can be bounded by $p^{2^k} \cdot 2^{2^{Ck}}$ for some absolute constant $C$. 
\end{lemma}

While the paper~\cite{Frankl1987} omits precise numerical dependencies between $n, k$ and $t$, we need them stated explicitly. For the sake of completeness, we present here a proof of the sunflower theorem by Frankl and F\"uredi. Moreover, the core idea will be used in the proof of our result.

\begin{theorem}
\label{theorem: delta-system solution}
    Let $k \ge 2t + 1$. Suppose that a family $\ff$ does not contain a sunflower with $s$ petals and the core of size $t - 1$. Then
    \begin{align*}
        |\ff| \le \phi(s, t) \binom{n}{k - t} + \frac{k \cdot c(sk, k)}{n - k} \binom{n}{k - t}.
    \end{align*}
\end{theorem}

\begin{proof}
Applying Lemma~\ref{lemma: delta-system method} with $p = sk$, we infer that there exists $\ff^* \subset \ff$, $|\ff \setminus \ff^*| \le c(sk, k) \binom{n}{k - t - 1}$, such that for any $F \in \ff^*$ there is a subset $T \subset F$ of size $t$ such that any $E$, $T \subset E \subsetneq F$, is a core of a sunflower with $sk$ petals in $\ff^*$. For each $F \in \ff^*$, fix one such $T = T(F)$. Then, we have
\begin{align*}
    |\ff^*| = \sum_{D \in \binom{[n]}{k - t}} |\{F \in \ff^* \mid F \setminus D = T(F) \}|.
\end{align*}
Define $\ff_D = \{F \in \ff^* \mid F \setminus D = T(F) \}$. We claim that $\ff_D$ does not contain a sunflower with $s$ petals and arbitrary core. Otherwise choose such sunflower $T_1, \ldots, T_s$. Choose $X \subset D$, $|X| = t - 1 - |\cap_i T_i|$. Each $T_i \cup X$ is the core of a sunflower in $\ff^*$ with $sk$ petals. Let $m$ be the maximum number such that there are disjoint $F_1 \in \ff^*(T_1 \cup X), \ldots, F_m \in \ff^*(T_m \cup X)$ that do not intersect $\cup_{i = 1}^m T_i$. Since $\ff^*$ does not contain a sunflower with $s$ petals and the core of size $t - 1$, $m$ is strictly less than $s$. Consider a family $\ff^*(T_{m + 1} \cup X)$. By the definition of $\ff^*$, it contains pairwise disjoint sets $G_1, \ldots, G_{sk}$. At least one such set do not intersect $\cup_{i = 1}^m (T_i \cup X \cup F_i)$ since $|\cup_{i = 1}^m (T_i \cup X \cup F_i)| \le (s - 1) k < sk$. Hence, we can choose $F_{m + 1} \in \ff^*(T_i \cup X)$, that do not intersect $\cup_{i = 1}^{m + 1} T_i \cup X$, contradicting with the maximality of $m$. 

Therefore, $\ff_D \subset \binom{n}{t}$ does not contain a sunflower with $s$ petals, and so $|\ff_D| \le \phi(s, t)$. Since 
\begin{align*}
    |\ff^*| \le \sum_{D \in \binom{[n]}{k - t}} |\ff_D|,
\end{align*}
we have
\begin{align*}
    |\ff| \le \phi(s,t) \binom{n}{k - t} + c(sk, k) \binom{n}{k - t - 1} \le \phi(s,t) \binom{n}{k - t} + \frac{k \cdot c(sk, k)}{n - k} \binom{n}{k - t}. & \qedhere
\end{align*}
\end{proof}

We shall use Theorem~\ref{theorem: delta-system solution} for $k < 5t$. For larger $k$, we use different techniques, that are built on the results presented in the following sections.

\subsection{Sunflower Tur\'an number: the result of Brada\v{c} et al.}

Theorem~\ref{theorem: delta-system solution} is quite restrictive, as it requires the lower bound on $n$ of the form $(st)^{2^{\Omega(t)}}$. To relax this restriction, we employ the following result due to Brada\v{c} et al.~\cite{bradavc2023turan}.

\begin{theorem}
\label{theorem: bucic upper bound}
    Let $\ff \subset \binom{[n]}{k}$ be a family that does not contain a sunflower with $s$ petals and the core of size exactly $t - 1$. Then, for each $k$ there is some constant $C_k$ depending on $k$ only, such that
    \begin{align*}
        |\ff| \le \begin{cases}
            C_k n^{k - t} s^t & \text{ if } k \ge 2t - 1 \\
            C_k n^{t - 1} s^{k - t + 1} & \text{ if } t \le k < 2t - 1.
        \end{cases}
    \end{align*}
\end{theorem}

{The bounds provided by Theorem~\ref{theorem: bucic upper bound} are tight up to the dependence on $k$. The example can be contructed using a nice probabilistic argument, see Lemmas 10 and 11 of~\cite{bradavc2023turan}.}

We will use this theorem for $k \le 2t + 1$. In that regime, $C_k \sim t! \cdot t^{\binom{2t - 1}{t}}$, so $C_k$ is double-exponential in $k$. In what follows, we assume that $C_k$ is increasing in $k$ and larger than $1$.

\subsection{Homogeneity and general domains}
\label{subsection: homogeneity and general domains}

While the main theorem of our paper deals with subfamilies of $\binom{[n]}{k}$, some parts of the analysis can be extended to other domains. As in Theorem~\ref{theorem: general theorem for down-closed sunflowers}, we consider families $\ff$ that are subfamilies of some $k$-uniform family $\cA$ which satisfies a number of assumptions.

When dealing with the uniform families $\ff \subset \cA$, it is convenient to replace the notion of $R$-spreadness with the notion of $\tau$-homogeneity, introduced in \cite{kupavskii_spread_2022}.

\begin{definition}
    A family $\ff$ is {\em $\tau$-homogeneous} if for each $X$, we have
    \begin{align*}
        \frac{|\ff(X)|}{|\cA(X)|} \le \tau^{|X|} \frac{|\ff|}{|\cA|}.
    \end{align*}
\end{definition}

In what follows, we use the notation $\mu(\ff) = |\ff| / |\cA|$. Using this, we can rewrite the condition above as follows: $\mu(\ff(X))\le \tau^{|X|}\mu(\ff)$. A variant of the definition of $\tau$-homogeneity, called $\tau$-globalness, was used in~\cite{keller_sharp_2023}, see Definition~\ref{deftaugl}. If the domain $\cA$ is $r$-spread, a $\tau$-homogeneous family $\ff$ is $r/\tau$-spread. Indeed, we have
\begin{align*}
    |\ff(X)| \le \tau^{|X|} |\ff| \frac{|\cA(X)|}{|\cA|} \le (r/\tau)^{-|X|} |\ff|.
\end{align*}
In general, we require $\cA(S)$ to be $r$-spread for any $S$ that has size smaller than some integer $q$, $q \ge t$. That implies the first assumption on the domain $\cA$, which strengthens the corresponing assumption of Theorem~\ref{theorem: general theorem for down-closed sunflowers}.

\begin{assumption}
\label{assumption: spreadness - small r}
    Let $q$ be an integer. The family $\cA$ is $k$-uniform and $(r, q + t - 1)$-spread. 
\end{assumption}
The necessary $q$ will be chosen later. 

Solution of hypergraph Turan-type problems in $\binom{[n]}{k}$ often involve the binomial coefficient $\binom{n - t}{k - t}$, which is the number of sets in $\binom{[n]}{k}$ containing some set of size $t$. We recall natural equivalent for this quantity in arbitrary domains:
\begin{align*}
    A_t = \max_{T \in \partial_t \cA} |\cA(T)|.
\end{align*}
The second assumption on $\cA$ requires $A_t$ to be not too small.

\begin{assumption}
\label{assumption: large t-link}
Let $\eta \ge 1$ be a real number. We have $A_t \ge r^{-\eta t} |\cA|$.
\end{assumption}

The two next assumptions concern the shadow of $\cA$.

\begin{assumption}
\label{assumption: regularity}
    For any set $S \in \partial_{\le q} \cA$, a family $\ff \subset \cA(S)$, and $h \le t - 1$, we have
    \begin{align*}
        \frac{|\ff|}{|\cA(S)|} = \frac{1}{|{\partial_h} (\cA(S))|} \sum_{H \in \partial_{h } (\cA(S))} \frac{|\ff(H)|}{|\cA(H \cup S)|},
    \end{align*}
    or, equivalently, for a random set $\bH$ uniformly distributed on $\partial_h (\cA(S)) $, we have
    $
        \mu(\ff) = \EE \mu(\ff(\bH)).$
\end{assumption}

\begin{assumption}
\label{assumption: weak shadow consistency - small r}
    For any set $R \in \partial_{\le q} \cA$, an integer $h$, $h \le t$, and some $\mu > 1$, it holds
    \begin{align*}
        \frac{|\partial_h (\cA(R))|}{|\partial_h \cA|} \ge \left ( 1 - \frac{q}{\mu k}\right )^h.
    \end{align*}
\end{assumption}

The natural question is when these assumptions are satisfied. The following propositions show that some interesting domains fit these assumptions.

\begin{proposition}
\label{proposition: assumptions satisified}
    Suppose and $n, k \ge  q + t$, $q \ge t$ and $n > 8k$. Then, the following holds:
    \begin{enumerate}
        \item Suppose additionally that $n \ge 2k t$. Then $\binom{n}{k}$ satisfies Assumptions~\ref{assumption: spreadness - small r}-\ref{assumption: weak shadow consistency - small r} with $r = n/k, \mu = n/2k, \eta = 2$;
        \item Suppose that  $n \ge 2t, k 
        \ge 2 q$. Then $[n]^k$ satisfies Assumptions~\ref{assumption: spreadness - small r}-\ref{assumption: weak shadow consistency - small r} with $r = n, \mu = 1/2, \eta = 2$;
        \item Fix an integer $w$, such that $n\ge 2kt/w$ and $w|k$. Then $\binom{[n]}{k / w}^w$ satisfies Assumptions~\ref{assumption: spreadness - small r}-\ref{assumption: regularity} with $r =  \frac{wn}{k}$ and $\eta = 2$. Moreover, if  $k \ge w (q + t)$, then $\binom{[n]}{k/w}^w$ satisfies Assumption~\ref{assumption: weak shadow consistency - small r} with $\mu = n/2k$.
    \end{enumerate}
\end{proposition}

\begin{proof}
    By taking $w = 1$ and $w = k$ respectively, satisfiability of Assumptions~\ref{assumption: spreadness - small r}-\ref{assumption: regularity} for $[n]^k, \binom{[n]}{k}$ follows from the third statement, so we consider only the family $\binom{[n]}{k/w}^w$ for these assumptions.  Assumption~\ref{assumption: spreadness - small r} is satisfied due to Proposition~\ref{proposition: spread families examples}, which is used with $q + t - 1$ in place of $t$. 
    
    Next, we check Assumption~\ref{assumption: large t-link}. Consider a set $T$ of size $t$, and set $t_i = |T \cap I_i|$. Then, we have
    \begin{align}
        \frac{|\cA(T)|}{|\cA|} \ge \prod_{i = 1}^w \frac{\binom{n - t_i}{k/w - t_i}}{\binom{n}{k/w}} = \prod_{i = 1}^w \frac{\binom{k/w}{t_i}}{\binom{n}{t_i}} \ge \prod_{i = 1}^w \left ( \frac{k/w + 1 - t_i}{n} \right )^{t_i} \label{eq: product for A_t lower bound}
    \end{align}
    If $t \le k/2w$, then the latter product is at least
    \begin{align*}
        \prod_{i = 1}^w \left ( \frac{k}{2wn} \right )^{t_i} = \left ( \frac{k}{2wn} \right )^t \ge r^{-2t},
    \end{align*}
    provided $n \ge 2 k/w$. In particular, this bound implies that Assumption~\ref{assumption: large t-link} holds for $\binom{[n]}{k}$, since $k \ge 4q \ge 2t$. If $t \ge k/2w$, then we bound the RHS of~\eqref{eq: product for A_t lower bound} as follows:
    \begin{align*}
        \prod_{i = 1}^w \left ( \frac{k/w + 1 - t_i}{n} \right)^{t_i} \ge n^{-t} = (2t)^{-t} \left (2t \right )^t n^{-t} \ge (2t)^{-t} \left ( \frac{k}{wn} \right )^t \ge \left ( \frac{k}{wn} \right )^{2t}= r^{-2t},
    \end{align*}
    where we used $w n \ge 2t k$ in the last inequality. Hence, Assumption~\ref{assumption: large t-link} holds with $\eta = 2$.

    Similarly to Proposition~\ref{proposition: spread families examples}, we assume that $\cA$ has the form
    \begin{align*}
        \cA = \binom{I_1}{k / w} \times \ldots \binom{I_w}{k / w},
    \end{align*}
    where disjoint sets $I_1, \ldots, I_w$ have size $n$. The following claim ensures that Assumption~\ref{assumption: regularity} holds for any domain of the form
    \begin{align*}
        \cA = \binom{I_i}{k_i} \times \ldots \binom{I_w}{k_w}.
    \end{align*}
     We need this claim stated separately, because we use it in what follows.
    \begin{claim}
    \label{claim: regularity for general families}
    Let $I_1, \ldots, I_w$ be arbitrary disjoint sets and $k_1 \le |I_1| , \ldots, k_w \le |I_w|$ be integer numbers. Then the family
    \begin{align*}
        \cA = \binom{I_i}{k_i} \times \ldots \binom{I_w}{k_w}
    \end{align*}
    satisfies Assumption~\ref{assumption: regularity}.
    \end{claim}
    \begin{proof}
    Fix $S$ of size at most $q$, and consider a set $\bH$ uniformly distributed on $\partial_h (\cA(S))$. Let $\bh_1 = |\bH \cap I_1 \setminus S|, \ldots, \bh_w = |\bH \cap I_w \setminus S|$ and $\bh = (\bh_1, \ldots, \bh_w)$. Let $\partial_{\bh}(\cA(S))$ be a family of sets $H$ from $\partial_h (\cA(S))$ such that $|H \cap I_i| = \bh_i$.  Then, for a family $\ff \subset \cA(S)$, we have
    \begin{align*}
        \EE [\mu(\ff(\bH)) \,|\, \bh] = \frac{1}{|\partial_{\bh} (\cA(S))|} \sum_{H \in \partial_{\bh} (\cA(S))} \frac{|\ff(H)|}{|\cA(H \cup S)|}.
    \end{align*} 
    Each set $F \in \ff$ is counted $\prod_{i = 1}^w \binom{k_i - |S \cap I_i|}{\bh_i}$ times, so we have
    \begin{align*}
        \EE [\mu(\ff(\bH)) \,|\, \bh] = \frac{|\cA(S)| \prod_{i = 1}^w \binom{k_i - |S \cap I_i|}{\bh_i}}{|\partial_{\bh} \cA(S)| \prod_{i = 1}^w \binom{n - \bh_i - |I_i \cap S|}{k_i - \bh_i - |S \cap I_i|}} \cdot \mu(\ff),
    \end{align*}
    where $\mu(\ff) = |\ff| / |\cA(S)|$.
    The multiplicative factor of $\mu(\ff)$ in the right-hand side equals
    \begin{align*}
        \frac{\prod_{i = 1}^w \binom{n - |I_i \cap S_i|}{k_i - |S \cap I_i|}  \binom{k_i - |S \cap I_i|}{\bh_i}}{\prod_{i = 1}^w \binom{n - |S \cap I_i|}{\bh_i} \binom{n - \bh_i - |S \cap I_i|}{k_i - \bh_i - |S \cap I_i|}} = 1,
    \end{align*}
    so $\EE [\mu(\ff(\bH)) \,|\, \bh] = \mu(\ff)$, and, by taking the expectation of both sides, Assumption~\ref{assumption: regularity} follows. 
    \end{proof}
    
    Then, we check Assumption~\ref{assumption: weak shadow consistency - small r} for $\binom{[n]}{k/w}^w$. By taking $w = 1$, it implies the first statement of the proposition. Consider arbitrary $R , |R| \le q$. If $k / w \ge  q + t$, then $k \ge w (|R| + t)$. For a partition of $h$ into non-negative ordered summands $h_1, \ldots, h_w$, define $\overline{h} = (h_1, \ldots, h_w)$, and note that $h_i \le t \le k / w - |R|$.  Thus, we have
    \begin{align*}
        \frac{|\partial_{h} (\cA(R))|}{|\partial_h \cA|}  & = \frac{
            \sum_{\overline{h}}  \prod_{i = 1}^w \binom{n - |R \cap I_i|}{h_i}
        }{
            \sum_{\overline{h}} \prod_{i = 1}^w \binom{n}{h_i}
        } \ge 
        \min_{\overline{h}}
        \prod_{i = 1}^w 
            \frac{
                \binom{n - |R \cap I_i|}{h_i}
            }{\binom{n}{h_i}} \\
        & \ge \min_{(h_1, \ldots, h_w), \sum_i h_i = h} \prod_{i = 1}^w \left ( \frac{n - |R| - h}{n} \right )^{h_i} \ge \left ( 1 - \frac{|R| + h}{n}\right )^h \\
        & \ge \left (1 - \frac{q + t}{n}\right )^h \ge \left ( 1 - \frac{2q}{n}\right )^h,
    \end{align*}
    so Assumption~\ref{assumption: weak shadow consistency - small r} is satisfied with $\mu = n/2k$.


    Thus, the first and the third statements of the proposition are proved. It remains to check that $[n]^k$ satisfies Assumption~\ref{assumption: weak shadow consistency - small r} with $\mu = 1/2$. For any $R$, we have
    \begin{align*}
        \left | \partial_h ([n]^k(R)) \right | = \binom{k - |R|}{h} n^h,
    \end{align*}
    so
    \begin{align*}
        \frac{\left | \partial_h ([n]^k(R))\right |}{|\partial_h [n]^k|} \ge \frac{\binom{k - |R|}{h}}{\binom{k}{h}} \ge \left ( \frac{k - |R| - h}{k} \right )^h \ge \left ( 1 - \frac{q + h}{k}\right )^h \ge \left (1 - \frac{2q}{k} \right )^h,
    \end{align*}
    provided $k \ge 2 q$. It implies the second statement of the proposition.
\end{proof}

It turns out that under mild restrictions on $n, t, q$ the family of permutations $\Sigma_n$ as the family of sets $\mathfrak{S}_n$ (see Section~\ref{subsection: spread domains}) also satisfies Assumptions~\ref{assumption: spreadness - small r}-\ref{assumption: weak shadow consistency - small r}. Note that $n$ plays the role of uniformity of the family $\Sigma_n$.

\begin{proposition}
\label{proposition: permutations satisfy assumptions}
Suppose that $n > 4( q + t)$, $n \ge 16$ and $q \ge t$. Then, $\mathfrak{S}_n$ satisfies Assumptions~\ref{assumption: spreadness - small r}-\ref{assumption: weak shadow consistency - small r} with $k = n, r = n/4$, $\eta = 2$ and $\mu = 1/4$.
\end{proposition}

\begin{proof}
Since $n > 4(q + t)$, Proposition~\ref{proposition: r-t spreadness of permutations} implies that $\mathfrak{S}_n$ is $(r, q + t - 1)$-spread for $r = n/4$, so Assumption~\ref{assumption: spreadness - small r} is satisfied. 

Next, we check Assumption~\ref{assumption: large t-link}. We have
\begin{align*}
    (n - t)! \ge n^{-t} n! =  4^{-t} (n/4)^{-t} n! \ge (n/4)^{-2t} n!,
\end{align*}
since $n \ge 16$. Therefore, Assumption~\ref{assumption: large t-link} holds with $\eta = 2$.

Then, we check Assumption~\ref{assumption: regularity}. We may consider a set $S \in \partial_{\le q} \mathfrak{S}_n$ as an injection that maps some set $\operatorname{Dom} S$ of size $|S|$ to a subset $\operatorname{Im} S$ of $[n]$. Thus, $\mathfrak{S}_n(S)$ can be considered as a family of such permutations that coincide with $S$ on the domain $\operatorname{Dom} S$ of $S$. The shadow $\partial_h (\mathfrak{S}_n)$ can be considered as a number of bijections from a subset of $[n] \setminus \operatorname{Dom} S$ of size $h$ to a subset of $[n] \setminus \operatorname{Im} S$. It yields
\begin{align*}
    |\mathfrak{S}_n(S)| & = (n - |S|)! \\
    |\partial_h (\mathfrak{S}_n(S))| & = \binom{n - |S|}{h}^2 h! \\
    |\mathfrak{S}_n(S \cup H)| & = (n - |S| - h)!
\end{align*}
for any $H \in \partial_h (\mathfrak{S}_n(S))$. Finally, we have
\begin{align*}
    \sum_{H \in \partial_h (\mathfrak{S}_n(S))} |\ff(H)| = \binom{n - |S|}{h} |\ff|.
\end{align*}
It implies
\begin{align*}
    \frac{1}{|\partial_h(\mathfrak{S}_n(S))|}\sum_{H \in \partial_h (\mathfrak{S}_n(S))} \frac{|\ff(H)|}{|\mathfrak{S}_n(S \cup H)|} & = \frac{1}{\binom{n - |S|}{h}^2 h!} \cdot \frac{\binom{n - |S|}{h}}{(n - |S| - h)!} \cdot |\ff| = \frac{|\ff|}{(n - |S|)!} \\
    & = \frac{|\ff|}{|\mathfrak{S}_n(S)|},
\end{align*}
and Assumption~\ref{assumption: regularity} follows.

Finally, we check Assumption~\ref{assumption: weak shadow consistency - small r}. Fix a set $R \in \partial_{\le q} \mathfrak{S}_n$, and bound
\begin{align*}
    \frac{|\partial_h (\mathfrak{S}_n(R))|}{|\partial_h \mathfrak{S}_n|} = \frac{\binom{n - |R|}{h}^2 h!}{\binom{n}{h}^2 h!} \ge \left ( 1 - \frac{|R| + h}{n}\right )^{2h} \ge \left (1 - \frac{4 q}{n} \right )^h,
\end{align*}
so Assumption~\ref{assumption: weak shadow consistency - small r} holds with $\mu = 1/4$.
\end{proof}

We complete this section with three statements about $\tau$-homogeneous families. The first one shows that most restrictions of a $\tau$-homogeneous family $\ff$ are also homogeneous.

\begin{proposition}
\label{proposition: most of subsets are homogeneous}
Suppose that Assumption~\ref{assumption: regularity} holds. Let $\ff \subset \cA(S)$, $S \in \partial_{\le q}\cA$, be a $\tau$-homogeneous family. Let $h$ be a positive integer less than $k$ and $\alpha \in (0, 1)$. Suppose that $\tau \le (1 - \alpha \rho)^{-1/h}$. Then, for all but $\alpha \left | \partial_h \cA \right |$ sets $H \in \partial_h \cA $, the family $\ff(H)$ is $\hat{\tau}$-homogeneous for $\hat \tau =  \tau / (1 - \rho)$ and $\mu(\ff(H)) \ge (1 - \rho) \cdot \tau^h \cdot \mu(\ff)$.
\end{proposition}

\begin{proof}
    If $\ff(H)$ is not $\hat \tau$-homogeneous, then there exists $Z$ such that
    \begin{align*}
        \hat \tau^{-|Z|} \mu(\ff(H \cup Z)) \ge \mu(\ff(H)).
    \end{align*}
    Since $\ff$ is $\tau$-homogeneous, we have
    \begin{align*}
        \tau^{|H| + |Z|} \hat \tau^{-|Z|} \mu(\ff) \ge \hat \tau^{-|Z|} \mu(\ff(H \cup Z))  \ge \mu(\ff(H)),
    \end{align*}
    so $\mu(\ff(H)) \le \left (\frac{\tau}{\hat \tau} \right )^{|Z|} \cdot \tau^{|H|} \mu(\ff) \le \tau / \hat{\tau} \cdot \tau^{|H|} \mu(\ff)$. Define 
    \begin{align*}
        \hh = \{H \in \partial_h (\cA(S)) \mid \ff(H) \text{ is not } \hat \tau\text{-homogeneous or } \mu(\ff(H)) \le \frac{\tau}{\hat \tau} \cdot \tau^h \mu(\ff)\}
    \end{align*}

    Due to Assumption~\ref{assumption: regularity}, we have
    \begin{align*}
        \frac{1}{|\partial_h (\cA(S)) |} \sum_{H \in \partial_h (\cA(S)) } \mu(\ff(H)) = \mu(\ff),
    \end{align*}
     For each $H \in \hh$, we have $\mu(\ff(H)) \le \frac{\tau}{\hat \tau} \cdot \tau^h \mu(\ff)$ by the reasoning above. If $H \not \in \hh$, we bound $\mu(\ff(H)) \le \tau^{|H|} \mu(\ff)$ by the $\tau$-homogeneity of $\ff$,
    and obtain
    \begin{align*}
         \frac{1}{|\partial_h (\cA(S)) |} \sum_{H \in \partial_h \cA} \left (  \frac{\tau}{\hat \tau} \cdot \tau^{|H|} \mu(\ff)\indicator\{H  \in \hh\} + \tau^{|H|} \mu(\ff) \indicator\{H \not \in \hh\} \right ) & \ge \mu(\ff), \\
         \mu(\hh) \cdot \frac{\tau}{\hat \tau} \cdot \tau^{h} + \tau^h (1 - \mu(\hh)) & \ge 1,
    \end{align*}
    where $\mu(\hh) = |\hh| / |\partial_h (\cA(S))|$. Rearranging the terms, we get
    \begin{align*}
        \mu(\hh) \le \frac{1 - \tau^{-h}}{1 - \tau / \hat \tau} \le \frac{1 - (1 - \alpha \rho)}{1 - (1 - \rho)} = \alpha. & \qedhere
    \end{align*}
\end{proof}

Next, we show that $\tau$-homogeneous families have large shadows.

\begin{proposition}
\label{proposition: shadow of tau-homogeneous family}
    Suppose that Assumption~\ref{assumption: regularity} holds. Let $\ff \subset \cA(S)$, $S \in \partial_{\le q} \cA,$ be a non-empty $\tau$-homogeneous family. Then, we have $\mu(\partial_h \ff) \ge \tau^{-h}$ for any $h \le k$.
\end{proposition}
\begin{proof}
    Due to Assumption~\ref{assumption: regularity}, we have
    \begin{align*}
        \mu(\ff) = \frac{1}{|\partial_h (\cA(S))|} \sum_{P \in \partial_h \ff} \mu(\ff(P)) \le \frac{|\partial_h \ff|}{|\partial_h (\cA(S))|} \cdot \tau^h \mu(\ff).
    \end{align*}
    Rearranging the terms, we obtain the claim.
\end{proof}

Finally, we show that homogeneity can be preserved after removing elements from the support of a $\tau$-homogeneous families.

\begin{claim}
    \label{claim: removing intersections} 
    If $\cG \subset \cA$ is $\tau$-homogeneous for some $r$-spread $\cA$, and $|X| < r / \tau$, then $\cG(\varnothing, X)$ is non-empty and $\tau'$-homogeneous for 
    \begin{align*}
        \tau' = \frac{\tau}{1 - \frac{|X| \tau}{r}}.
    \end{align*}
\end{claim}

\begin{proof}
    If $\cG$ is $\tau$-homogeneous, then it is $r/ \tau$-spread, so
    \begin{align}
    \label{eq: excluding X, intersection removing lemma, tau-hom.}
        |\cG(\varnothing, X)| \ge \left ( 1 - \frac{|X| \tau}{r} \right) |\cG|
    \end{align}
    due to Claim~\ref{claim: spread removing intersections}.
    It yields
    \begin{align*}
        \frac{|\cG(\varnothing, X)(Y)|}{|\cA(Y)|} & = \frac{|\cG(Y, X \cup Y)|}{|\cA(Y)|} \le \frac{|\cG(Y)|}{|\cA(Y)|} \\
        & \overset{{\color{teal} (\cG \text{ is $\tau$-homogeneous})}}{\le} \tau^{|Y|} \frac{|\cG|}{|\cA|} \overset{\eqref{eq: excluding X, intersection removing lemma, tau-hom.}}{\le} \tau^{|Y|} \frac{|\cG(\varnothing, X)|}{(1 - |X| \tau / r) |\cA|} \\ & \le \left (\tau / (1 - |X| \tau / r) \right )^{|Y|} \frac{|\cG(\varnothing, X)|}{|\cA|}. \qedhere
    \end{align*}
    
\end{proof}

\subsection{Spread approximation}
\label{subection: spread approximation Kupavskii}

There is a natural counterpart of Observation~\ref{observation: R-spread restriction} for the $\tau$-homogeneity property. 
\begin{observation}[Observation~1 from~\cite{kupavskii_spread_2022}]
\label{observation: homogeneous restriction}
Let $\ff \subset \cA$ be arbitrary subfamily of a $k$-uniform family $\cA$. Fix some $\tau > 1$. Let $\cS$ be the maximal set such that $\mu(\ff(S)) \ge \tau^{-|S|} \mu(\ff)$. Then, the family $\ff(S)$ is $\tau$-homogeneous as a subfamily of $\cA(S)$.
\end{observation}

Using it, one may obtain the following lemma which allows applying the $\tau$-homogeneity machinery in an arbitrary Turan-type problem.

\begin{lemma}[\cite{kupavskii_spread_2022}]
    \label{lemma: spread approximation}
    Fix an integer $q$ and a real number $\tau > 1$. Let $\cA$ be a $k$-uniform $r$-spread family. Let $\ff \subset \cA$. Then exists a family $\cS \subset \partial_{\le q} \cA$ and a partition of $\ff$ into families $\cR$ and $\ff_S \subset \cA(S), S \in \cS$:
    \begin{align*}
        \ff = \cR \sqcup \bigsqcup_{S \in \cS} \ff_S \vee  \{S\},
    \end{align*}
    such that each $\ff_S$ is $\tau$-homogeneous as a subfamily of $\cA(S)$ and $|\cR| \le \tau^{- (q + 1)} |\cA|$.
\end{lemma}

Informally, we refer to the family $\cS$ as the spread approximation of $\ff$. Depending on the problem, $\cS$ possess numerous useful properties. For example, if $\ff$ does not contain two sets intersecting in exactly $t - 1$ elements, then each two elements of $\cS$ intersects in at least $t$ elements under some assumptions on $r, k$ and $t$~\cite{kupavskii_spread_2022}. The accurate analysis of $\cS$ depending on $n, k$ and $t$ is the core of our approach.

However, Lemma~\ref{lemma: spread approximation} is not the only way to obtain a spread approximation of $\ff$. Later, we present other ways, so we require a precise definition of such families $\cS$.

\begin{definition}
    A family $\cS$ is an $R$-spread ($\tau$-homogeneous) {\Sname} of $\ff$, if $\ff$ admits a decomposition
    \begin{align*}
        \ff = \bigsqcup_{S \in \cS} \ff_S \vee \{S\},
    \end{align*}
    where families $\ff_S \subset \ff(S)$ are $R$-spread ($\tau$-homogeneous).
\end{definition}

\subsection{Boolean analysis}
\label{subsection: boolean analysis}

When $k$ is extremely large (namely, $k \ge 2^{(st)^C}$ for some large constant $C$), techniques based on $\tau$-homogeneous approximation give suboptimal results. The reason is  that the spread lemma works for the families without any assumption on their measure. In general, the factor $\log k$ in applications of Theorem~\ref{theorem: r-spread theorem} is unavoidable, as was shown in \cite{alweiss2021improved}, but in our special case it can be removed.

The idea is to use recent results in Boolean Analysis from papers~\cite{keevash2021global} and~\cite{keller_sharp_2023}. Before we start, we give a number of definitions.

We consider the Boolean cube $\{0, 1\}^n$ with the $p$-biased measure $\mu_p$, $\mu_p(x) = p^{\sum_{i} x_i} (1 - p)^{n - \sum_i x_i}, x \in \{0, 1\}^n$. Next, for the linear space of functions $f : \{0, 1\}^n \to \RR$, we define the scalar product 
\begin{align*}
    \langle f, g\rangle = \EE_{x \sim \mu_p} f(x) g(x).
\end{align*}
We denote the obtained Hilbert space  by $L_2(\{0, 1\}^n, \mu_p)$. We define the $q$-norm of the function $f$ as
\begin{align*}
    \Vert f \Vert_q = \EE^{1/q}_{x \sim \mu_p} |f(x)|^q.
\end{align*}
We denote expectation $\EE_{x \sim \mu_p} f(x)$ of $f$ by $\mu_p(f)$. There is a natural correspondence between the Boolean functions $f: \{0, 1\}^n \to \{0, 1\}$ and families of subsets $\ff$ of the Boolean cube, defined as follows: $f(x) = 1$ if and only if for the corresponding family $\ff$ holds that $\support(x) \in \ff$. Then, the measure of $\mu_p(\ff)$ is just the expectation $\mu_p(f)$ of the corresponding Boolean function and can be obtained as follows:
\begin{align*}
    \mu_p(\ff) = \sum_{F \in \ff} p^{|F|} (1 - p)^{n - |F|}.
\end{align*}
One can interpret $\mu_p(\ff)$ as the probability that $\ff$ contains a $p$-random subset $\bX$ of $[n]$.

When dealing with the uniform families $\ff \subset \binom{[n]}{k}$, one can use techniques from Boolean Analysis, by applying them to the upper-closure $\ff^\uparrow$ of the family $\ff$. 
There is a simple inequality, which relates the $p$-biased measure of $\ff^\uparrow$ and the size of $\ff$:
\begin{claim}[Lemma IV.2.2.~\cite{keevash2021global}]
\label{claim: p-biased measure bound on uniform measure}
    Let $\ff$ be a subfamily of $\binom{[n]}{k}$ and set $p = \frac{k}{n}$, then
    $
        \mu_p(\ff^\uparrow) \ge \frac{|\ff|}{4 \binom{n}{k}}.
    $
\end{claim}

When dealing with the restrictions $\ff(A, B)$ of a family $\ff$, one should be careful with the $p$-biased measure of $\ff(A, B)$. Indeed, for a set $B$, one can treat $\ff(\varnothing, B)$ either as a subfamily of the Boolean cube $2^{[n]}$ or as a subfamily of the Boolean cube $2^{[n] \setminus B}$. Both cases arise in our proofs, so we define $\mu_p^{-B}$ for the $p$-biased measure on  $2^{[n] \setminus B}$ and $\mu_p^{[n]}$ for $2^{[n]}$. We may omit the superscript if the measure and the domain are clear from the context. For restrictions of the form $\ff(B)$, we assume that the measure $\mu_p^{-B}(\ff(B))$ is considered by default.

The next definition, which is a variant of $\tau$-homogeneity, is the core notion of this section.
\begin{definition}\label{deftaugl}
    A family $\ff \subset 2^{[n]}$ is $\tau$-global with respect to the measure $\mu_p^{[n]}$, if for any set $B$ and a subset $A \subset B$, the following holds:
    \begin{align*}
        \mu_p^{-B}(\ff(A, B)) \le \tau^{|B|} \mu_p^{[n]}(\ff).
    \end{align*}
\end{definition}

In terms of the Boolean functions $f : \{0, 1\}^n \to \{0, 1\}$,  $\tau$-globalness is defined as follows:
\begin{align*}
    \mu_p^{-S}(f_{S \to x}) \le \tau^{|S|} \mu_p(f)
\end{align*}
for any $S \subset [n]$ and any $x \in \{0, 1\}^S$, where the function $f_{S \to x}: \{0, 1\}^{[n] \setminus S} \to \{0, 1\}$ is the restriction of $f$ on $S$, that is, it equals $f(x, y)$ for any $y \in \{0, 1\}^{[n] \setminus S}$. In the set families notation, $f_{S\to x}$ vs $f$ corresponds to $\ff(A,S)$ vs $\ff,$ where $A = \{i\in S: x_i=1\}.$

Importantly, given a family $\ff$, it is relatively simple to obtain a restriction $B$, $A \subset B$, such that $\ff(A, B)$ is $\tau$-global.

\begin{observation}[Section 2.3~\cite{keller_t-intersecting_2024}]
\label{observation: tau globalness restriction}
    For a family $\ff \subset 2^{[n]}$, consider a pair of sets $A, B$, $A \subset B$, that maximizes the quantity $\tau^{-|B|} \mu_p^{-B}(\ff(A, B))$. Then, $\ff(A, B)$ is $\tau$-global with respect to $\mu_p$. Additionally, if $\frac{1}{1 - p} < \tau$, then the maximum is attained when $A = B$.
\end{observation}

Before we will state the main theorem of this section, we need several other definitions. For a vector $x \in \{0, 1\}^n$, we say a random vector $y$ is $\rho$-correlated with $x$, if it can be obtained as follows: independently for all $i$, with probability $\rho$ set $y_i = x_i$ and with probability $1 - \rho$ sample $y_i$ from the measure $\mu_p$ on $\{0, 1\}$. We denote the resulting distribution as $N_\rho(x)$.

Next, we define the following linear operator on $L_2(\{0,1 \}^n, \mu_p)$. For a function $f \in L_2(\{0,1 \}^n, \mu_p)$, set $(T_\rho f)(x) = \EE_{y \sim N_{\rho}(x)} f(y)$. The operator $T_\rho$ is called the {\it noise operator}.

In general, for small enough $\rho$, the operator $T_\rho$ admits the following property: $\Vert T_\rho f \Vert_q \le \Vert f \Vert_2$ for $q>2$. Such property is referred to as the \textit{hypercontractivity}. For an extensive introduction into  Boolean Analysis and  hypercontractivity on the Boolean cube, we refer the reader to the seminal book by O'Donnel~\cite{o2014analysis}. For general hypercontractivity and its connections to Markov chains and functional analysis, we recommend the book~\cite{bakry_analysis_2014} by Bakry, Gentil and Ledoux.

For a function $f$, one may also consider the following quantity, which is called the {\it stability} of $f$:
\begin{align*}
    \Stab_\rho(f) = \langle T_\rho f, f \rangle.
\end{align*}
In general, if the $q$-th norm of $T_{\rho}f$ is efficiently bounded, one can control $\Stab_{\rho}(f)$ via the H\"older inequality, see the proof of Lemma~\ref{lemma: one step sharp threshold}.

The main result of the paper~\cite*{keller_sharp_2023} is that the hypercontractivity phenomenon happens for  $\tau$-global functions $f$.

\begin{theorem}[Theorem 1.4~\cite*{keller_sharp_2023}]
\label{theorem: global hypercontractivity}
    Let $f: \{0, 1\}^n \to \{0, 1\}$ be a $\tau$-global Boolean function, i.e.
    \begin{align*}
        \mu_p(f_{S \to x}) \le \tau^{|S|} \mu_p(f)
    \end{align*}
    for all $S \subset [n]$ and all $x \in \{0,1\}^S$. Then for any $q > 2$ and $\rho \le \frac{\ln q}{16 \tau q}$ we have $\Vert T_{\rho} f \Vert_q \le \Vert f \Vert_2$.
\end{theorem}

The theorem holds for any function $f : \{0, 1\}^n \to \RR$, but we do not use this fact. The result has numerous applications in combinatorics, as it was shown by the authors in the follow-up papers~\cite{keller_t-intersecting_2024, keller2023improved}. Intriguingly, in some cases, techniques based on the spread approximation and techniques based on the hypercontractivity can be used to obtain almost the same result, see papers~\cite{kupavskii_spread_2022, keller_t-intersecting_2024, kupavskii2024almost}.

To apply Theorem~\ref{theorem: global hypercontractivity} in our special case, we notice the following. We use the spread Lemma~\ref{theorem: r-spread theorem} to prove that $R$-spread families has large $\mu_{1/8s}$ measure, and, therefore, they contain a matching. In Boolean Analysis, the phenomenon when the $p$-biased measure $\mu_p(f)$ of a function $f$ is rapidly increasing when $p$ is slightly increasing is known as the \textit{sharp threshold} phenomenon. The fact that pseudorandom families (such as spread or global) admit sharp thresholds was exploited before to obtain Chv\'atal Simplex conjecture in the seminal paper~\cite{keller2021junta} by Ellis, Keller and Lifshitz. Later, this tool was further developed in paper~\cite{keevash2021global}, where the following proposition is presented.

\begin{proposition}[Proposition III.3.4~\cite{keevash2021global}]
\label{proposition: sharp threshold Keevash}
    Let $f: \{0, 1\}^n \to \{0, 1\}$ be a monotone Boolean function. Let $0 < p < \widetilde{p} < 1$ and $\rho = \frac{p (1 - \widetilde p)}{\widetilde p ( 1- p)}$. Then
    \begin{align*}
        \mu_{\widetilde p}(f) \ge \frac{\mu_p^2(f)}{\Stab_{\rho}(f)},
    \end{align*}
    where $\Stab_{\rho}$ is w.r.t. the measure $\mu_p$.
\end{proposition}

This result together with the hypercontractivity bound will be used in Section~\ref{subsection: sharp threshold and matchings} to reprove the results of Sections~\ref{subsection: homogeneity and general domains} and~\ref{subsection: spread approximation} for $\mu_p$ in the case when $n$ is linear in $k$.

\subsection{Sharp thresholds and matchings}
\label{subsection: sharp threshold and matchings}

We start by the following lemma, which will be a building block for the $\mu_p$-counterpart of Theorem~\ref{theorem: r-spread theorem}.

\begin{lemma}
\label{lemma: one step sharp threshold}
    Suppose that an upper-closed family $\ff$ is $\tau$-global. For a number $p \in (0; 2^{-7}/\tau)$, define $\widetilde p = 2^{6} \tau p$. Then, the following holds:
    \begin{align*}
        \mu_{\widetilde{p}}(\ff) \ge \mu_p^{3/4}(\ff).
    \end{align*}
\end{lemma}

\begin{proof}
    We consider $\ff$ as a monotone Boolean function. Set $\rho = \frac{p ( 1- \widetilde p)}{\widetilde p ( 1- p)}$. Since $p \le \widetilde p$, we have $1 - \widetilde p \le 1- p$, and $\rho \le \frac{p}{ \widetilde p} = \frac{1}{2^6 \tau}$. Hence, we may apply Theorem~\ref{theorem: global hypercontractivity} with $q = 4$ and obtain $\Vert T_\rho f \Vert_4 \le \Vert f \Vert_2$. An application of H\"older's inequality implies
    \begin{align*}
        \Stab_\rho(f) = \langle T_{\rho} f, f \rangle \le \Vert T_\rho f \Vert_{4} \Vert f \Vert_{4/3} \le \Vert f \Vert_2 \Vert f \Vert_{4/3}.
    \end{align*}
    Since $f$ is Boolean, $\Vert f \Vert_2 = \mu_p^{1/2}(f)$ and $\Vert f \Vert_{4 / 3} = \mu_p^{3/4}(f)$. Therefore, $\Stab_\rho(f) \le \mu_p^{5/4}(f)$. By our choice of $\rho$, we can apply Proposition~\ref{proposition: sharp threshold Keevash} and obtain
    \begin{align*}
        \mu_{\widetilde p}(f) \ge \frac{\mu^{2}_p(f)}{\Stab_\rho(f)} \ge \frac{\mu^2_p(f)}{\mu^{5/4}_p(f)} = \mu^{3/4}_p(f). & \qedhere
    \end{align*}
\end{proof}

The main idea is to iterate Lemma~\ref{lemma: one step sharp threshold}. Applying Lemma~\ref{lemma: one step sharp threshold} several times to a family $\ff$, one may obtain that  $\mu_{\widetilde p}(\ff)$ is large for $\widetilde p$ much greater than $p$.

\begin{corollary}
\label{corollary: upgrading the measure}
    Fix integers $m, z$ and $\tau \ge 2$. Let $\ff$ be an upward closed family such that $\mu_p(\ff) \ge \tau^{-z}$ and $p \in (0; (2^6 \tau)^{-m} / 2)$. Then, there exists a set $R$ of size at most $4 z$ such that, for $\widetilde p = (2^6 \tau)^m p$, the following holds
    \begin{align*}
        \mu_{\widetilde p}^{-R}(\ff(R)) \ge \left ( \mu_{p}(\ff) \right )^{(\frac{3}{4})^m},
    \end{align*}
    and $\ff(R)$ is $\tau$-global w.r.t the measure $\mu_{\widetilde p}^{-R}$.
\end{corollary}

\begin{proof}
    We obtain the result by an iterative application of Lemma~\ref{lemma: one step sharp threshold}. Set $\ff_0 = \ff, p_0 = p$. Note that $\mu_{p_0}(\ff_0) \ge \tau^{-z}$. At step $i$, find $R_i$ that maximizes the quantity $\tau^{-|S|} \mu_{p_i}(\ff_i(R_i))$ and set $\ff_{i + 1} = \ff_i(R_i)$. We prove by induction that $\mu_{p_i}(\ff_i) \ge \mu_p(\ff)^{\left ( \frac{3}{4} \right )^i} \ge \tau^{-z \left ( \frac{3}{4} \right )^i}$. Clearly, the base $i = 0$ holds.
    
    Due to Observation~\ref{observation: tau globalness restriction}, $\ff_i(R_i)$ is $\tau$-global.
    By the inductive assumption and the choice of $R$, we have 
    \begin{align*}
        \tau^{-z \left ( \frac{3}{4} \right )^i} \le \mu_{p_i}(\ff_i) \le \mu_{p_i}(\ff_i(R_i)) \tau^{-|R_i|} \le \tau^{-|R_i|},
    \end{align*}
    and, therefore, the set $R_i$ has size at most $z \left ( \frac{3}{4}\right )^i$. Define $p_{i + 1} = 2^6 \tau p_i$ and $\ff_{i + 1} = \ff_i (R_i)$. Due to Lemma~\ref{lemma: one step sharp threshold}, we have
    \begin{align*}
        \mu_{p_{i + 1}}(\ff_{i + 1}) \ge \mu_{p_i}^{3/4}(\ff_i) \ge (\mu_p(\ff_0))^{(3/4)^{i + 1}},
    \end{align*}
    and the induction step is proved.
    
   Stop the procedure after $m$ steps and set $R = \bigsqcup_{i = 1}^m R_i$. We have
    \begin{align*}
        |R| = \sum_{i = 0}^m |R_i| \le z \sum_{i =0}^\infty \left ( \frac{3}{4}\right )^i = 4 z.
    \end{align*}
    We have $\ff_m = \ff(R)$, and the corollary follows.
\end{proof}

Corollary~\ref{corollary: upgrading the measure} is a counterpart of Theorem~\ref{theorem: r-spread theorem}. To obtain a statement corresponding to Proposition~\ref{proposition: coloring trick}, we need the following claim, which coincides with Claim~\ref{claim: removing intersections} up to the measure used.

\begin{claim}
\label{claim: removing intersections mu p}
    Fix numbers $\tau, p$, $p < 1 < \tau$. Let $X$ be a set of size strictly less than $1 / \tau p$. Suppose that a family $\ff \subset 2^{[n]}$ is $\tau$-global w.r.t. measure $\mu_p^{[n]}$. Then the family $\cG = \ff(\varnothing, X)$ is $\hat \tau$-global w.r.t. $\mu_p^{[n]}$ for
    \begin{align*}
        \hat{\tau} = \frac{\tau}{1 - |X| p \tau},
    \end{align*}
    and $\mu_p(\cG) \ge (1 - |X| p \tau) \mu_p(\ff)$.
\end{claim}
\begin{proof}
    Let $\cG = \ff(\varnothing, X)$. Consider an arbitrary non-empty $B$ and $A \subset B$. We have
    \begin{align}
    \label{eq: excluded X mu p}
        \mu_p^{-B}(\cG(A, B)) \le \mu_p^{-B}(\ff(A, B)) \le \tau^{|B|} \mu_p^{[n]}(\ff),
    \end{align}
    since $\ff$ is $\tau$-global.
    We claim that $\mu_p^{[n]} (\ff[x]) \le p \tau \mu_p(\ff)$ for any $x \in [n]$. Indeed, we have
    \begin{align*}
        \mu_p^{[n]} (\ff[x]) = \sum_{F \in \ff(x)} p^{|F| + 1} (1 - p)^{n - |F| - 1} = p \mu_p^{-\{x\}}(\ff(x)) \le p \tau \mu_p(\ff).
    \end{align*}
    It yields $\mu_p^{[n]}(\cG) \ge \mu_p(\ff) - \sum_{x \in X} \mu_p^{[n]} (\ff[x]) \ge (1 - p \tau |X|) \mu_p(\ff)$. Substituting it into~\eqref{eq: excluded X mu p}, we get
    \begin{align*}
        \mu^{-B}_p(\cG(A, B)) \le \tau^{|B|} \mu_p(\ff) \le \left (\frac{\tau}{1 - |X| p \tau} \right)^{|B|} \mu_p^{[n]}(\cG),
    \end{align*}
    so $\cG$ is indeed $\hat \tau$-global.
\end{proof}

Finally, we obtain a counterpart of Proposition~\ref{proposition: coloring trick}. The difference is that we require the families to have not too small measure, so that we can upscale it by Corollary~\ref{corollary: upgrading the measure}.

\begin{lemma}
\label{lemma: matchings in global families}
    Let $z$ be an integer and $\tau \ge 4$ some positive number. Let $\ff_1 \subset \binom{[n]}{k_i}, \ldots, \ff_s \subset \binom{[n]}{k_s}$ be $\tau/2$-homogeneous families with $|\ff_i| \ge 8 \tau^{-z} \binom{[n]}{k_i}$. Set $p_i = k_i/n$ and $m = \left \lceil \log_{4/3} \left ( \frac{z \log \tau}{\log 2} \right ) \right \rceil$. Suppose that $s z (2^6 \tau)^m \tau p_i < 1/64$ for each $i = 1, \ldots, s$. Then, there exist pairwise disjoint sets $F_1 \in \ff_1, \ldots, F_s \in \ff_s$.
\end{lemma}

\begin{proof}
    We apply Corollary~\ref{corollary: upgrading the measure} to the upper closures of families $\ff_1, \ldots, \ff_s$ step by step. By our choice of $m$, we have
    \begin{align*}
        \left ( \frac{3}{4}\right )^m z \log \tau \le \log 2.
    \end{align*}
    Set $p = k_i/n$  and $\widetilde p_i = (2^6 \tau)^m p_i$. Consider the upper closure $\cQ_1 = \ff_1^\uparrow$ of the family $\ff_1$. Due to Claim~\ref{claim: p-biased measure bound on uniform measure}, we have $\mu_{p_1}(\cQ_1) \ge \mu(\ff_1) / 4 \ge \tau^{-z}$. Applying Corollary~\ref{corollary: upgrading the measure} to $\cQ_1$, we obtain a set $R_1$, $|R_1| \le 4 z$, such that
    \begin{align*}
        \mu_{\widetilde p_1}^{-R_1}(\cQ_1(R_1)) \ge (\mu_{p_1}(\cQ_1))^{(3/4)^m} \ge (\tau^{-z})^{(3/4)^m} \ge e^{-\log 2} = \frac{1}{2},
    \end{align*}
    and $\cQ_1(R_1)$ is $\tau$-global.

    Then we proceed inductively. On step $i + 1$, suppose that we have obtained disjoint sets $R_1, \ldots, R_i$. Set $R_{\le i} = \bigsqcup_{j = 1}^i R_j$, $|R_{\le i}| \le 4 i z$. By Proposition~\ref{proposition: spread families examples}, the family $\binom{[n]}{k_i}$ is $n/k_i$-spread. Put $r_i = n/k_i$. By Claim~\ref{claim: removing intersections}, the family $\ff_{i + 1}(\varnothing, R_{\le i})$ is $\hat \tau$-homogeneous for
    \begin{align*}
        \hat \tau = \frac{\tau/2}{1 - \frac{2 |R_{\le i}| \tau }{r_i}} \le \frac{\tau/2}{1 - \frac{8 s z k_i \tau}{n }} = \frac{\tau/2}{1 - 8 s z p_i \tau} \le \tau.
    \end{align*}
    and
    \begin{align*}
        \mu(\ff_{i + 1}(\varnothing, R_{\le i})) \ge\left ( 1 - \frac{2 |R_{\le i}| \tau }{r_i}\right ) \mu(\ff_{i + 1}) \ge (1 - 8 s z p_i \tau) \mu(\ff_{i + 1})  \ge 4 \tau^{-z}.
    \end{align*}
    Consider the upper closure $\cQ_{i + 1} = [\ff_{i + 1}(\varnothing, R_{\le i})]^\uparrow$ of the family $\ff_{i + 1}(\varnothing, R_{\le i})$ w.r.t $[n] \setminus R_{\le i}$. Due to Claim~\ref{claim: p-biased measure bound on uniform measure}, we have $\mu_{p_{i + 1}}^{-R_{\le i}}(\cQ_{i + 1}) \ge \tau^{-z}$. Apply Corollary~\ref{corollary: upgrading the measure} and obtain a set $R_{i + 1}, |R_{i + 1}| \le 4 z$, such that $\mu_{\widetilde p_{i + 1}}^{-(R_{\le i} \sqcup R_{i + 1})} (\cQ_{i + 1}(R_{i + 1})) \ge 1/2$ and $\cQ_{i + 1}(R_{i + 1})$ is $\tau$-global. 

    When the procedure stops, define $R = \bigsqcup_{i = 1}^s R_i$. Set $\cG_i = \cQ_{i}(R_{i}, R)$. Due to Claim~\ref{claim: removing intersections mu p}, we have
    \begin{align*}
        \mu_{\widetilde p_i}^{-R_{\le i}}(\cG_i) & \ge (1 - \widetilde p_i \tau |R|) \mu_{\widetilde p_i}^{-R_{\le i}}(\cQ_{i}(R_i)) \\
        & \ge (1 - 4 (2^6 \tau)^{m} p_i \tau s z) \mu_{\widetilde p_i}^{- R_{\le i}}(\cQ_i(R_i))\ge \frac{1}{4}, 
    \end{align*}
    with a convention $R_{\le 0} = \varnothing$.
    Since $\mu_{\widetilde p_i}^{-R}(\cG_i) \ge \mu_{\widetilde p_i}^{-R_{\le i}}(\cG_i)$, we have $\mu_{\widetilde p_i}^{-R}(\cG_i) \ge 1/4$. Since $s \widetilde p_i \le 1/16$, we have $\mu_{\widetilde p_i}^{-R}(\cG_i) \ge 3 s \widetilde p_i$. Therefore, there exist disjoint sets $G_1 \in \cG_1, \ldots, G_s \in \cG_s$ due to Lemma~\ref{lemma: matching lemma}. Thus, there exist disjoint sets $R_1 \cup G_1 \in \ff_s, \ldots, R_s \cup G_s \in \ff_s$, and the lemma follows.
\end{proof}

\section{Proof of Theorem~\ref{theorem: main theorem}}
\label{section: proof of the main result}

\begin{proof} We consider several cases depending on the behaviour of $k$ as a function of $n$. 

\textbf{Case 1: $\mathbf{k \ge 33 s t^2 \ln n}$.} Here $\ln(\cdot)$ stands for the natural logarithm. We apply the following theorem, which is proved in Section~\ref{section: proof of theorem for large uniformity}. Moreover, using  this theorem, we will derive analogs of Theorem~\ref{theorem: main theorem} for subfamilies of $\Sigma_n$ and $[n]^k$. The statement of the theorem is given in the notation of Assumptions~\ref{assumption: spreadness - small r}-\ref{assumption: weak shadow consistency - small r}.

\begin{theorem}
    \label{theorem: large uniformity sunflowers}
    Suppose that $\cA$ satisfies Assumptions~\ref{assumption: spreadness - small r}-\ref{assumption: weak shadow consistency - small r} with $ \varepsilon r \ge  s^2q^3$ for some $\varepsilon \in (0, 1/2]$ and $q = \lceil  16 \eta s t^2  \ln r \rceil$. Suppose additionally, that at least one of the following two conditions holds:
    \begin{enumerate}[label={\roman*)}]
        \item \label{condition: large spreadness} $r \ge 2^{11} s \lceil \log_2 k \rceil$;
        \item \label{condition: sharp threshold applicability} $\cA = \binom{[n]}{k}$ (so $\eta = 2, r=n/k, q = \lceil 32 s t^2 \ln \frac{n}{k} \rceil, \mu = n/k$ due to Proposition~\ref{proposition: assumptions satisified}) and $n / k \ge 2^{390}s^{43} t^{82} + 2^{23000}$.
    \end{enumerate}
    Assume that $k \ge 2 t + q$, $\mu k \ge 2 s t q $. Let $\ff \subset \cA$ be a family without a sunflower with $s$ petals and the core of size exactly $t - 1$. Then, there exists a family $\cT \subset \partial_t \cA$ of uniformity $t$ without a sunflower with $s$ petals (and the core of arbitrary size), such that
    \begin{align*}
        |\ff \setminus \cA[\cT]| \le 32 r^{-\eta t} A_t +   \phi(s, t) \cdot 4 \varepsilon A_t.
    \end{align*}
\end{theorem}

The proof of Theorem~\ref{theorem: large uniformity sunflowers} is postponed to Section~\ref{section: proof for small k theorem}.

Let us discuss the case $\mathcal A = {[n]\choose k}$. Note that $\cA = \binom{[n]}{k}$ can satisfy both~\ref{condition: large spreadness} and~\ref{condition: sharp threshold applicability}, depending on $n$ and $k$. First, we check when the conditions of the theorem are satisfied. Due to Proposition~\ref{proposition: assumptions satisified}, we have $r = n/k$, $\eta = 2$, $\mu = n/k$ and $q = \lceil 32 s t^2 \ln r\rceil \le 33 s t^2 \ln(n/k)$. Thus, the condition $\varepsilon r \ge s^2 q^3$ for some $\varepsilon \le 1/2$ is satisfied if
\begin{align*}
    n/k \ge 2s^2 (33 s t^2 \ln(n/k))^3 = 2 \cdot 33^3 s^5 t^6 \ln^3(n/k).
\end{align*}
If $n/k \ge 2^{25}$, then $\ln(n/k) \le (n/k)^{1/6}$. Therefore, it is enough to guarantee the following two inequalities
\begin{align*}
    n \ge 2^{25} k \quad \text{ and } n/k \ge (2 \cdot 33^3)^2 s^{10} t^{12}.
\end{align*}
Both of them are satisfied if $n \ge 2^{33} s^{10} t^{12} k$. Condition~\ref{condition: large spreadness} is satisfied if $n \ge 2^{11} s k \log_2 k$. Condition~\ref{condition: sharp threshold applicability} is satisfied if $n \ge (2^{390} s^{43} t^{82} + 2^{23000}) k$. The condition $\mu k \ge 2 s t q$ is satisfied, since $\mu = n/k$ and $n \ge 2 s^2 q^3 \ge 2 s t q$. Finally, we have $k \ge 33 st^2 \ln n \ge 2t + q$, and so Theorem~\ref{theorem: large uniformity sunflowers} is applicable.




By the definition, we have $|\cT| \le \phi(s, t)$. Choosing $\varepsilon = \frac{s^2 q^3}{r}$, we obtain
\begin{align*}
    |\ff| & \le \phi(s, t) \binom{n - t}{k - t} + 32 \left ( \frac{k}{n} \right )^{2t} \binom{n - t}{k - t} + \phi(s, t) \cdot \frac{4 \cdot s^2 (33 s t^2 \ln(n/k))^3}{n/k} \cdot \binom{n - t}{k - t} \\
    & \le \left (1 + \frac{2^{18} s^5 t^6 \ln^3(n/k)}{n/k} \right ) \phi(s, t) \binom{n - t}{k - t} \le \left ( 1 + \sqrt[3]{\frac{2^{54} s^{15} t^{18}}{n/k}}\right) \phi(s, t) \binom{n - t}{k - t}.
\end{align*}

The above argument gives us the following lower bound on $n$:
\begin{align*}
    n \ge 2^{11} s k \cdot \min \left \{  \log_2 k  + 2^{22} s^9 t^{11}, 2^{379} s^{42} t^{82} + 2^{23000}\right \}.
\end{align*}
It remains to study the case $k \le 33 s t^2 \ln n$, or, equivalently, $n \ge e^{k/33 s t^2}$. 

\textbf{Case 2. We have $\mathbf{n \ge e^{k/33st^2}}$ and $\mathbf{k \ge 2^{10} s t^3 \ln k}$.} In this case, we use that fact that $n \ge s^{40 t} \cdot g_0(t)$ for some large enough function $g_0(\cdot)$ depending on $t$. For such $k$, we have $k^{20 t} \le e^{k/33 st^2}$, and thus, we may apply the following theorem.

\begin{theorem}
\label{theorem: large r theorem}
Let $\cA = \binom{[n]}{k/w}^w$ for some $w | k$. Suppose that $n \ge k^{20 t}$, $n \ge s^{40 t} \cdot g_0(t)$ and $k \ge 5t$, where $g_0(t) \ge (2^{20} t^6)^{20t}$. Let $\ff \subset \cA$ be a family without a sunflower with $s$ petals and the core of size exactly $t - 1$. Then, we have
\begin{align*}
    |\ff| \le  \phi(s, t) |\partial_{k - t} \cA| + n^{-1/3} \cdot A_t.
\end{align*}
\end{theorem}

The proof of Theorem~\ref{theorem: large r theorem} is postponed to Section~\ref{section: proof for small k theorem}.

For the ambient family $\cA = \binom{[n]}{k}$, we have $|\partial_{k - t} \cA| = \binom{n}{k - t}$. Therefore, the family $\ff$ as in Theorem~\ref{theorem: large r theorem} has size at most
\begin{align*}
    |\ff| \le \phi(s, t) \binom{n}{k - t} + n^{-1/3} \binom{n - t}{k - t}.
\end{align*}
Bounding 
\begin{align*}
    \frac{\binom{n}{k - t}}{\binom{n - t}{k - t}} \le \frac{n^{k - t}}{(n - k)^{k - t}} = \left (1 + \frac{k}{n - k} \right )^{k - t} \le 1+\frac {2k^2} {n - k},
\end{align*}
we obtain
\begin{align*}
    |\ff| \le \left (1 + \frac{k^2}{n - k} + n^{-1/3} \right ) \phi(s, t) \binom{n - t}{k - t} \le \left (1 + \sqrt[3]{\frac{2^{54} s^{15} t^{18}}{n/k}} \right ) \phi(s, t) \binom{n - t}{k - t}.
\end{align*}

Cases 1 and 2 together cover the regime
\begin{align*}
     n \ge 2^{11} s k \cdot \min \left \{  \log_2 k  + 2^{22} s^9 t^{11}, 2^{379} s^{42} t^{82} + 2^{23000}\right \}, \quad n \ge s^{40t} \cdot g_0(t),
\end{align*}
and $k \ge 2^{10} s t^3 \log k$. It remains to consider the case $2t + 1 \le k \le 2^{13} s t^3 \log k$.

\textbf{Case 3. We have $\mathbf{5t \le k \le 2^{10} s t^3 \log k}$.} In particular, we have $k \le 2^{20} s^2 t^6$, and, since $n \ge s^{40t} \cdot g_0(t) \ge k^{20 t}$, we may again apply Theorem~\ref{theorem: large r theorem} and obtain
\begin{align*}
    |\ff| \le  \left (1 + \sqrt[3]{\frac{2^{54} s^{15} t^{18}}{n/k}} \right ) \phi(s, t) \binom{n - t}{k - t} 
\end{align*}
as before.

\textbf{Case 4. We have $\mathbf{2t + 1\le k < 5t}$.} In this case, we apply Theorem~\ref{theorem: delta-system solution}, and, obtain
\begin{align*}
    |\ff| \le \phi(s, t) \binom{n}{k - t} + \frac{(5st)^{2^{Ct}}}{n} \binom{n}{k - t}
\end{align*}
for some absolute constant $C$. The desired bound holds provided $n \ge n_0(s, t)$. 
\end{proof}

\section{Erd\H{o}s--Duke problem for spread domains}
\label{section: erdos-duke for spread domains}

Before we move on to proofs of Theorems~\ref{theorem: large uniformity sunflowers},\ref{theorem: large r theorem}, let us discuss how results of Section~\ref{section: proof of the main result} can be generalized to other domains.

\subsection{Permutations}
Let us show that Theorem~\ref{theorem: large uniformity sunflowers} can be applied to families of permutations. Due to Proposition~\ref{proposition: permutations satisfy assumptions}, the family $\mathfrak{S}_n$ satisfies Assumption~\ref{assumption: spreadness - small r}-\ref{assumption: weak shadow consistency - small r} with $k = n$, $r = n/4$, $\eta = 2$ and $\mu = 4$. Hence, the condition $\varepsilon r \ge s^2 q^3$ for $\varepsilon \le 0.5$ of Theorem~\ref{theorem: large uniformity sunflowers} requires $n/4 \ge 2 s^2 q^3 = 2 s^2 \lceil 32 s t^2 \ln(n/4)\rceil^3$. This inequality is fulfilled if $n \ge 2^7 s^5 t^6 \ln^3 n$. Next, condition $r \ge 2^{11} s \lceil \log_2 k \rceil$ is implied by $n/4 \ge 2^{12} s \ln(n/4) / \ln(2)$. Thus, if $n \ge 2^{15} s^5 t^6 \ln^3 n$, both conditions are satisfied. Under this assumption, the inequalities $k = n \ge 2t + q$ and $\mu k = n/2 \ge 2 s t q$ hold, so we can deduce the following corollary of Theorem~\ref{theorem: large uniformity sunflowers} with $\varepsilon = 4s^2 q^3/n$.

\begin{corollary}
\label{corollary: erdos-duke for permutations}
    Suppose that $n \ge 2^{15}s^5 t^6 \ln^3 n$. Let $\ff \subset \Sigma_n$ be a family of permutations that does not contain a sunflower with $s$ petals and the core of size $t -1$ in the sense of Definition~\ref{definition: sunflower of permutation}. Then, we have
    \begin{align*}
        |\ff| \le \left (1 + \frac{2^{21} s^5 t^6 \ln^3 n}{n} \right )\phi(s, t) (n - t)!.
    \end{align*}
\end{corollary}
\begin{proof}
    Consider the subfamily $\mathfrak{F} \subset \mathfrak{S}_n$ of sets corresponding to the family $\ff \subset \Sigma_n$, see Section~\ref{subsection: spread domains}. Then, $\mathfrak{F}$ meets conditions of Theorem~\ref{theorem: large uniformity sunflowers} by the reasoning above. It implies
    \begin{align*}
        |\mathfrak{F}| & \le \phi(s, t) \cdot (n - t)! + 32 \cdot (n/4)^{-2t} (n - t)! + \frac{16 s^2 (33  s t^2 \ln(n/4))^3}{n} \cdot (n - t)! \\
        & \le \left ( 1 + \frac{2^{21} s^5 t^6 \ln^3 n}{n}\right ) \phi(s, t) \cdot (n - t)!. & \qedhere
    \end{align*}
\end{proof}

One can construct an example similar to Example~\ref{example: basic example}, which shows that Corollary~\ref{corollary: erdos-duke for permutations} is tight up to second-order terms.
\begin{proposition}
    Let $n >t \phi(s, t) + t$. Then, there exists a family $\ff \subset \Sigma_n$ without a sunflower with $s$ petals and the core of size $t - 1$, such that
    \begin{align*}
        |\ff| \ge \left (1 - \frac{t \phi(s, t)}{n - t} \right )\cdot \phi(s,t) (n - t)!.
    \end{align*}
\end{proposition}
\begin{proof}
Let $\cG$, $|\cG| = \phi(s, t)$, be an extremal family of uniformity $t$ without a sunflower with $s$ petals. Assume that $n \ge t^2 \phi^2(s, t) \ge |\support \cG|^2$ and $\support \cG \subset [n]$. We will construct a family $\ff \subset \Sigma_n$ without a sunflower with $s$ petals as the family of all permutations $\sigma$ for which there exists $G \in \cG$ with the following properties:
\begin{enumerate}
    \item for any $x \in G$, we have $\sigma(x) = x$;
    \item for any $x \in (\support \cG) \setminus G$, we have $\sigma(x) \neq x$.
\end{enumerate}
The number of permutations such that for some set $Y$ we have $\sigma(y) = y$ for any $y \in Y$, and for some set $X$ we have $\sigma(x) \neq x$ for any $x \in X$, is at least
\begin{align*}
    (n - |Y|)! - |X| \cdot (n - |Y| - 1)! = (n - |Y|)! \left (1 - \frac{|X|}{n - |Y|} \right ).
\end{align*}
Hence, we can bound the size of $\ff$ as follows:
\begin{align*}
    |\ff| \ge \left (1 - \frac{t \phi(s, t)}{n - t} \right )\cdot \phi(s,t) (n - t)!.
\end{align*}
It remains to show that $\ff$ does not contain a sunflower with $s$. Suppose that $\sigma_1, \ldots, \sigma_s$ form a sunflower with the core of size $t - 1$. Define $G_i$ as the set of all $x \in \support \cG$ such that $\sigma_i(x) = x$. Note that by the construction of $\ff$, we have $G_i \in \cG$. We claim that $G_1, \ldots, G_s$ form a sunflower with the core of size at most $t - 1$. Indeed, if $x \in G_i \cap G_j$, then $\sigma_i(x) = \sigma_j(x) = x$. Then, we have $\sigma_1(x) = \sigma_2(x) = \ldots = \sigma_s(x) = x$ by Definition~\ref{definition: sunflower of permutation} of sunflowers for permutations. By the construction of the family $\ff$, we have $x \in G_\ell$ for all $\ell \in [s]$, and so $G_i \cap G_j = \cap_\ell G_\ell$. Thus, $G_1, \ldots, G_s$ form a sunflower with $s$ petals and the core of size at most $t - 1$, contradicting the definition of $\cG$.
\end{proof}

\subsection{Subfamilies of product domains}

In contrast to the case of permutations, one should study the regime when $k \le \ln n$ to obtain a counterpart of Theorem~\ref{theorem: main theorem} for $\cA = \binom{[n]}{k/w}^w$. Theorem~\ref{theorem: large r theorem} gives suboptimal bounds. For example, for $\cA = [n]^k$, we have $|\partial_{k - t} \cA| = \binom{k}{t} n^{k - t}$ and so, for an extremal family $\ff$, Theorem~\ref{theorem: large r theorem} implies $|\ff| \le (1 + o(1)) \binom{k}{t} \phi(s, t) n^{k - t}$ as $n$ tends to infinity. It turns out that with additional effort, this bound can be improved to $|\ff| \le \left (1 + O(\sqrt{t/k}) \right ) \phi(s, t) n^{k - t}$.  The next two theorems show that the result of Theorem~\ref{theorem: large r theorem} can be made essentially sharp for subfamilies of $\binom{[n]}{k/w}^w$.
\begin{theorem}
\label{theorem: product case - k-th power}
Let $\ff \subset [n]^k$ be a family without a sunflower with $s$ petals and the core of size $t - 1$. Suppose that $k \ge 5t$, $n \ge s^{80t} \cdot \tilde g_0(t)$ for some large enough function $\tilde g_0(t)$, and
\begin{align*}
    n  \ge 2^{33} s^{10} t^{12} + 2^{11} s \log_2 k.
\end{align*}
Then, we have
\begin{align*}
    |\ff| \le (1 + \zeta) \phi(s, t) \cdot n^{k - t}, \quad \text{where} \quad 
    \zeta = \max \left \{7\sqrt{\frac{t}{k}}, \frac{2^{18} s^5 t^6\ln^3 n}{n} \right \}.
\end{align*}
\end{theorem}

\begin{theorem}
\label{theorem: product case - w-th power}\
Consider an ambient family $\cA = \binom{[n]}{k/w}^w$ for $k \ge 2^{12} w^6 t$. Suppose that $n \ge s^{60t} \cdot \tilde g_0(t)$ for large enough function $\tilde g_0(t)$, and
\begin{align*}
    wn \ge (2^{33} s^{10} t^{12} + 2^{11} s \log_2 k) \cdot k.
\end{align*}
Then, for any family $\ff \subset \cA$ that does not contain a sunflower with $s$ petals and the core of size exactly $t - 1$, we have
\begin{align*}
    |\ff| \le \left (1 + \zeta' \right ) \phi(s, t) \cdot A_t, \quad \text{where} \quad \zeta' = \max \left \{ 7 \sqrt{\frac{t}{k}}, \frac{2^{18} s^5 t^6 \ln^3(wn/k)}{wn/k} \right \}.
\end{align*}
\end{theorem}

We prove both theorems simultaneously. Here $\cA$ stands for the ambient family, which is $[n]^k$ for Theorem~\ref{theorem: product case - k-th power} and $\binom{[n]}{k/w}^w$ for Theorem~\ref{theorem: product case - w-th power}.

\begin{proof}
As in the proof of Theorem~\ref{theorem: main theorem}, we consider several cases.

\textbf{Case 1: $\mathbf{k \ge 33 ws t^2 \ln n}$ if $\cA = \binom{[n]}{k/w}^w$ and $\mathbf {k \ge 33st^2 \ln n}$ if $\mathbf{\cA = [n]^k}$. }  In this case, we apply  Theorem~\ref{theorem: large uniformity sunflowers}. We check its  conditions for both families. 

\begin{itemize}
    \item If $\cA = \binom{[n]}{k/w}^w$ for $k \ge33 ws t^2 \ln n$, then $\cA$ satisfies Assumptions~\ref{assumption: spreadness - small r}-\ref{assumption: weak shadow consistency - small r} with $r = \frac{wn}{k}$, $\eta = 2$, $\mu = n/k$ and $q = \lceil  32 st ^2 \ln r \rceil$ due to Proposition~\ref{proposition: assumptions satisified}. The condition $\varepsilon r \ge s^2 q^3$ for some $\varepsilon \le 1/2$ is satisfied if
    \begin{align*}
        \frac{wn}{k} \ge 2s^2 (33 s t^2 \ln(n/k))^3 = 2 \cdot 33^2 s^5 t^6 \ln^3 \left (\frac{wn}{k} \right).
    \end{align*}
    If $\frac{wn}{k} \ge 2^{25}$, then $\ln \left ( \frac{wn}{k} \right ) \le \left ( \frac{wn}{k} \right )^{1/6}$. Thus, it is enough to guarantee the following two inequalities:
    \begin{align*}
        \frac{wn}{k} \ge 2^{25} \quad \text{and} \quad \frac{wn}{k} \ge (2 \cdot 33^3)^2 s^{10} t^{12}.
    \end{align*}
    Both are satisfied if $wn \ge 2^{33} s^{10} t^{12}k$. The condition $r \ge 2^{11} s \lceil \log_2 k \rceil$ holds if $wn \ge 2^{12} s k \log_2 k$. Assumptions $k \ge 2t + q$ and $n = \mu k \ge 2 st q$ are satisfied provided $k \ge 33 s w t^2n \ln n$ and $n \ge \frac{wn}{k} \ge 2 s^2 q^3$. The conclusion of Theorem~\ref{theorem: large uniformity sunflowers} yields
    \begin{align*}
        |\ff| & \le \phi(s, t) A_t + 32 \left ( \frac{wn}{k} \right )^{-2t} A_t+ \frac{4 s^2 (33 s t^2 \ln (wn/k))^3}{wn/k} \phi(s, t) A_t \\
        & \le \left (1 + \frac{2^{18}s^5 t^6 \ln^3 (wn/k)}{wn/k} \right ) \phi(s, t) \cdot A_t.
    \end{align*}
    We summarize observations above in the following lemma.
    \begin{lemma}
    \label{lemma: suflower large uniformity for w-th power}
        Let $\cA = \binom{[n]}{k/w}^w$ for $k \ge 33 w st^2 \ln n$ and $wn \ge (2^{33} s^{10} t^{12}  + 2^{12} s \log_2 k) \cdot k$. Then, for any family $\ff \subset \cA$ without a sunflower with $s$ petals and the core of size $t - 1$, we have
        \begin{align*}
            |\ff| \le \left (1 + \frac{2^{18}s^5 t^6 \ln^3 (wn/k)}{wn/k} \right ) \phi(s, t) \cdot A_t.
        \end{align*}
    \end{lemma}
    \item If $\cA = [n]^k$ for $k \ge 2^8 s^2 t^3 \ln n$, then $\cA$ satisfies Assumptions~\ref{assumption: spreadness - small r}-\ref{assumption: weak shadow consistency - small r} with $r = n$, $\eta = 2$, $\mu = 1/2$ and $q = \lceil 32 s t^2 \ln n \rceil$ due to Proposition~\ref{proposition: assumptions satisified}. The condition $\varepsilon r \ge s^2 q^3$ for some $\varepsilon \le 1/2$ is satisfied if
    \begin{align*}
        n \ge 2 s^2 (33 s t^2 \ln n)^3 = 2 \cdot 33^3 s^5 t^6 \ln^3 n.
    \end{align*}
    If $n \ge 2^{25}$, then $\ln n \le n^{1/6}$. Thus, it is enough to guarantee the following two inequalities:
    \begin{align*}
        n \ge 2^{25} \quad \text{and} \quad n \ge (2 \cdot 33^3)^2 s^{10} t^{12}.
    \end{align*}
    Both are satisfied if $n \ge 2^{33} s^{10} t^{12}$. The condition $r \ge 2^{11} s \lceil \log_2 k \rceil$ holds if $n \ge 2^{12} s \log_2 k$. Assumptions $k \ge 2t + q$ and $k/2 = \mu k \ge 2 s t q$ are satisfied since $k \ge 2^8 s^2 t^3 \ln n \ge 4 s t q $. The conclusion of Theorem~\ref{theorem: large uniformity sunflowers} implies
    \begin{align*}
        |\ff| & \le \phi(s, t) n^{k - t} + 32 n^{-2t} n^{k - t} + \frac{4 s^2 (33 s t^2 \ln n)^3}{n} \phi(s, t) n^{k - t} \\
        & \le \left (1 + \frac{2^{18} s^5 t^6 \ln^3 n}{n} \right ) \phi(s, t) \cdot n^{k - t}.
    \end{align*}
    We summarize the above analysis in the following lemma.
    \begin{lemma}
    \label{lemma: sunflower large uniformity for k-th power}
        Let $n, k, s, t$ be integers such that $n \ge 2^{33} s^{10} t^{12} + 2^{12} s \log_2 k$ and $k \ge 2^8 s^2 t^3 \ln n$. Let  $\ff \subset [n]^k$ be a family without a sunflower with $s$ petals and the core of size $t - 1$. Then, we have
        \begin{align*}
            |\ff| \le \left (1 + \frac{2^{18} s^5 t^6 \ln^3 n}{n} \right ) \phi(s, t) \cdot n^{k - t}.
        \end{align*}
    \end{lemma}
    We conjecture that the lower bound $n \ge 2^{33} s^{10} t^{12} + 2^{12} s \log_2 k$ in the above lemma can be replaced with $n \ge \poly(s, t)$, using the semigroup technique introduced in~\cite{keevash2023forbidden}.
\end{itemize}

Thus, the case is resolved by Lemmas~\ref{lemma: suflower large uniformity for w-th power} and~\ref{lemma: sunflower large uniformity for k-th power}.\\

\textbf{Case 2: $\mathbf{\max\{2^{10} w s t^3 \ln k, 2^{12} w^6 t\} \le k \le 33w st^2 \ln n}$ for $\mathbf{\cA = \binom{[n]}{k/w}^w}$  and \newline $\mathbf{2^{13} s^2 t^4 \ln k \le k \le 2^8 s^2 t^3 \ln n}$ for $\mathbf{\cA = [n]^k}$. } In this case, we apply the following theorem.
\begin{theorem}
\label{theorem: product case - small k uniformity}
    Assume that one of the two following conditions hold:
    \begin{enumerate}
        \item $\cA = \binom{[n]}{k/w}^w$ for $k \ge 2^{12} w^6 t$;
        \item $\cA = [n]^k$.
    \end{enumerate}
    Suppose that $n \ge k^{20t}$, $n \ge s^{40t} \cdot g_0(t)$ and $k \ge 5t$, where $g_0(t) \ge (2^{20}t^6)^{20t}$.
    Let $\ff \subset \cA$ be a family without a sunflower with $s$ petals and the core of size exactly $t - 1$. Then, we have
    \begin{align*}
        |\ff| \le \left (1 +7 \sqrt{\frac{t}{k}} \right ) \phi(s,t) \cdot A_t.
    \end{align*}
\end{theorem}
The proof of Theorem~\ref{theorem: product case - small k uniformity} is sketched in Section~\ref{section: proof sketch for product case -- small k uniformity}.
\begin{itemize}
    \item If $\cA = \binom{[n]}{k/w}^w$ and $k \le 33w s t^2 \ln n$, then $n \ge e^{k/33 st^2}$. Since $k \ge 2^{10} w s t^3 \ln k$, $n \ge k^{20t}$ and since $k \ge 2^{12} w^6 t$, we can apply Theorem~\ref{theorem: product case - small k uniformity}, and obtain
    \begin{align*}
        |\ff| \le \left ( 1 + 7 \sqrt{\frac{t}{k}}\right ) \phi(s, t) A_t,
    \end{align*}
    provided $n \ge s^{40t} \cdot g_0(t)$.
    \item If $\cA = [n]^k$ and $k \le 2^{8} s^2 t^3 \ln n$, then $n \ge e^{k/2^{8} s^{2} t^{3}}$. Since $k \ge 2^{13} s t^4 \ln k$, we have $n \ge k^{20t}$, so we can apply Theorem~\ref{theorem: product case - small k uniformity}, and obtain
    \begin{align*}
        |\ff| \le \left ( 1 + 7 \sqrt{\frac{t}{k}}\right ) \phi(s, t) A_t,
    \end{align*}
    provided $n \ge s^{40t} g_0(t)$.
\end{itemize}

\textbf{Case 3. We have $\mathbf{2^{12}w^6 t \le k \le 2^{10} w s t^3 \ln k}$ if $\mathbf{\cA = \binom{[n]}{k/w}^w}$ and $\mathbf{5t \le k \le 2^{13} s^2 t^4 \ln k}$ if $\mathbf{\cA = [n]^k}$.} Consider two cases:
\begin{itemize}
    \item If $\cA = \binom{[n]}{k/w}^w$ and $2^{12} w^6 t \le k \le 2^{10} w s t^3 \ln k$, then
    \begin{align*}
        k & \le 2^{10} \left ( \frac{k}{2^{12} t} \right )^{1/6} s t^3 \ln k \le 2^8 k^{1/6} s t^3 \ln k \\
        & \le 2^8 s t^3 k^{2/3},
    \end{align*}
    so $k \le 2^{24} s^3 t^9$. Since $n \ge s^{60t} \cdot \tilde  g_0(t)$ for large enough $\tilde g_0(t)$, we have $n \ge k^{20t}$. Thus, we can apply Theorem~\ref{theorem: product case - small k uniformity}, and obtain Theorem~\ref{theorem: product case - w-th power}.
    \item If $\cA = [n]^k$ and $5t \le k \le 2^{10} s^2 t^4 \ln k$. In particular, we have $k \le 2^{20} s^4 t^8$. Since $n \ge s^{80t} \cdot \tilde g_0(t)$ for large enough function $\tilde g_0(t)$, we have $n \ge k^{20t}$. Thus, we can apply Theorem~\ref{theorem: product case - small k uniformity}, and obtain Theorem~\ref{theorem: product case - k-th power}.
\end{itemize}
\end{proof}

\section{Proof of Theorem~\ref{theorem: large uniformity sunflowers}}
\label{section: proof of theorem for large uniformity}

We start by presenting the proof sketch. First, using Lemma~\ref{lemma: spread approximation} with $\tau = e^{1/8st}$, we find a family $\cS \subset \partial_{\le q} \cA$, such that
\begin{align*}
    |\ff \setminus \cA[\cS]| \le 32 \left ( e^{1/(8st)}\right )^{-q} |\cA| \le r^{-\eta t} A_t.
\end{align*}
We prove that $\cS$ admits two following properties:
\begin{itemize}
    \item it does not contain a sunflower with $s$ petals and the core of size at most $t - 1$;
    \item it consists of sets of cardinality at least $t$.
\end{itemize}
This is done in Lemma~\ref{lemma: sunflower spread approximation}.

Next, using Lemma~\ref{lemma: sunflower simplification}, we obtain a simplification of $\cS$, i.e. a family $\cT$ that possesses the same properties as $\cS$, consists of sets of size exactly $t$ and such that
\begin{align*}
    |\cA[\cS] \setminus \cA[\cT]| \le \phi(s, t) \cdot \frac{2\varepsilon}{1 - \varepsilon} A_t \overset{{\color{teal} (\varepsilon \le 0.5)}}{\le}  \phi(s, t) \cdot 4 \varepsilon A_t.
\end{align*}

While $\cS$ does not contain a sunflower with the core of size at most $t - 1$,  $\cT$ does not contain a sunflower with an arbitrary core. Thus, the theorem follows.

\subsection{Spread approximation}
\label{subsection: spread approximation}

\begin{lemma}
    \label{lemma: sunflower spread approximation}
    Let $\ff, \cA, q$ be from the statement of Theorem~\ref{theorem: large uniformity sunflowers}. Then, there exists a family $\cS$ such that
    \begin{enumerate}
        \item \label{property: sunflower spread approximation, no sunflower <t core} 
            it does not contain a sunflower with $s$ petals and the core of size at most $t - 1$;
        \item \label{property: size of sets at least t}
            it consists of sets of size at least $t$;
        \item \label{property: remainder term}
            the following inequality holds:
            \begin{align*}
                |\ff \setminus \cA[\cS]| \le 32 \cdot r^{-\eta t} A_t.
            \end{align*}
    \end{enumerate}
\end{lemma}

Note that property~\ref{property: S does not contain a set of size less than t} follows from property~\ref{property: S is s, t sunflower avoiding}, but we state it explicitly. Indeed, as we have already mentioned, $s$ copies of a set of size at most $t - 1$ form a sunflower with $s$ empty petals and the core of size at most $t - 1$.

\begin{proof}[Proof of Lemma~\ref{lemma: sunflower spread approximation}]
    For our purposes, we modify Lemma~\ref{lemma: spread approximation} a bit. 
    \begin{lemma}
        \label{lemma: spread approximation with large measure}
        Fix an integer $q$ and a real number $\tau > 1$. Let $\cA$ be a $k$-uniform $r$-spread family. Let $\ff \subset \cA$. Then exists a family $\cS \subset \partial_{\le q} \cA$ and a partition of $\ff$ into families $\cR$ and $\ff_S \subset \cA(S), S \in \cS$:
        \begin{align*}
            \ff = \cR \sqcup \bigsqcup_{S \in \cS} \ff_S \vee  \{S\},
        \end{align*}
        such that each $\ff_S$ is $\tau$-homogeneous as a subfamily of $\cA(S)$, $\mu(\ff_S) = \frac{|\ff_S|}{|\cA(S)|} \ge 32 \tau^{-q}$ and $|\cR| \le 32 \tau^{- q} |\cA|$.
    \end{lemma}
    \begin{proof}
        We construct the spread approximation $\cS$ iteratively. Set $\ff_0 = \ff$ and $\cS = \varnothing$. Given a family $\ff_i$, we construct a family $\ff_{i + 1}$ as follows. Let $S_i$ be the maximal set such that $\mu(\ff_i(S_i)) > \tau^{|S_i|} \mu(\ff_i)$ holds. If $|S_i| \ge q + 1$ or $\mu(\ff_i(S_i)) < 32 \tau^{-q}$, we stop and put $\cR = \ff_i$. Otherwise, we add $S_i$ to $\cS$, set $\ff_{S_i} = \ff_i(S_i)$ and put $\ff_{i + 1} = \ff_i \setminus \ff_i[S_i]$. Note that $\ff_{S_i}$ is $\tau$-homogeneous by Observation~\ref{observation: homogeneous restriction}.

        Suppose that the procedure stops at step $m$. Then, we have $|S_m| \ge q + 1$ or $\mu(\ff_m(S_m)) \le 32 \tau^{-q}$. In the former case, we have
        \begin{align*}
            \mu(\cR) \le \tau^{- (q + 1)} \mu(\ff_m(S_m)) \le 32 \tau^{-q}.
        \end{align*}
        In the later case, we have
        \begin{align*}
            \mu(\cR) \le \tau^{-|S_m|} \mu(\ff_m(S_m)) \le 32 \tau^{-q}.
        \end{align*}
        In either case, we obtain the desired bound on $\mu(\cR)$. This completes the proof.
    \end{proof}

    Applying Lemma~\ref{lemma: spread approximation with large measure} with $\tau = e^{1 / (8 s t)}$ and $q = \lceil 16 \eta s t^2 \ln r\rceil $, we obtain a family $\cS \subset \partial_{\le q} \cA$ and families $\ff_S$ such that
    \begin{align}
        \label{eq: sunflower spread approximation lemma, decomposition}
        \ff = \cR \sqcup \bigsqcup_{S \in \cS} \ff_S \vee \{S \},
    \end{align}
    and $\ff_S$ are $\tau$-homogeneous. 

    First, we have
    \begin{align*}
        |\cR| & \le 32 \tau^{-q} |\cA| \le 32 e^{-q/8 st} |\cA| \le 32 \exp \left \{-\frac{16\eta s t^2 \ln r}{8 st } \right \} |\cA| \\
        & = 32 r^{-2 \eta t} |\cA| \overset{\text{Assumption~\ref{assumption: spreadness - small r}}}{\le} 32 r^{-\eta t} A_t.
    \end{align*}
    Since each $\ff_S$ is a subset of $\cA(S)$, it implies that $|\ff \setminus \cA[S]| \le |\cR| \le 32 r^{- \eta t} A_t$, and property~\ref{property: remainder term} holds.

    Next, we  check property~\ref{property: sunflower spread approximation, no sunflower <t core}. Suppose that $S_1, \ldots, S_s \in \cS$ form a sunflower with the core of size $t - 1 - h$, where $h \le t - 1$. Each family $\ff_{S_i}$ is $\tau$-homogeneous. Define $\cG_k = \ff_{S_k} (\varnothing, \bigcup_{i = 1}^s S_i)$.  According to Claim~\ref{claim: removing intersections}, they are $\tilde \tau$-homogeneous for 
    \begin{align*}
        \tilde{\tau} & = \tau / \left ( 1 - \frac{\tau}{r} \sum_{i = 1}^s |S_i| \right ) \le \tau / \left ( 1 - \frac{s q \tau}{r}\right )  \\
        & = e^{1/8st}/ \left ( 1 - \frac{s q \tau}{r}\right ) \overset{{\color{teal} (\tau \le 2)}}{\le} \exp \left ( \frac{1}{8 s t} + \frac{4 s q}{r} \right ) \le e^{1/(4 st)},
    \end{align*}
    where we use $- \log (1 - x) \le 2 x$ for any $x < 0.5$ and $r \ge 2s^2 q^3 \ge 32 s^2 t q$. Moreover, we have
    \begin{align}
    \label{eq: proof of large uniformity result - cG cardinality lower bound}
        |\cG_i| \ge \left (1 - \frac{\tau |\cup_i S_i|}{r} \right ) |\ff_{S_i}| \ge \left (1 - \frac{s \tau q}{r} \right ) |\ff_{S_i}|\ge \frac 12 |\mathcal F_{S_i}| \ge 16 \tau^{-q} \cdot |\cA(S_i)|,
    \end{align}
    where we used the property $\mu(\ff_{S_i}) \ge 32 \tau^{-q}$ guaranteed by Lemma~\ref{lemma: spread approximation with large measure}.
    
    Consider two cases.
    
    \textbf{Case 1. We have $\mathbf{h = 0}$.} Consider two cases again. 
    
    \textbf{Case 1a. Condition~\ref{condition: large spreadness} of Theorem~\ref{theorem: large uniformity sunflowers} holds.} Subfamilies $\cA(S)$ are $r$-spread, so $\cG_i \subset \cA(S_i), i = 1, \ldots, s,$ are $r / \tilde \tau$-spread. Since $r \ge 2^{8} s \lceil \log_2 k \rceil \ge \tilde{\tau} \cdot 2^{7} s  \lceil \log_2 k \rceil$, Proposition~\ref{proposition: coloring trick} ensures us that there are $s$ disjoint sets $F_1 \in \cG_1, \ldots, F_s \in \cG_s$ that do not intersect $\bigcup_{i = 1}^s S_i$ by the definition of $\cG_i$. From decomposition~\eqref{eq: sunflower spread approximation lemma, decomposition}, $\ff$ contains a sunflower $S_1 \cup F_1, \ldots, S_s \cup F_s$ with the core of size exactly $t-1$, a contradiction.

    \textbf{Case 1b. Condition~\ref{condition: sharp threshold applicability} of Theorem~\ref{theorem: large uniformity sunflowers} holds.}  To prove that families $\cG_i$, $i = 1, \ldots, s$, contain disjoint sets $F_i \in \cG_i$, we are going to use Lemma~\ref{lemma: matchings in global families}. First, we check that conditions of the lemma are satisfied.

    \begin{claim}
    \label{claim: satisfied conditions of matching lemma}
        Suppose that $n / k \ge 2^{390}s^{43} t^{82} + 2^{23000}$. Fix numbers $n' \ge n/2$, $k_i \le k$, $i = 1, \ldots, s$ and $\tau \le 16$. Consider arbitrary $\tau$-homogeneous families $\ff_1 \subset \binom{[n']}{k_1}, \ldots, \ff_s \subset \binom{[n']}{k_s}$, $\mu(\ff_i) \ge 8 \tau^{-q}$. Then, conditions of Lemma~\ref{lemma: matchings in global families} hold with $z = q$ and $n'$ in place of $n$.
    \end{claim}
    \begin{proof}
        In terms of Lemma~\ref{lemma: matchings in global families}, we have $\tau \le 32$. Therefore, it holds that
    \begin{align*}
        (2^6 \tau)^{m} \le 2^{11 (1 + \log_{4/3}(5z))} = 2^{11}\cdot 2^{11 \log_{4/3}(5 q)} = 2^{11} (5 q)^{\frac{11}{\log_2 (4/3)}} \le 2^{11} (5q)^{39} \le 2^{128} q^{39},
    \end{align*}
    where $m = \left \lceil \log_{4/3} \left ( \frac{z \log \tau}{\log 2} \right ) \right \rceil$ by the condition of Lemma~\ref{lemma: matchings in global families}.
    Thus, it is enough to ensure that
    \begin{align*}
        s q 2^{128} q^{39} \cdot 32 p_i = 2^{133} q^{40}sp_i  < 1/64,
    \end{align*}
    where $p_i = n'/k_i \ge n/2k$.
        Since $q = 2^5 s t^2 \ln \frac{n}{k} \le 2^6 s t^2 \ln \frac{1}{p_i}$, the above is implied by $2^{139} s (2^6 s t^2)^{40} p_i \cdot \ln^{40} \frac{1}{p_i} < 1$. Numerically, we checked that provided $p_i \le 2^{-22999}$, we have $\ln^{40} \frac{1}{p_i} \le p_i^{-1/41}$. Therefore, it is enough to ensure $2^{379} s^{41} t^{80} p_i^{40/41} < 1$ and $p_i \le 2^{-22999}$, or $n' \ge 2^{389} t^{82} s^{43} k$ and $n' \ge 2^{22999} k$. Both inequalities are satisfied, since $n' \ge n/2$ and $n / k \ge 2^{390} s^{43} t^{82} + 2^{23000}$.
    \end{proof}

    It is easy to check that $\cG_1, \ldots, \cG_s$ meet the conditions of Claim~\ref{claim: satisfied conditions of matching lemma} with $n' = n - |\cup_i S_i|$ and $k_i = k - |S_i|$. Thus, there are disjoint sets $G_1 \in \cG_1, \ldots, G_s \in \cG_s$, and so $\ff$ contains a sunflower $G_1 \cup S_1, \ldots, G_s \cup S_s$ with the core of size exactly $t - 1$, a contradiction.
    
    \textbf{Case 2. We have $\mathbf{h > 0}$.} Set $\alpha = 1/2s$ and $\rho = 1/2$. Recall that $\cG_1, \ldots, \cG_s$ are $\tilde \tau$-homogeneous with $\tilde \tau = e^{1/4st}$. We have
    \begin{align*}
        \tilde{\tau} \le \left (1 - \frac{1}{4 s} \right )^{- 1/ h} =\left ( 1 - \alpha \rho \right )^{-1/h},
    \end{align*}
    thus, due to Proposition~\ref{proposition: most of subsets are homogeneous}, for each $S \in \cS$ and for all but $\alpha |\partial_{h} (\cA(S))|$ sets $H$ families $\cG_i(H)$ are $2 \tilde{\tau}$-homogeneous and $\mu(\cG_i(H)) \ge \mu(\cG_i) / 2  \ge 8 \cdot (2 \tilde{\tau})^{-q}$, where we used the lower bound~\eqref{eq: proof of large uniformity result - cG cardinality lower bound}. Thus, for a random $\bH$ uniformly distributed on $\partial_{h} \cA$, we have
    \begin{align*}
        & \PP \left (\cG_i(\bH) \text{ is $2 \tilde\tau$-homogeneous and } \mu(\cG_i(\bH)) \ge 8 \cdot (2 \tilde \tau)^{-q} \; | \; \bH \in \partial_{h} (\cA(S)) \right )  \ge 1 - \alpha, \\
        & \PP \left (\cG_i(\bH) \text{ is $2 \tilde\tau$-homogeneous and } \mu(\cG_i(\bH)) \ge 8 \cdot (2 \tilde \tau)^{-q} \right ) \ge (1 - \alpha ) \PP \left [ \bH \in \partial_{h} (\cA(S)) \right ],
    \end{align*}
    where $\mu(\cG_i(\bH))$ is defined w.r.t. $\cA(S_i \cup \bH, \cup_i S_i \cup \bH)$, i.e. $\mu(\cG_i(\bH)) = \frac{|\cG_i(\bH)|}{|\cA(S_i \cup \bH, \cup_i S_i \cup \bH)|}$.
    Due to Assumption~\ref{assumption: weak shadow consistency - small r} and the condition $\mu k \ge 2 st q$, we have
    \begin{align*}
        \PP \left [ \bH  \in \partial_{h} (\cA(S)) \right ] \ge \left (1 - \frac{q}{\mu k} \right)^h \ge 1 - \frac{1}{2 s}.
    \end{align*}
    Hence, we have
    \begin{align*}
        \PP \left (\cG_i(\bH) \text{ is $2 \tilde\tau$-homogeneous and } \mu(\cG_i(\bH)) \ge 8 \cdot (2 \tilde \tau)^{-q} \right ) & \ge (1 - \alpha ) (1 - \frac{1}{2 s}) \\
        & > 1 - \alpha - \frac{1}{2 s} \\
        & = 1 - \frac{1}{s}, 
    \end{align*}
    which by the union bound implies
    \begin{align*}
        \PP \left ( \forall i \in [s] \; \cG_i(\bH) \text{ is $2 \tilde\tau$-homogeneous and } \mu(\cG_i(\bH)) \ge 8 \cdot (2 \tilde \tau)^{-q} \right ) & > 0.
    \end{align*}
Here we have $2 \tilde \tau \le 4$. Therefore, there exists a set $H \in \partial_h \cA$ such that all $\cG_i(H)$, $i = 1, \ldots, s,$ are $4$-homogeneous and $\mu(\cG_i(H)) \ge 8 \cdot 4^{-q}$. Again, we consider two different cases.

\textbf{Case 2a. Condition~\ref{condition: large spreadness} of Theorem~\ref{theorem: large uniformity sunflowers} holds.} Since each $\cG_i(H)$ is $4$-homogeneous, they are $r/4$-spread due to Assumption~\ref{assumption: spreadness - small r}.  Proposition~\ref{proposition: coloring trick} implies that there exist disjoint sets $F_1 \in \cG_i(H), \ldots, F_s \in \cG_i(H)$. These sets do not intersect $H \cup \bigcup_{i = 1}^s S_i$, so decomposition~\eqref{eq: sunflower spread approximation lemma, decomposition} implies that $\ff$ contains a sunflower $H \cup S_1 \cup F_1, \ldots, H \cup S_s \cup F_s$ with $s$ petals and the core $H \cup \bigcap_{i = 1}^s S_i$ of size exactly $t - 1$, a contradiction.

\textbf{Case 2b. Condition~\ref{condition: sharp threshold applicability} of Theorem~\ref{theorem: large uniformity sunflowers} holds.} Since $\cG_i(H)$ are 4-homogeneous and $\mu(\cG_i(H)) \ge 8 \cdot 4^{-q}$, they meet the conditions of Claim~\ref{claim: satisfied conditions of matching lemma} with $n' = n - |\cup_i S_i| - |H|$ and $k_i = k - |S_i| - |H|$, and, therefore, there are disjoint sets $G_1 \in \cG_1(H), \ldots, G_s \in \cG_s(H)$. Thus, $\ff$ contains a sunflower $S_1 \cup H \cup G_1, \ldots, S_s \cup H \cup G_s$ with the core of size exactly $t - 1$, a contradiction.
\end{proof}


\subsection{Proof of Theorem~\ref{theorem: large uniformity sunflowers}}

\begin{proof}[Proof of Theorem~\ref{theorem: large uniformity sunflowers}]
Combining Lemma~\ref{lemma: sunflower spread approximation} and Lemma~\ref{lemma: sunflower simplification}, we get
\begin{align*}
    |\ff \setminus \cA[\cT]| & \le |\ff \setminus \cA[\cS]| + |\cA[\cS] \setminus \cA[\cT]| \\ 
    & \le 32 \cdot r^{-\eta t} A_t + \phi(s, t) \cdot \frac{2\varepsilon}{1 - \varepsilon} A_t \cdot \\
    & \overset{{\color{teal} (\varepsilon \le 0.5)}}{\le} 32 \cdot r^{-\eta t} A_t + \phi(s, t) \cdot 4 \varepsilon A_t \cdot \qedhere
\end{align*}
\end{proof}

\section{Proof of Theorem~\ref{theorem: large r theorem}}
\label{section: proof for small k theorem}


We divide the proof into two parts, the first one, Section~\ref{section: intersection reduction}, can be applied for $\cA$ that satisfies Assumptions~\ref{assumption: spreadness - small r},\ref{assumption: regularity}, and the second one (Sections~\ref{section: simplification argument ii}-\ref{subsection: proof of large r thereom}) is specific for $\cA = \binom{[n]}{k/w}^w$, where $w$ is an arbitrary integer such that $w | k$.

The reason why the case of small $k$ requires a different approach is that it is much harder, if even possible, to obtain stability results, i.e., find a family $\mathcal T$ of uniformity $t$ and with no $s$-sunflower, such that most $\ff$ is contained in $\mathcal A[\mathcal T]$. In Section~\ref{section: proof sketch for product case -- small k uniformity}, we will see that a direct approach based on the Kruskal--Katona theorem gives the desired result up to a remainder that depends on $k$ instead of $n$. Thus, there is little chance to obtain a result like Theorem~\ref{theorem: large uniformity sunflowers} by our methods in general. However,  for $k \ge 5t$, we can get some structural information on the extremal example. As we saw in Section~\ref{section: delta-system method}, the case $k < 5t$ can be resolved by Theorem~\ref{theorem: delta-system solution}.

Section~\ref{section: intersection reduction} develops tools to obtain the aforementioned structural results. Applying them with a number of tricks, we show in Section~\ref{section: simplification argument ii} that there exists a family $\cC^* \in \partial_t \cA$, such that $|\cC^*| \le 70 t 2^t \phi(s, t) \sqrt{n} \log n$ and
\begin{align*}
    |\ff \setminus \ff[\cC^*]| \le \frac{C_t s^t\log n}{\sqrt{n}} \cdot A_t,
\end{align*}
where $C_t$ is some constant depending on $t$ only.

The rest of the proof is pretty simple and simlar to what Frankl and F\"uredi did in~\cite{Frankl1987}. One can apply to each $\ff(C)$, $C \in \cC^*$ a technique similar 
 to Lemma~\ref{lemma: delta-system method}, and then repeat the proof of Theorem~\ref{theorem: delta-system solution}.

\subsection{Intersection reduction}
\label{section: intersection reduction}

When $k$ is much smaller than $\log r$, we cannot obtain $\cS$ with properties as strong as in Lemma~\ref{lemma: spread approximation}. Fortunately, $\cS$ inherits restrictions imposed on $\ff$, namely, it does not contain a sunflower with $s$ petals and the core of size $t - 1$. Additionally, all sets of $\cS$ has cardinality at least $t$.

\begin{proposition}
    \label{proposition: spread approximation, small uniformity case}
    Suppose that Assumptions~\ref{assumption: spreadness - small r},\ref{assumption: regularity} hold for some $q$ and $r$. Let $\cS \subset \partial_{\le q} \cA$ be a $\tau$-homogeneous {\Sname} of a family $\ff \subset \cA$, $\ff$ does not contain a sunflower with $s$ petals and the core of size $t - 1$. Assume that $r > 2^{10} \tau^t s \lceil \log_2 k \rceil$ and $r \ge 2 q s \tau$. Then, it holds that
    \begin{enumerate}
        \item \label{property: S is s, t sunflower avoiding} $\cS$ does not a sunflower with $s$ petals and the core of size $t - 1$;
        \item \label{property: S does not contain a set of size less than t} each $S \in \cS$ has cardinality at least $t$. 
    \end{enumerate}
\end{proposition}

\begin{proof}
    Suppose that $\cS$ contains a sunflower $S_1, \ldots, S_s$ with the core of size $t - 1$. For each $\ff_{S_i}$, define $\cG_i = \ff_{S_i}\left (\varnothing, \bigcup_{p = 1}^s S_p \right )$. Since $\ff_{S_i}$ is $\tau$-homogeneous, $\cG_i$ is $\tau'$-homogeneous for
    \begin{align*}
        \tau' = \frac{\tau}{1 - \frac{q s \tau}{r}} \le 2 \tau
    \end{align*}
    due to Claim~\ref{claim: removing intersections}.
    Hence, $\cG_i, i \in [s],$ are $R$-spread for
    \begin{align*}
        R = \frac{r}{2 \tau} > 2^{7} s \lceil \log_2 k \rceil.
    \end{align*}
    Proposition~\ref{proposition: coloring trick} guarantees that there exist disjoint $G_1 \in \cG_1, \ldots, G_s \in \cG_s$. By the definition of $\cG_i$, they do not intersect $\bigcup_{p = 1}^s S_p$. Hence, $S_1 \cup G_1, \ldots, S_s \cup G_s$ form a sunflower with the core of size $t - 1$, a contradiction. That proves property~\ref{property: S is s, t sunflower avoiding}.

    Next, assume that there exists $S \in \cS$ of size at most $t - 1$. Set $h = t - 1 - |S|$ and consider two cases.

    \textbf{Case 1. We have $\mathbf{h = 0}$.} Since $\ff_S$ is $\tau$-homogeneous, it is $R$-spread for $R = r / \tau > 2^{7} s \lceil \log_2 k \rceil$. Due to Proposition~\ref{proposition: coloring trick}, it contains a matching $F_1, \ldots, F_s$. Thus, $\ff$ contains a sunflower $S \cup F_1, \ldots, S \cup F_s$, a contradiction.

    \textbf{Case 2. We have $\mathbf{h > 0}$.}  We are going to apply Proposition~\ref{proposition: most of subsets are homogeneous} with $\alpha= \rho = 1 - \tau^{-(t - 1)} / 2$. Let us check that the conditions of the proposition hold. We have
    \begin{align*}
        (1 - \alpha \rho)^{-1/h} = (1 - (1 - \tau^{t - 1}/2)^2)^{-1/h} = (\tau^{-(t - 1)} - \tau^{-2t - 2}/4)^{-1/h} \ge \tau^{(t - 1)/h} \ge \tau,
    \end{align*}
    and so there exist $H \in \partial_h \cA(S)$ such that $\ff_S(H)$ is $2 \tau^{t}$-homogeneous. In particular, it is $R$-spread for
    \begin{align*}
        R = \frac{r}{2 \tau^t} > 2^{7} s \lceil \log_2 k \rceil.
    \end{align*}
    Thus, $\ff_S(H)$ contains $s$ disjoint sets $F_1, \ldots, F_s$, and $F_1 \cup S \cup H, \ldots, F_s \cup S \cup H$ belong to $\ff$ and form a sunflower with the core of size $t - 1$, a contradiction.

    Thus, property~\ref{property: S does not contain a set of size less than t} holds.
\end{proof}

We can impose some additional restrictions on $\cS$. The key element for doing that is the following simple lemma.

\begin{lemma}
\label{lemma: homogeneous subfamily}
    Let $\ff \subset \cA$ be $\tau$-homogeneous family. Then, for arbitrary $\alpha \le 1/2k$, there exists a subfamily $\cG \subset \ff$, such that
    \begin{itemize}
        \item $|\cG| \ge (1 - 2 \alpha k ) |\ff|$;
        \item for any $P$ of size at most $t - 1$, if $\cG(P)$ is not empty, then $\ff(P)$ is $\alpha (\tau / \alpha)^t$-homogeneous.
    \end{itemize}
\end{lemma}

\begin{proof}
    Define
    \begin{align*}
        \cP = \left \{P \in \partial_{\le t - 1}\cA \mid \mu(\ff(P)) < \alpha^{|P|} \mu(\ff) \right \}.
    \end{align*}
    Then, let $\cG = \ff \setminus \ff[\cP]$. We have
    \begin{align*}
        |\cG| \ge |\ff| - \sum_{P \in \cP} \alpha^{|P|} \frac{|\cA(P)|}{|\cA|} |\ff| \ge |\ff| \left (1 - \sum_{p = 1}^{t - 1} \alpha^p \binom{k}{p} \right ).
    \end{align*}
    Since 
    \begin{align*}
        \sum_{p = 1}^{t - 1} \alpha^p \binom{k}{p} \le \sum_{p = 1}^{t - 1} (\alpha k)^{p} \le \frac{\alpha k}{1 - \alpha k} \le 2 \alpha k,
    \end{align*}
    we have $|\cG| \ge (1 - 2 \alpha k)|\ff|$.
    
    Clearly, if $\cG(P)$ is not empty, then $\mu(\ff(P)) \ge \alpha^{|P|} \mu(\ff)$. Then, it is $\alpha (\tau / \alpha)^{t}$-homogeneous, since overwise we get
    \begin{align*}
        \mu(\ff) \tau^{|P| + |S|} \ge \mu(\ff(P\cup S)) > \alpha^{|S|} (\tau / \alpha)^{t |S|} \mu(\ff(P)) \ge \alpha^{|P| + |S|} (\tau / \alpha)^{t |S|} \mu(\ff), \\
        (\tau / \alpha)^{|P|} > (\tau / \alpha)^{(t - 1)|S|}
    \end{align*}
    for some non-empty $S$ and $P$ of size $\le t - 1$, a contradiction.
\end{proof}

The idea is to apply Lemma~\ref{lemma: homogeneous subfamily} to each $\ff_S, S \in \cS$, and obtain a large and well-structured subfamily of $\ff$. The established property is somewhat similar to the one given in Lemma~\ref{lemma: spread approximation}, but much weaker.

\begin{lemma}
    \label{lemma: intersection graph structure}
Suppose that Assumptions~\ref{assumption: spreadness - small r},\ref{assumption: regularity} hold for some $q$ and $r$, $k \ge q + (t - 1)$. Let $\ff$ be a family that does not contain a sunflower with $s$ petals and the core of size $t - 1$. Let $\cS \subset \partial_{\le q} \cA$ be its $\tau$-homogeneous {\Sname}.  Assume that $r > \max \{ 2^{10} \lceil \log_2 k \rceil, 2 q\} \cdot s \alpha^{1 - t} \tau^t$ for some $\alpha \in (0, (4 k)^{-1}]$. Then there exist families $\cU_S \subset \ff_S, S \in \cS$, such that
    \begin{enumerate}
        \item $|\cU_S| \ge (1 - 2 \alpha k) |\ff_S|$ for each $S \in \mathcal{S}$,
        \item $\cS$ does not contain a sunflower with $s$ petals and the core of size $t - 1$; if $S_1, \ldots, S_s \in \mathcal{S}$ form a sunflower with $s$ petals and the core of size at most $t - 2$, then for any $F_1 \in \cU_{S_1}, \ldots, F_s \in \cU_{S_s}$ we have $|\bigcap_{p = 1}^s F_p| \le t - |\bigcap_{p = 1}^s S_s| - 2$,
        \item let $\tau' \le \tau$ be the minimal homogeneity of $\ff_S$, then $|\partial_{h} (\cU_S)| \ge \left ( \frac{\tau'}{1 - 2 \alpha k} \right )^{-h}|\partial_{h} \cA(S)|$ for any $h < k$.
    \end{enumerate}
\end{lemma}

\begin{proof}
    We have
    \begin{align*}
        \ff = \bigsqcup_{S \in \cS} \ff_S \vee \{S\},
    \end{align*}
    and for each $S \in \mathcal{S}$, $\ff_S$ is $\tau$-homogeneous. Define $\cU_S$ as a family $\mathcal G$ obtained via Lemma~\ref{lemma: homogeneous subfamily} for $\ff_S$. If $\tau'$ is the minimal homogeneity of $\ff_S$,  then $\cU_S$ is $\tau' / (1 - 2 \alpha k)$-homogeneous, and so $\mu(\partial_h \cU_S) \ge \left ( \frac{\tau'}{1 - 2 \alpha k} \right )^{-h}$ for any $h$ due to Proposition~\ref{proposition: shadow of tau-homogeneous family}.

    Due to Proposition~\ref{proposition: spread approximation, small uniformity case}, $\cS$ does  not contain a sunflower with $s$ petals and the core of size $t - 1$. Consider $s$ sets $S_1, \ldots, S_s \in \mathcal{S}$ that form a sunflower with the core $C$ of size at most $t - 1$. Suppose that there exist $F_i \in \cU_{S_i},i \in [s]$ such that $\left |\bigcap_{p = 1}^s F_p \right | \ge  t - |C| - 1$. Choose $P \subset \bigcap_{p = 1}^s F_p$ of size exactly $t - |C| - 1$. Lemma~\ref{lemma: homogeneous subfamily} implies that $\ff_{S_i}(P), i \in [s],$ are $\alpha^{1 - t} \tau^t$-homogeneous. Note that we need the assumption $k \ge q + (t - 1)$ to ensure that $\ff_{S_i}(P)$ is not empty. Due to Claim~\ref{claim: removing intersections}, $\ff_{S_i} \left (P, P \cup \bigcup_{p = 1}^s S_p \right )$ is $\tau'$-homogeneous for
    \begin{align*}
        \tau' = \frac{\alpha^{1 - t} \tau^t}{1 - \frac{s q \alpha^{1 - t} \tau^t}{r}} \le 2 \alpha^{1 - t} \tau^t.
    \end{align*}
    Thus, each $\cG_i = \ff_{S_i}\left (P, P \cup \bigcup_{p = 1}^s S_p \right )$ is $R$-spread, for
    \begin{align*}
        R \ge \frac{r}{2 \tau^t} \alpha^{t - 1} > 2^{7} s \lceil \log_2 k\rceil .
    \end{align*}
    Proposition~\ref{proposition: coloring trick} implies that there are $s$ disjoint sets $G_1 \in \cG_1, \ldots, G_s \in \cG_s$. By the definition of $\cG_i$, they do not intersect $P \cup \bigcup_{p = 1}^s S_p$, hence, sets $P \cup S_1 \cup G_1, \ldots, P \cup S_s \cup G_s$ are contained in $\ff$ and form a sunflower with $s$ petals and the core $C \cup P$ of size $t - 1$, a contradiction.
\end{proof}

\begin{definition}
\label{definition: Sst system}
    Let $\mathcal{S}\subset \bigcup_{p = t}^q \partial_p \cA$ be a family that does not contain a sunflower with $s$ petals and the core of size $t - 1$. We say that a family $\cB \subset \cA$ is {\em an $(\mathcal{S}, s, t)$-system}, if $\cB$ can be decomposed into disjoint families $\cB_S \vee \{S\}$, $S\in \mathcal S$ and  $\cB_S \subset \cA(S)$, such that for any $s$ sets $S_1, \ldots, S_s \in \mathcal{S}$ the following properties hold.
    \begin{itemize}
        \item $S_1,\ldots, S_s$  do not form a sunflower with the core of size exactly $t - 1$.
        \item if $S_1,\ldots, S_s$ form a sunflower with the core of size at most $t - 2$, then for any $F_i \in \cB_{S_i}, i \in [s]$ we have $\left |\bigcap_{i = 1}^s F_i \right | \le t - |S_1 \cap S_2| - 2$.
    \end{itemize}
\end{definition}

Given a $\tau$-homogeneous {\Sname} $\cS$ of a family $\ff$ without a sunflower with $s$ petals and the core of size $t - 1$, define
\begin{align*}
    \cU = \bigsqcup_{S \in \cS} \cU_S \vee \{S\},
\end{align*}
where $\cU_S, S \in \cS,$ are families obtained from Lemma~\ref{lemma: intersection graph structure}. Then, Proposition~\ref{proposition: spread approximation, small uniformity case} guarantees that the family $\cS$ does not contain a sunflower with $s$ petals and the core of size $t - 1$, and each its set has size at least $t$. Applying Lemma~\ref{lemma: intersection graph structure}, we infer that $\cU$ is an $(\cS, s, t)$-system.

We shall use reduction to $(\cS, s, t)$-systems  to prove that $\ff$ admits some nice decomposition up to a polynomial-size remainder.
\begin{proposition}
\label{proposition: clustering of (S, s, t)-system into t-sets}
    Let $\cA$ be a $k$-uniform $(r, t)$-spread family and $\cU$ be an $(\cS, s, t)$-system such that $|\partial_{t - 1} \cU_S| \ge \lambda |\partial_{t - 1} \cA|$ for every $S \in \cS \subset \partial_{\le q} \cA$. Suppose that $r \ge 2 s^2 q^3$. Then, there exists a family $\hat{\cT} \subset \partial_{t} \cA$, such that
    \begin{align}
    \label{eq: number of clusters}
        |\hat{\cT}| \le \frac{1}{\lambda} \left (1 +  2 \phi(s, t) \ln(\lambda |\partial_{\le q} \cA|) \right )
    \end{align}
    and
    \begin{align}
    \label{eq: size of the reminder of clustering}
        \sum_{S \in \cS \setminus \cS[\hat{\cT}]} |\cU_S| 
        & \le \frac{8\phi(s, t) s^2 q^3}{\lambda r}  \cdot  \ln (\lambda |\partial_{\le q} \cA|) \cdot A_t .
    \end{align}
\end{proposition}

\begin{proof}

    The set $\hat{\cT}$ is formed by the following procedure. Set $\cS^0 = \cS$. For a defined $\cS^i$, consider two cases. If $|\cS^i| \le \frac{1}{\lambda}$, then put $m: = i$ and stop. Otherwise, we claim that there exists $H_i \in \partial_{t - 1} \cA$, such that
    \begin{align*}
        \cS^i_{H_i} = \{S \in \cS^i \mid H_i \in \partial_{t - 1} \cU_S \}
    \end{align*}
    has size at least $\lambda |\cS^i|$. This follows from a double-counting argument. Consider a bipartite graph with vertices $\cS^i$ and $\partial_{t - 1} \cA$, and connect $S \in \cS^i$ and $H \in \partial_{t - 1} \cA$, if $H \in \partial_{t - 1} (\cU_S)$. The assumptions of the lemma imply that the number of edges is at least $\sum_{S \in \cS^i} |\partial_{t - 1} \cU_S| \ge \lambda |\partial_{t - 1} \cA| |\cS^i|$. Thus, there exists $H_i \in \partial_{t - 1} \cA$ that is connected to at least $\lambda |\partial_{t - 1} \cA| |\cS^i| / |\partial_{t - 1} \cA| = \lambda |\cS^i| > 1$ indices $S$ in $\cS^i$.

    By Definition~\ref{definition: Sst system} of an $(\cS^i, s, t)$-system,  the family $\cS^i_{H_i}$ does not contain a sunflower with $s$ petals and the core of size at most $t - 1$. Due to Lemma~\ref{lemma: sunflower simplification}, there exists a family $\cT_i$ of uniformity $t$ such that
    \begin{align}
    \label{eq: bound on the reminder of T_i}
       \sum_{S \in \cS^i_{H_i} \setminus \cS^i_{H_i}[\cT_i]} | \cU_S | \le |\cA[\cS^i_{H_i} \setminus \cS^i_{H_i}[\cT_i]]| \le \frac{4\phi(s,t) s^2 q^3}{r} A_t,
    \end{align}
    and $\cT_i$ does not contain a sunflower with $s$ petals. Define $\cS^{i + 1} = \cS^i \setminus \cS^i_{H_i}$ and repeat the procedure. 

    Once the procedure stops, put $\cT_m$ to be the smallest family of uniformity $t$ such that $\cS^m = \cS^m[\cT_m] $. Clearly, $|\cT_m| \le |\cS^m| \le \frac{1}{\lambda}$. Then, define
    \begin{align*}
        \hat{\cT} = \bigcup_{i = 1}^{m} \cT_i.
    \end{align*}
    
    Next, we bound the number of steps $m$. Since $|\cS^i_H| \ge \lambda |\cS^i|$, we have $|\cS^{i + 1}| \le (1 - \lambda) |\cS^i|$. 
    If $m > 0$ and  so $\lambda |\cS| > 1$, let $\tilde{m}$ be defined as the least integer such that $e^{-\lambda \tilde{m}} |\cS| \le \frac{1}{\lambda}$, that is $\tilde{m} = \lceil \frac{1}{\lambda} \ln (\lambda |\cS|) \rceil $. We claim that $m \le \tilde m$. Otherwise, we have  $\tilde m < m$ and, therefore, $\cS^{\tilde m}$ is defined, $|\cS^{\tilde m}| > \frac{1}{\lambda}$ and
    \begin{align*}
        |\cS^{\tilde m}| \le (1 - \lambda)^{\tilde m} |\cS| \le e^{-\lambda \tilde m} |\cS| \le \frac{1}{\lambda},
    \end{align*}
    a contradiction. Thus, we have $$m \le \frac{2}{\lambda} \ln(\lambda |\cS|) \indicator\{\lambda |\cS| > 1\} \le \frac{2}{\lambda} \ln (\lambda |\partial_{\le q} \cA|),$$ and $|\hat{\cT}| \le m \phi(s, t) + \frac{1}{\lambda} \le \frac{1}{\lambda} \left (1 +  2 \phi(s, t) \ln(\lambda |\partial_{\le q} \cA|) \right )$, which proves~\eqref{eq: number of clusters}.

    Finally, we bound the sum $\sum_{S \in \cS \setminus \cS[\hat{\cT}]} |\cU_S|$:
    \begin{align*}
        \sum_{S \in \cS \setminus \cS[\hat{\cT}]} |\cU_S| & \le  \sum_{i = 0}^{m - 1} \sum_{S \in \cS^{i}_H \setminus \cS^i_H[\cT_i]} |\cU_S|  \overset{\eqref{eq: bound on the reminder of T_i}}{\le} m \cdot \frac{4\phi(s, t) s^2 q^3}{r} A_t \\
        & \le \frac{8 s^2 q^3\ln(\lambda |\partial_{\le q} \cA|)}{\lambda r} \cdot \phi(s, t)  \cdot A_t. \qedhere
    \end{align*}
\end{proof}

\subsection{Simplification argument II}
\label{section: simplification argument ii}

The main result of this section is the following lemma:
\begin{lemma}
\label{lemma: small packing of F -- small uniformity case}
    Assume that $\cA = \binom{[n]}{k/w}^w$ for some $w | k$. Suppose that $n \ge  2^{30}(s k^2)^{4t + 2}$ and $k \ge 5 t$. Then exists a family $\cC^* \subset \partial_{t} \cA$, such that $|\cC^*| \le 70 t \cdot 2^t \phi(s, t) \sqrt{n} \ln n$ and
\begin{align*}
    |\ff \setminus \ff[\cC^*]| \le \frac{C_t s^t k \ln n}{\sqrt{n}} \cdot A_t,
\end{align*}
where $C_t$ is some constant depending on $t$ only.
\end{lemma}

We will construct $\cC^*$ via iterative simplification of $\ff$. First, we find a family $\cT$ that has small uniformity and covers $\ff$ almost completely. 

\begin{lemma}
\label{lemma: T packing without monotonicity}
    Assume that $\cA = \binom{[n]}{k/w}^w$ for some $w | k$. Suppose that $w n \ge 2 (s k^2)^{4t + 2}$ and $k \ge 2t + 1$. Let $\ff \subset \binom{[n]}{k/w}^w$ be a family without a sunflower with $s$ petals and the core of size $t - 1$. Then, there exists a family $\cT \subset \partial_{2t} \cA \cup \partial_{2t + 1} \cA$ such that $\cT$ does not contain a sunflower with $s$ petals and the core of size $t - 1$, and
    \begin{align*}
        |\ff \setminus \ff[\cT]| \le\frac{\tilde C_t s^{t - 1}}{\sqrt{n}} \cdot A_t,
    \end{align*}
    where $\tilde C_t$ is some constant depending on $t$ only.
\end{lemma}


\begin{proof}
    The simplification argument is a modification of that of Lemma~\ref{lemma: sunflower simplification}. We will obtain $\cT$ at the end of the following iterative process. Set $\cT_0 = \ff$. Given a family $\cT_{i}$, we will construct a family $\cT_{i + 1}$ by the following procedure:
    \begin{algorithmic}[1]
        \State Set $\cW_i \leftarrow \cT_i^{(k - i)}$, $\alpha = sk$, $\cU_{i + 1} \leftarrow \varnothing$;
        \While{there exists a set $T$, such that $|T| \le k - i - 1$ and $\cW_{i}(T)$ is $\alpha$-spread}
            \State Find a maximal set $T$ such that $|T| \le k - i - 1$ and $\cW_i(T)$ is $\alpha$-spread;
            \State Define $\cT_{i + 1, T} := \cW_{i}(T)$;
            \State Remove sets containing $T$ from $\cW_{i}$: $$\cW_{i} \leftarrow \cW_{i } \setminus \cW_{i}[T];$$
            \If{$|T| > 2t - 1$}
                \State Add $T$ to $\cT_{i + 1}$: $$\cT_{i + 1} \leftarrow \cT_{i + 1} \cup \{T\};$$
            \Else
                \State Add $T$ to $\cU_{i + 1}$: $$\cU_{i + 1} \leftarrow \cU_{i + 1} \cup \{T\};$$
            \EndIf
        \EndWhile
    \end{algorithmic}

    We define $\cT$ as $\cT_{k - (2t + 1)}$. Since any set from $\cT_i$ has cardinality between $k - i$ and $2t$ by the construction, $\cT$ consists of sets of size either $2t$ or $2t + 1$.

    First, we describe several properties of the obtained system.

    \begin{claim}
    \label{claim: absence of sunflower in simplification argument ii}
        The family $\bigcup_{i = 0}^{k - (2 t + 1)} (\cT_i \cup \cU_i)$ does not contain a sunflower with $s$ petals and the core of size $t - 1$.
    \end{claim}

    \begin{proof}
        Let $T_{1} \in \cT_{i_ 1} \cup \cU_{i_1}, \ldots, T_s \in \cT_{i_s} \cup \cU_{i_s}$ be such sunflower that $(i_1, \ldots, i_s)$ is minimal in the lexicographical order. If $i_1, \ldots, i_s = 0$, then $T_1, \ldots, T_s$ belongs to the initial family $\ff$, a contradiction. Thus, there is at least one $\ell \in [s], i_\ell > 0$. It means that for some $j< i_{\ell}$, $\cT_{i_\ell, T_{\ell}} \subset \cT_{j}(T_{\ell})$ is $\alpha$-spread. Due to Claim~\ref{claim: covering number}, the covering number of $\cT_{j}(T_{\ell})$ is at least $\alpha = sk$. Thus, we may find $F \in \cT_{j}(T_{i_\ell})$ that does not intersect $\bigcup_{p = 1}^s T_p$. Note that $T_1 \in \cT_{i_1} \cup \cU_{i_1}, \ldots, F \cup T_{\ell} \in \cT_{j}, \ldots, T_s \in \cT_{i_s} \cup \cU_{i_s}$ forms a sunflower with $s$ petals and the core of size $t - 1$, while $(i_1, \ldots, i_{\ell - 1}, j, i_{\ell + 1}, \ldots, i_s)$ is strictly less than $(i_1, \ldots, i_s)$ in the lexicographical order, contradicting with the minimality of the latter tuple.
    \end{proof}

    Thus, $\cT$ does not contain such a sunflower, as stated. The second property we establish allows us to bound the size of $\cW_{i}$ and $\cT_i[\cU_{i + 1}]$.
    \begin{claim}
        We have $|\cW_i| \le \alpha^{k - i}$ and
        \begin{align*}
            |\cT_i[\cU_{i + 1}]| \le \tilde{C}_t \alpha^{k - i} w^t n^t s^{t -1},
        \end{align*}
        where a constant $\tilde{C}_t$ depends on $t$ only.
    \end{claim}
    \begin{proof}
        In order to bound $|\cW_i|$, note that there is no set $T$ such that $|T| \le k - i - 1$ such that $\cW_i(T)$ is $\alpha$-spread. In this case, choose maximal $X$ such that $|\cW_i(X)| \ge \alpha^{|X|} |\cW_i|$. We have $|X| = k - i$, and so $|\cW_i| \le \alpha^{k - i}$.

        Next, we bound $|\cT_i[\cU_{i + 1}]|$. For some $j \in [2t - 1]$, consider the layer $\cU_{i + 1}^{(j)}$:
        \begin{align*}
            \cU^{(j)}_{i + 1} = \{F \in \cU_{i + 1} \mid |F| = j\}.
        \end{align*}
        Define the upper shadow $\partial^u \cU_{i + 1}^{(j)}$ of the family $\cU^{(j)}_{i + 1}$ as follows:
        \begin{align*}
            \partial^u \cU_{i + 1}^{(j)}  = \{F \in \partial_{j + 1} \cA \mid \exists F' \in \cU_{i + 1}^{(j)} \text{ such that } F' \subset F\}.
        \end{align*}
        Consider a set $T'$ that belongs to the upper shadow  and some set $T \in \cU_{i + 1}^{(j)}$ such that $T \subset T'$. Due to the third line of the procedure above, the family $\cT_{i + 1, T}(T' \setminus T)$ is not $\alpha$-spread. Choose a maximal $X$ disjoint from $T'$ such that $|\cT_{i + 1, T}((T' \setminus T) \cup X)| \ge \alpha^{|X|} |\cT_{i + 1, T}(T' \setminus T)|$. Since $\cT_{i + 1, T}(T' \setminus T)$ is not $\alpha$-spread, $X$ is not empty. It implies that $|X| = k - i - |T'|$, since otherwise, $\cT_{i+1, T}(T' \cup X)$ is $\alpha$-spread due to Observation~\ref{observation: R-spread restriction}, and we may choose $X \cup T'$ in the third line of the procedure instead of $T$. Therefore, $|\cT_{i + 1, T}(T' \setminus T)| \le \alpha^{k - i - (j+1)}$.

        Hence, we may bound $|\cT_i[\cU_{i+1}^{(j)}]|$. Indeed, it admits a decomposition
        \begin{align}
        \label{eq: Uj above bound}
            |\cT_i[\cU_{i + 1}^{(j)}]| = \sum_{T \in \cU^{(j)}_{i + 1}} |\cT_{i + 1, T}| \le \sum_{T \in \cU^{(j)}_{i+1}} \sum_{x \in \support(\cA) \setminus T} |\cT_{i + 1, T}(x)| \le \alpha^{k - i - (j + 1)} \cdot wn \cdot |\cU^{(j)}_{i + 1}|.
        \end{align}
        If $j < t$, then we can bound $|\cU^{(j)}_{i + 1}| \le \binom{wn}{t - 1} \le w^{t - 1} n^{t - 1}$. Otherwise, we use Theorem~\ref{theorem: bucic upper bound}. Claim~\ref{claim: absence of sunflower in simplification argument ii} ensures that $\cU^{(j)}_{i + 1}$ does not contain a sunflower with $s$ petals and the core of size $t - 1$. We consider $\cU^{(j)}_{i + 1}$ as a subfamily of $\binom{[wn]}{j}$,  and apply Theorem~\ref{theorem: bucic upper bound}, which implies $|\cU^{(j)}| \le C_{j} w^{t - 1} n^{t - 1} s^{j - t + 1}$.

        Thus, in either case, we have $|\cU^{(j)}_{i + 1}| \le C_t w^{t - 1} n^{t - 1} s^{t}$, since $C_t$ is increasing in $t$. Substituting this bound into~\eqref{eq: Uj above bound}, we get
        \begin{align*}
            |\cT_i[\cU_{i + 1}]| \le \alpha^{k - i - 1} (2t - 1) \cdot C_t n^t w^t s^{t} \sum_{j = 0}^{2t - 2} \alpha^{-j} \le \tilde{C}_t w^{t} \alpha^{k - i} n^t s^{t},
        \end{align*}
        where  we introduced a constant $\tilde C_t$ that depends on $t$ only.
    \end{proof}

    To bound $|\ff \setminus \ff[\cT]|$, note that for each $F \in \ff \setminus \ff[\cT]$ there exists $i \le k - (2t + 1)$ such that some subset $F' \subset F$ of size $k - i$ belongs to $\cT_i[\cU_{i + 1}] \cup \cW_i$. It yields
    \begin{align*}
        |\ff \setminus \ff[\cT]| & \le \sum_{i = 0}^{k - (2t + 1)} \left (|\cT_i[\cU_{i + 1}]| + |\cW_i| \right ) \cdot A_{k - i} \\
        & \le 2 \tilde{C}_t s^{t} \left ( \sum_{i = 0}^{k - (2t + 1)} \alpha^{k - i} w^t  n^t  \left (\frac{wn}{k} \right )^{-(k - i - t)} \right ) \cdot A_t,
    \end{align*}
    where we used $(wn/k, t)$-spreadness of $\cA=\binom{[n]}{k/w}^w$ due to Proposition~\ref{proposition: spread families examples}.
    Next, we bound the sum. We have
    \begin{align*}
        \sum_{i = 0}^{k - (2t + 1)} \alpha^{k - i} w^t  n^t \left ( \frac{wn}{k}\right )^{-(k - i - t)} \le n^{2t} w^{2t}\sum_{i = 0}^{k - (2t + 1)} \left ( \frac{\alpha k}{wn} \right )^{k - i} \le 2 n^{2t} w^{2t} \left ( \frac{\alpha k }{wn} \right )^{2t + 1}.
    \end{align*}
    Recall that $\alpha = sk$.  Since $w n \ge (\alpha k)^{4t + 2}$ by assumptions of the lemma, the above is at most $2 / \sqrt{n}$. Thus, we have
    \begin{align*}
        |\ff \setminus \ff[\cT]| \le \frac{4 \tilde{C}_t s^{t}}{\sqrt{n}} A_t. & \qedhere
    \end{align*}
\end{proof}

Now we ready to prove Lemma~\ref{lemma: small packing of F -- small uniformity case}.

\begin{proof}[Proof of Lemma~\ref{lemma: small packing of F -- small uniformity case}]

Take $\cT$ as in Lemma~\ref{lemma: T packing without monotonicity}. Consider an arbitrary partition of $\ff[\cT]$ into $|\cT|$ families $\ff_T \subset \ff(T)$, $T \in \cT$, so
\begin{align*}
    \ff[\cT] = \bigsqcup_{T \in \cT} \ff_T \vee \{T\}.
\end{align*}

Choose $\tau = n^{1/2t}$, $q = 2t$, and for each $T \in \cT$ apply Lemma~\ref{lemma: spread approximation} to $\ff_T$. We obtain a family $\cS_T$, a remainder $\cR_T \subset \cA[T]$ and a decomposition:
\begin{align*}
    \ff_T \setminus \cR_T = \bigsqcup_{S \in \cS_T} \ff_{T \sqcup S} \vee \{S\},
\end{align*}
such that each $\ff_{T \sqcup S}$ is $\tau$-homogeneous and $|\cR_T| \le n^{-1} A_{|T|}$. Then, we have
\begin{align*}
    \sum_{T \in \cT} |\cR_T| \le n^{-1} \sum_{j = 2t}^{2t + 1} \sum_{T \in \cT^{(j)}} A_{j}.
\end{align*}
Due to Theorem~\ref{theorem: bucic upper bound}, for each layer $j = 2t, 2t + 1$, we have
\begin{align*}
    |\cT^{(j)}| \le C_j s^t \binom{wn}{j - t}.
\end{align*}
Using $(wn/k, t)$-spreadness of $\cA = \binom{[n]}{k/w}^w$, guaranteed by Proposition~\ref{proposition: spread families examples}, we obtain
\begin{align*}
    \sum_{T \in \cT} |\cR_T| & \le \frac{C_{2 t + 1} s^t }{n} \sum_{j = 2t}^{2 + 1} \binom{w n}{j - t} A_{j}\\
    & \le \frac{C'_t s^t}{n}\sum_{j = 2t}^{2t+1} \left ( wn\right )^{j - t} \left ( \frac{wn}{k} \right )^{-(j - t)} A_t \\
    & \le \frac{2C'_t s^t}{n} \cdot k^{t + 1} A_t  \le \frac{2 C_{2t+1} k^{t + 1} s^t}{n} \cdot A_t.
\end{align*}

Define
\begin{align*}
    \cS^* = \left \{ \{T \} \vee \cS_T \mid T \in \cT \right \}.
\end{align*}
Then, $\cS^*$ is $\tau$-homogeneous {\Sname} of $\ff[\cT] \setminus \bigsqcup_{T \in \cT} (\{T\} \vee \cR_T)$. Since $k \ge 5t \ge |T| + q + (t - 1)$ and $wn/k \ge  \sqrt{2^{30}(s k^2)^{4t + 2}} / k  \cdot \sqrt{n} \ge 2^{10} (4k)^t s \sqrt{n}$, we may apply Lemma~\ref{lemma: intersection graph structure} with $\alpha = 1/4k$, $r = wn/k$, $q = 4t + 1$ to $\cS^*$, and obtain\footnote{Note that $q$ here is the parameter of Lemma~\ref{lemma: intersection graph structure} instead of $q = 2t$ used in the spread approximation} $(\cS^*, s, t)$ system $\cU$, such that for each $V \in \cS^*$
\begin{align*}
    |\cU_V| \ge \frac{1}{2} |\ff_V| \text{ and } |\partial_{h}(\cU_V)| \ge (\tau/2)^{-h} |\cA(V)|\ge 2^{-h} n^{-h/2t} |\cA(V)|
\end{align*}
for any $h < k - |V|$. Applying Proposition~\ref{proposition: clustering of (S, s, t)-system into t-sets} with $\lambda = 2^{-t} n^{-1/2}$ and $r = wn/k$, we infer that there exists a family $\hat{\cT}$ such that
\begin{align*}
    |\hat{\cT}| & \le 2^t \sqrt{n} \left \{1 + 2 \phi(s, t) \ln \sum_{j = 0}^q \binom{wn}{q} \right \}  \le 2^t \sqrt{n} \left \{ 1 + 2 \phi(s, t) \ln \sum_{j = 0}^q (wn)^{2q} \right \} \\
    & \le 2^t \sqrt{n}\cdot 7q \phi(s, t) \ln (wn) 
     = 35 t \cdot 2^t \phi(s, t) \sqrt{n} \ln (wn) \le 70 t 2^t \phi(s, t) \sqrt{n} \ln n,
\end{align*}
and
\begin{align*}
    \sum_{\cS^* \setminus \cS^*[\hat{\cT}]} |\cU_S| & \le \frac{8 \cdot 2^t s^2 (5t)^3k}{w\sqrt{n}} \cdot  \phi(s, t) \ln\left ( \sum_{j = 0}^q \binom{wn}{j}\right ) \cdot A_t \\
    & \le \frac{2^{14 + t} s^{2} t^4 k \ln (wn)}{w \sqrt{n}}  \cdot  \phi(s, t) \cdot A_t \le \frac{2^{15 + t} s^{2} t^4 k \ln n}{w \sqrt{n}}  \cdot  \phi(s, t) \cdot A_t
\end{align*}
Employing bound from Claim~\ref{claim: sunflower bound}, we conclude that there is some constant $C_t'$ depending on $t$ only such that the following holds:
\begin{align*}
    |\ff \setminus \ff[\hat{\cT}]| & \le  |\ff \setminus \ff[\cT]|  + \sum_{T \in \cT} |\cR_T| + \sum_{V \in \cS^* \setminus \cS^*[\hat{\cT}]} |\ff_V|   \\
    & \le |\ff \setminus \ff[\cT]|  + \sum_{T \in \cT} |\cR_T| + 2 \sum_{V \in \cS^* \setminus \cS^*[\hat{\cT}]} |\cU_V| \\
    & \le \frac{C'_t s^t}{n} \cdot A_t + \frac{C'_t s^t k^{t + 1}}{n} \cdot A_t  + \frac{C'_t s^t k\ln n}{\sqrt{n}} \cdot A_t \\
    & \le \frac{\hat{C}_t s^t k \ln n}{\sqrt{n}}\cdot A_t,
\end{align*}
where $\hat{C}_t$ is some constant depending on $t$ only.
Setting $\cC^* = \hat{\cT}$, we obtain the lemma.
\end{proof}

\subsection{Proof of Theorem~\ref{theorem: large r theorem}}
\label{subsection: proof of large r thereom}

\begin{proof} We have $n \ge k^{20 t} \cdot s^{10t} \sqrt{g_0(t)}$, so conditions of Lemmma~\ref{lemma: small packing of F -- small uniformity case} are satisfied, provided $g_0(t)$ is a large enough function of $t$. Due to Lemma~\ref{lemma: small packing of F -- small uniformity case}, we have a family $\cC^* \in \binom{n}{t}$ of size at most $70 t 2^t \phi(s, t) \sqrt{n} \ln n$, such that
\begin{align}
    |\ff \setminus \ff[\cC^*]| \le \frac{C_t s^t k \ln n}{\sqrt{n}}\cdot A_t. \label{eq: not covered by small packing C*}
\end{align}
Consider an arbitrary partition of $\ff[\cC^*]$ into $|\cC^*|$ families $\ff_C$, $C \in \cC^*,$ such that $\ff_C \subset \ff(C)$ and
\begin{align*}
    \ff[\cC^*] = \bigsqcup_{C \in \cC^*} \ff_C \vee \{C\}.
\end{align*}

For each $\ff_C$, we construct a family $\ff^*_C$ as follows:
\begin{align*}
    \ff^*_C = \left \{ F \in \ff_C \mid \forall X \in \binom{F}{\le t - 1} \text{ the  covering number of } \ff_C(X) \text{ is at least $sk$}\right \}.
\end{align*}
Then, using $(wn/k, t)$-spreadness of $\cA$ from Proposition~\ref{proposition: spread families examples}, we obtain
\begin{align}
    |\ff_C \setminus \ff^*_C| 
    & \le \sum_{j = 0}^{t - 1} \sum_{X \in \partial_j \cA \mid \tau(\ff_C(X)) < sk} \coveringNumber(\ff_C(X)) \cdot A_{|C| + j + 1} \nonumber \\
    & \le s k \sum_{j = 0}^{t - 1} \binom{wn}{j} \left ( \frac{wn}{k} \right )^{-(j + 1)} A_t \nonumber \\
    & = \frac{sk}{wn} \sum_{j = 0}^{t - 1} k^j \cdot A_t = \frac{sk}{wn} \cdot \frac{k^t - 1}{k - 1} \overset{{\color{teal}(k \ge 5)}}{\le} \frac{2 sk^t}{n} \cdot A_t, \label{eq: small covering number subsets}
\end{align}
where $\tau(\ff_C(X))$ stands for the covering number of $\ff_C(X)$.
Finally, we consider the decomposition
\begin{align*}
    \sum_{C \in \cC^*} |\ff_C^*| = \sum_{F \in \partial_{k - t} \cA} |\{C \in \cC^* \mid F \in \ff_C^*\}|.
\end{align*}
Given $F \in \partial_{k-t}\mathcal A$, we denote $\cC_F = \{C \in \cC^* \mid F \in \ff_C^*\}$. We claim that $\cC_F$ does not contain a sunflower with $s$ petals and the core of arbitrary size. Indeed, let $C_1, \ldots, C_s \in \cC_F$ be such a sunflower. Take any $X \subset F$ of size $t - 1 - \left |\cap_{\ell = 1}^s C_\ell \right |$, and note that $X$ can be empty. Then, for each $\ell = 1, \ldots, s$, $\ff_{C_\ell}(X)$ has covering number at least $sk$. Take the maximum number of pairwise disjoint sets $F_1 \in \ff_{C_1}(X), \ldots, F_m \in \ff_{C_m}(X)$, that do not intersect $\cup_{\ell = 1}^m C_\ell$. To avoid contradiction, we should assume that $m < s$, and so consider $\ff_{C_{m + 1}}(X)$. But since the covering number of $\ff_{C_{m + 1}}(X)$ is at least $sk$, there is a set $F_{m + 1}$ that does not intersect $\cup_{\ell = 1}^m C_\ell \cup F_\ell$, contradicting  the maximality of $m$. Hence, we have
\begin{align*}
    \sum_{C \in \cC^*} |\ff_C^*| = \sum_{F \in \partial_{k - t} \cA} |\cC_F| \le \phi(s, t) |\partial_{k - t} \cA|.
\end{align*}
Combining the above, \eqref{eq: small covering number subsets} and \eqref{eq: not covered by small packing C*}, we obtain
\begin{align*}
    |\ff| & \le \phi(s, t) |\partial_{k - t} \cA| + \frac{2 s k^t |\cC^*|}{n} \cdot A_t + \frac{C_t s^t k \ln n}{\sqrt{n}} \cdot A_t.
\end{align*}
Since $|\cC^*| \le 14 t \cdot 2^t \phi(s, t) \sqrt{n} \ln n$, using Claim~\ref{claim: sunflower bound}, we get  that $|\cC^*| \le C'_t s^t \sqrt{n} \ln n$ for some constant depending on $t$ only.
Therefore, we have
\begin{align*}
    |\ff|  \le \phi(s, t) \cdot \left |\partial_{k - t}\cA \right | + \frac{\tilde C_t s^t k^t \ln n}{\sqrt{n}} \cdot A_t.
\end{align*}
for some constant $\tilde C_t$ depending on $t$ only.
 By the conditions of the theorem, $n \ge k^{20 t}$, so $k^t \le n^{1/20}$. Similarly, we have that $s^t \le n^{1/40}$ and $C'_t \ln n \le n^{1/40}$, provided $g_0(t)$ is large enough. Therefore, we have
\begin{align*}
    C'_t s^t k^t \ln n \le n^{1/10},
\end{align*}
and so 
\begin{align*}
    |\ff| \le \phi(s, t) \cdot |\partial_{k - t} \cA| + n^{-1/3} \cdot A_t. & \qedhere
\end{align*}
\end{proof}

\section{Proof sketch of Theorem~\ref{theorem: product case - small k uniformity}}
\label{section: proof sketch for product case -- small k uniformity}

We start by applying Lemma~\ref{lemma: small packing of F -- small uniformity case} to $\ff$. Then, we obtain a family $\cC^* \subset \partial_t \cA$, $|\cC^*| \le 70 t \cdot 2^t \phi(s, t) \sqrt{n} \ln n$, such that
\begin{align*}
    |\ff \setminus \ff[\cC^*]| \le \frac{C_t s^t k \ln n}{\sqrt{n}} \cdot A_t.
\end{align*}
Consider an arbitrary partition of $\ff[\cC^*]$ into $|\cC^*|$ families $\ff_C, C \in \cC^*,$ such that $\ff_C \subset \ff(C)$ and 
\begin{align*}
    \ff[\cC^*] = \bigsqcup_{C \in \cC^*} \ff_C \vee \{C\}.
\end{align*}
As in the proof of Theorem~\ref{theorem: large r theorem}, we define
\begin{align*}
    \ff^*_C = \left \{ F \in \ff_C \mid \forall X \in \binom{F}{\le t - 1} \text{ the  covering number of } \ff_C(X) \text{ is at least $sk$}\right \}.
\end{align*}
Due to~\eqref{eq: small covering number subsets}, we have
\begin{equation}\label{eqtwostars}
    |\ff_C \setminus \ff^*_C| \le \frac{2 s k^t}{n} \cdot A_t.
\end{equation}
The idea of the proof is to apply a Kruskal--Katona--type result to show that the number of $C \in \cC^*$, such that the shadow of $\ff_C$ is large, is at most $\left ( 1 + e^{-O(k/t)} \right )\phi(s, t)$, and if the shadow of $\ff_C$ is small, then the impact of $\ff_C$ to $|\ff|$ is negligible. Although the exact counterpart of Kruskal--Katona theorem for $[n]^{k}$ is known~\cite{frankl_shadows_1988}, we would like to highlight that a similar but less accurate result can be deduced from Assumptions~\ref{assumption: regularity},\ref{assumption: weak shadow consistency - small r}.

\begin{proposition}
\label{proposition: kruskal--katona}
Suppose that a $k$-uniform $\cA$ satisfies Assumptions~\ref{assumption: regularity},\ref{assumption: weak shadow consistency - small r} with $q = \gamma k$ for some $\gamma < \min\{1, \mu\}$. Let $\ff \subset \cA$ be any family such that $\mu(\partial_h \ff) \le x$. Then, we have
\begin{align*}
    \mu(\ff) \le \left (1 - \frac{\gamma}{\mu} \right )^{-\gamma k} x^{\gamma k / h}.
\end{align*}
\end{proposition}
\begin{proof}
    We assume that $x^{1/h} < \left (1 - \frac{\gamma}{\mu} \right )$, since otherwise there is nothing to prove. Consider an arbitrary $\tau$ such that $1 < \tau < x^{-1/h} (1 - \gamma/\mu)$. Consider the largest $R$ such that $\mu(\ff(R)) \ge \tau^{|R|}\mu(\ff)$. In particular, $\ff(R) \subset \cA(R)$ is $\tau$-homogeneous and $\mu(\ff) \le \tau^{-|R|}$.

    If $|R| \le \gamma k = q$, then by Proposition~\ref{proposition: shadow of tau-homogeneous family}, we have
    \begin{align*}
        \mu(\partial_h(\ff(R))) \ge \tau^{-h}.
    \end{align*}
    Since $|\partial_h(\ff(R))| \le |\partial_h\ff|$, the measure $\mu(\partial_h(\ff(R)))$ is at most
    \begin{align*}
        \mu(\partial_h(\ff(R))) = \frac{|\partial_h(\ff(R))|}{|\partial_h (\cA(R))|} \le \frac{|\partial_h \ff|}{|\partial_h \cA|} \cdot \frac{|\partial_h \cA|}{|\partial_h(\cA(R))|} = \mu(\partial_h \ff) \cdot \frac{|\partial_h \cA|}{|\partial_h (\cA(R))|}.
    \end{align*}
    We bound the above using Assumption~\ref{assumption: weak shadow consistency - small r}, and obtain
    \begin{align*}
        \tau^{-h} \le x \cdot \left (1 - \frac{\gamma}{\mu} \right )^{-h},
    \quad
    \text{which implies}
    \quad
        \tau \ge x^{-1/h} \left (1 - \frac{\gamma}{\mu} \right ),
    \end{align*}
    contradicting the definition of $\tau$. Thus, $|R| > \gamma k$, and we have
    \begin{align*}
        \mu(\ff) \le \tau^{-|R|} \le \tau^{-\gamma k}.
    \end{align*}
    Since $\tau$ is an arbitrary number less than $ x^{-1/h} (1 - \gamma/\mu)$, the displayed chain of inequalities above implies the proposition.
\end{proof}

Next, we show that Proposition~\ref{proposition: kruskal--katona} can be applied to $\cA(C)$, $C \in \cC^*$. 
\begin{claim}
\label{claim: kruskal-katona to w-th power}
    Suppose that $\cA = \binom{[n]}{k/w}^w$ for some $w | k$ such that $k \ge 4w t$. Then, for any $C \in \partial_t \cA$ and $\ff \subset \cA(C)$ such that $\mu(\partial_h\ff) \le x$, we have
    \begin{align*}
        \mu(\ff) \le \left ( 1 - \frac{3k}{2wn}\right )^{-k/20} x^{(k - t)/2wh}.
    \end{align*}
\end{claim}
\begin{proof}
    We are going to show that Assumptions~\ref{assumption: regularity},\ref{assumption: weak shadow consistency - small r} are satisfied for $\cA(C)$ with $q = (k - t) /2w$. Assumption~\ref{assumption: regularity} is satisfied due to Claim~\ref{claim: regularity for general families}. It remains to show that for any $R$ of size at most $q = (k - t)/2w$, we have
    \begin{align*}
        \frac{|\partial_h (\cA(C \cup R))|}{|\partial_h (\cA(C))|} \ge \left (1 - \frac{q}{\mu'(k - t)} \right )^h = \left ( 1 - \frac{1}{2w \mu'} \right )^{\textcolor{red}{h}}, \quad \text{where } \mu' = 2 \mu/3,
    \end{align*}
    so the condition of Proposition~\ref{proposition: kruskal--katona} holds with $\gamma = 1/2w$ and $\mu = \mu'$, $k - t$ in place of $k$. Note that 
    \begin{align*}
        |R \sqcup C| \le \frac{k - t}{2w} + t \le \frac{k + 2wt}{2w} \le \frac{3k}{4w},
    \end{align*}
    where we used $k \ge 4wt$, so due to Proposition~\ref{proposition: assumptions satisified} applied with $q = 3k/4w$, we have 
    \begin{align*}
        \frac{|\partial_h (\cA(C \cup R))|}{|\partial_h (\cA(C))|} \ge \frac{|\partial_h (\cA(C \cup R))|}{|\partial_h \cA|} \ge \left (1 - \frac{3k}{4 w n} \right )^h = \left (1 - \frac{1}{2w \mu'} \right )^h.
    \end{align*}
    Then, Proposition~\ref{proposition: kruskal--katona} implies the desired bound on $\mu(\ff)$.
\end{proof}

Given a real number $\delta \in (0, 1)$, define a family $\cC^*_\delta$ as follows:
\begin{align*}
    \cC^*_\delta = \left \{ C \in \cC^* \mid \mu(\partial_{t - 1}\ff^*_C) \ge \delta \right \}.
\end{align*}
Let us show that $|\cC^*_\delta| \le \phi(s, t)/ \delta$. First, by  double-counting, there exists $X \in \partial_{t - 1} \cA$ and a family $\cC' \subset \cC^*_{\delta}$ of size at least $\delta |\cC^*_{\delta}|$ such that $X \in \partial_{t - 1} \ff^*_C$ for any $C \in \cC'$. We claim that $\cC'$ does not contain a sunflower with $s$ petals. Arguing indirectly, let $C_1, \ldots, C_s$ be a sunflower in $\cC'$. Choose $Y \subset X$ of size $t - 1 - |\cap_{j \in [s]} C_j|$. Then, $C_1 \cup Y, \ldots, C_s \cup Y$ form a sunflower with $s$ petals and the core of size $t - 1$. Since the covering number of $\ff_{C_i}(Y)$ is at least $sk$ for each $i = 1, \ldots, s$, we can find a sunflower with the core of size $t - 1$ in $\ff$, a contradiction. Thus, $|\cC'| \le \phi(s, t)$, and $|\cC^*_\delta|\le \phi(s,t) / \delta$. 

Fix $\delta = \left ( 1 - \sqrt{\frac{t}{k}} \right )$. Then, we decompose the sum $\sum_{C \in \cC^*} |\ff_C|$ as follows:
\begin{align}
    \sum_{C \in \cC^*} |\ff_C^*| & = \sum_{C \in \cC^*_{\delta}} |\ff_C^*| + \sum_{j = 1}^{\infty} \sum_{C \in \cC^*_{\delta^{j + 1}} \setminus \cC^*_{\delta^{j}}} |\ff_C^*|. \label{eq: diadic decomposition of f with large covering number}
\end{align}
Consider two cases. 

\textbf{Case 1. We have $\mathbf{\cA = \binom{[n]}{k/w}^w}$ for $\mathbf{k \ge 2^6 w^6 t}$.} For each $j$ and $C \in \cC^*_{\delta^{j + 1}} \setminus \cC_{\delta^j}$, we have
\begin{align*}
    |\cC^*_{\delta^{j + 1}} \setminus \cC^*_{\delta^j}| \le |\cC^*_{\delta^{j + 1}}| \le \delta^{-j} \cdot \phi(s, t) \quad \text{ and } \quad |\ff_C^*| \le A_t \cdot \left (1 - \frac{3k}{2 wn} \right )^{-k/20} \delta^{\frac{j (k - t)}{2 w (t - 1)}},
\end{align*}
where the second fact follows from Claim~\ref{claim: kruskal-katona to w-th power}.
We have
{\small \begin{align}
    \sum_{C \in \cC^*} |\ff_C^*| & \le \frac{\phi(s, t) A_t}{1 - \sqrt{\frac{t}{k}}} + \left (1 - \frac{3k}{wn} \right )^{-k/20}  \phi(s, t)A_t  \cdot \sum_{j = 1}^{\infty} \left ( 1 - \sqrt{\frac{t}{k}}\right )^{j \cdot \left (\frac{k - t}{2w(t - 1)} - 1 \right )} \nonumber \\
    & \le \frac{\phi(s, t)}{1 - \sqrt{\frac{t}{k}}} + \left (1 - \frac{3k}{wn} \right )^{-k/20} \phi(s, t) A_t \cdot \sum_{j = 1}^{\infty} \exp \left ( - j \cdot \left ( \frac{k - t}{2w(t - 1)} - 1 \right ) \cdot \sqrt{\frac{t}{k}} \right ). \label{eq: bound of w-th power shadow decomposition}
\end{align}}
Next, we bound
\begin{align*}
    \left ( \frac{k - t}{2w(t - 1)} - 1 \right ) \cdot \sqrt{\frac{t}{k}} = \frac{k - t - 2w(t - 1)}{2w(t - 1)} \cdot \sqrt{\frac{t}{k}} \ge \frac{k}{4wt} \cdot \sqrt{\frac{t}{k}} = \frac{1}{4w} \cdot \sqrt{\frac{k}{t}},
\end{align*}
where we used $k \ge 2^{12} w^6 t \ge 2(t + 2w(t - 1))$. We have $k \ge 2^{12} w^6 t \ge 2^{18} t \ge e^5 t$, so $\left ( \frac{k}{t} \right )^{1/3} \ge \frac{1}{2} \ln \frac{k}{t} $, and
\begin{align*}
    \sqrt{\frac{k}{t}} \ge \left ( \frac{k}{t} \right )^{1/6} \ln \sqrt{\frac{k}{t}}.
\end{align*}
Since $k \ge 2^{12} w^6 t$, we have
\begin{align*}
    \left ( \frac{k - t}{2w(t - 1)} - 1 \right )\sqrt{\frac kt} \ge \frac{1}{4w} \left ( \frac{k}{t} \right )^{1/6} \ln \sqrt{\frac{k}{t}} \ge \ln \sqrt{\frac{k}{t}}.
\end{align*}
Using the displayed inequality, we may continue~\eqref{eq: bound of w-th power shadow decomposition} as follows.
\begin{align}
   \notag \sum_{C \in \cC^*} |\ff_C^*| &\le \frac{\phi(s, t) A_t}{1 - \sqrt{\frac{t}{k}}} + \left (1 - \frac{3k}{wn} \right )^{-k/20} \phi(s, t) A_t \sum_{j = 1}^\infty \left ( \frac{t}{k}\right )^{j/2}\\
   &= \frac{1 + \left (1 - \frac{3k}{wn} \right )^{-k/20} \sqrt{\frac{t}{k}}}{1 - \sqrt{\frac{t}{k}}} \cdot \phi(s, t) A_t. \label{eq: simpliefied bound of 2-th power shadow decompositiom}
\end{align}

Using $k \ge 2^{12} w^6 t \ge 4 t$, we bound
\begin{align*}
    \left (1 - \sqrt{\frac{t}{k}} \right )^{-1} = 1 + \left (1 - \sqrt{\frac{t}{k}} \right )^{-1} \sqrt{\frac{t}{k}} \le 1 + 2 \sqrt{\frac{t}{k}}.
\end{align*}
Similarly, since $wn \ge k^{20t} \ge 5^{18} k^2 t$, we bound
\begin{align*}
    \left (1 - \frac{3k}{wn} \right )^{-k/20} \le \left (1 + \frac{9k}{wn} \right )^{k/20} \le 3.
\end{align*}
Combining the above and~\eqref{eq: simpliefied bound of 2-th power shadow decompositiom}, we obtain
\begin{align}
    \sum_{C \in \cC^*} |\ff_C^*| \le \left ( 1 + 3 \sqrt{\frac{t}{k}} \right )\left ( 1 + 2 \sqrt{\frac{t}{k}} \right ) \phi(s, t) A_t\le \left ( 1 + 6 \sqrt{\frac{t}{k}} \right ) \phi(s, t) A_t. \label{eq: bound w-th power on large cov number families}
\end{align}

\textbf{Case 2. We have $\cA = [n]^k$.} In this case, instead of Claim~\ref{claim: kruskal-katona to w-th power}, we will use the following theorem due to Frankl, F\"uredi and Kalai.
\begin{theorem}[Theorem 5.1 of~\cite{frankl_shadows_1988}]
    Let $\ff \subset [n]^k$, such that $|\ff| = y^k$ for some real $y$. Then, for any $h \in [k]$, we have
    \begin{align*}
        |\partial_h \ff| \ge \binom{k}{h} y^h.
    \end{align*}
\end{theorem}
Since $|\partial_h [n]^k| = \binom{k}{h} n^h$, this theorem implies the following corollary:
\begin{corollary}
\label{corollary: kruskal-katona for k power of n}
    Consider a family of sets $\ff \subset [n]^k$, and suppose that $\mu(\partial_h \ff) \le x$. Then, we have $\mu(\ff) \le x^{k/h}$.
\end{corollary}
Hence, for each $j \ge 1$ and $C \in \cC^*_{\delta^{j + 1}} \setminus \cC^*_{\delta^j}$, we have
\begin{align*}
    |\cC^*_{\delta^{j + 1}} \setminus \cC^*_{\delta^j} | \le \delta^{-j} \cdot \phi(s, t) \quad \text{and} \quad |\ff^*_C| \le \delta^{\frac{j(k - t)}{t - 1}} A_t.
\end{align*}
Substituting the above inequalities into~\eqref{eq: diadic decomposition of f with large covering number}, we obtain
\begin{align*}
    \sum_{C \in \cC^*} |\ff_C^*| & \le \frac{\phi(s, t) \cdot A_t}{1 - \sqrt{\frac{t}{k}}} + \phi(s, t) A_t \cdot \sum_{j = 1}^{\infty} \left ( 1 - \sqrt{\frac{t}{k}}\right )^{j \cdot \left ( \frac{k - t}{t - 1} - 1 \right )} \\
    & \le \frac{\phi(s, t) A_t}{1 - \sqrt{\frac{t}{k}}} + \phi(s, t) A_t \cdot \sum_{j = 1}^{\infty} \exp \left ( - j \sqrt{\frac{t}{k}} \left ( \frac{k - t}{t - 1} - 1 \right )\right ).
\end{align*}
Since $k \ge 5t$ and $\sqrt{y} \ge \ln y$ for any positive $y$,we can bound
\begin{align*}
    \sqrt{\frac{t}{k}} \left ( \frac{k - t}{t - 1} - 1\right ) \ge \sqrt{\frac{t}{k}} \cdot \frac{k - 2t}{t} \ge \frac{1}{2} \cdot \sqrt{\frac{k}{t}} \ge \frac{1}{2} \ln \left (\frac{k}{t} \right ),
\end{align*}
and obtain
\begin{align*}
    \sum_{C \in \cC^*} |\ff_C^*| \le \frac{1 + \sqrt{\frac{t}{k}}}{1 - \sqrt{\frac{t}{k}}} \cdot \phi(s, t) \cdot A_t \le \left ( 1 + 6 \sqrt{\frac{t}{k}}\right ) \phi(s, t) \cdot A_t.
\end{align*}
\vskip+0.2cm

Thus, the above bound holds for both $\binom{n}{k/w}^w$, $k \ge 2^{12}w^6 t$, and $[n]^k$. Using~\eqref{eqtwostars}, we get
\begin{align*}
    \sum_{C \in \cC^*} |\ff_C| \le \left (1 + 6 \sqrt{\frac{t}{k}} \right ) \phi(s,t) A_t + |\cC^*| \cdot \frac{2sk^t}{n} \cdot A_t \le \left ( 1 + 9 \sqrt{\frac{t}{k}} + \frac{70 s t2^tk^t \ln n}{\sqrt n}\right ) \phi(s, t) \cdot A_t,
\end{align*}
where we use the bound on $|\cC^*|$ form Lemma~\ref{lemma: small packing of F -- small uniformity case}. Using the bound on $|\ff \setminus \ff[\cC^*]|$ provided by this lemma, we get
\begin{align*}
    |\ff| \le \left (1 + 6 \sqrt{\frac{t}{k}} + \frac{70 st2^tk^t \ln n}{\sqrt n} \right ) \phi(s, t) A_t + \frac{C_t s^t \ln n}{\sqrt{n}} A_t.
\end{align*}
Since $n \ge k^{20t}$ and $n \ge s^{40t} \cdot g_0(t)$, for large enough $g_0(t)$, we have
\begin{align*}
    |\ff| \le \left (1 + 7 \sqrt{\frac{t}{k}} \right ) \phi(s, t) A_t.
\end{align*}

\section{Conclusion}

In the present paper, we provide new bounds for the problem of Erd\H{o}s--Duke on forbidden sunflowers with prescribed core size. In particular, we extend the result of Frankl and F\"uredi~\cite{Frankl1987} to the case when $n$ is allowed to depend linearly on $k$ with the factor polynomial in $s, t$. To obtain this result, we combine the hypercontractivity technique of Keller, Lifshizt, Keevash et al.~\cite{keevash2021global, keller_sharp_2023}, the spread approximation technique of Kupavskii and Zakharov~\cite{kupavskii_spread_2022}, the result on forbidden sunflowers of Brada\v{c} et al.~\cite{bradavc2023turan} and some ideas from the original proof of Frankl and F\"uredi~\cite{Frankl1987}. As a byproduct, we answer the question of Kalai~\cite{Poly} and Brada\v{c} et al.~\cite{bradavc2023turan} about the maximal size of families without sunflowers with the core of size at most $t - 1$ in the regime when $s \le (t\ln (n/k))^{-3} \sqrt{n/k}$ and $n \ge 2^{11} s k \log_2 k$.

It turns out that our methods are quite general and can be applied to obtain similar results for subfamilies of permutations and $\binom{[n]}{k/w}^w$ for certain ranges of values $w$.

When looking through the proof of Theorem~\ref{theorem: main theorem}, one can ask the following question: could we employ modern techniques to improve dependencies in the Delta-system method? In our problem, $k \ge 5t$ is enough to significantly decrease the lower bound on $n$. Could we go further and avoid double-exponential dependence on $t$? Could we improve similar lower bounds for other Tur\'an-type problems?

For abstract hypergraph Tur\'an-type problems, it seems that modern techniques can do this when $k$ is at least constant times the uniformity of a forbidden hypergraph. Moreover, Alweiss claimed that his result on set system blowups~\cite{alweiss2020set} can significantly improve the key lemma~\cite{furedi1983finite} of the Delta-system method. However, his claim has not been verified. Meanwhile, we have other approaches like spread approximations~\cite{kupavskii_spread_2022}, the method of Lifshitz and coauthors based on Boolean analysis~\cite{keevash2021global, keller_sharp_2023}, including the junta method~\cite{keller2021junta}, and encoding techniques used in papers of Brada\v{c} et al.~\cite{bradavc2023turan}, Alweiss~\cite{alweiss2021improved} and Rao~\cite{rao_coding_2020}, that seem to be general enough to replace the Delta-system method. We hope that a framework alternative to the latter and without double-exponential dependencies will be designed in the near future.

\printbibliography

@article{kupavskii_spread_2022,
      title={Spread approximations for forbidden intersections problems},
      author={Kupavskii, Andrey and Zakharov, Dmitrii},
      journal={Advances in Mathematics},
      volume={445},
      pages={109653},
      year={2024},
      publisher={Elsevier},
      doi = {10.1016/j.aim.2024.109653}
}

@misc{keller_sharp_2023,
  title={Sharp hypercontractivity for global functions},
  author={Keller, Nathan and Lifshitz, Noam and Marcus, Omri},
  note={arXiv preprint},
  doi={10.48550/arXiv.2307.01356},
  year={2023}
}

@article{frankl_erdos-ko-rado_1999,
	title = {The Erd\H{o}s-Ko-Rado Theorem for Integer Sequences},
	volume = {19},
	issn = {0209-9683, 1439-6912},
	doi = {10.1007/s004930050045},
	pages = {55--63},
	number = {1},
	journaltitle = {Combinatorica},
	shortjournal = {Combinatorica},
	author = {Frankl, Peter and Tokushige, Norihide},
	date = {1999-01-01},
}

@article{frankl_shadows_1988,
  title={Shadows of colored complexes},
  author={Frankl, Peter and F{\"u}redi, Zolt{\'a}n and Kalai, Gil},
  journal={Mathematica scandinavica},
  pages={169--178},
  year={1988},
  publisher={JSTOR},
    url = {https://www.jstor.org/stable/24492632}
}

@misc{keevash2021global,
      title={Global hypercontractivity and its applications},
      author={Keevash, Peter and Lifshitz, Noam and Long, Eoin and Minzer, Dor},
      note={arXiv preprint},
      doi={10.48550/arXiv.2103.04604},
      year={2021}
}

@article{Frankl1987,
title={Exact solution of some Tur{\'a}n-type problems},
  author={Frankl, Peter and F{\"u}redi, Zolt{\'a}n},
  journal={Journal of Combinatorial Theory, Series A},
  volume={45},
  number={2},
  pages={226--262},
  year={1987},
  publisher={Elsevier},
    doi={10.1016/0097-3165(87)90016-1}
}

@article{alweiss2021improved,
  title={Improved bounds for the sunflower lemma},
  author={Alweiss, Ryan and Lovett, Shachar and Wu, Kewen and Zhang, Jiapeng},
  journal={Annals of Mathematics},
  volume={194},
  number={3},
  pages={795--815},
  year={2021},
  publisher={Department of Mathematics, Princeton University Princeton, New Jersey, USA},
    doi={10.4007/annals.2021.194.3.5}
}

@book{arora_computational_2009,
	edition = {1},
	title = {Computational {Complexity}: {A} {Modern} {Approach}},
	copyright = {https://www.cambridge.org/core/terms},
	isbn = {978-0-521-42426-4},
	shorttitle = {Computational {Complexity}},
	publisher = {Cambridge University Press},
	author = {Arora, Sanjeev and Barak, Boaz},
	month = apr,
	year = {2009},
	doi = {10.1017/CBO9780511804090}
}

@article{bell_note_2021,
	title = {Note on sunflowers},
	volume = {344},
	issn = {0012365X},
	doi = {10.1016/j.disc.2021.112367},
	language = {en},
	number = {7},
	journal = {Discrete Mathematics},
	author = {Bell, Tolson and Chueluecha, Suchakree and Warnke, Lutz},
	month = jul,
	year = {2021},
	pages = {112367},
}

@misc{Poly,
	title = {The Erd\H os-Rado Delta System Conjecture},
	author = {Polymath~10},
        url = {https://gilkalai.wordpress.com/2015/11/03/polymath10-the-erdos-rado-delta-system-conjecture},
        urldate = {2024-04-07}
}

@article{frankl_improved_2013,
	title = {Improved bounds for {Erd\H{o}s} {Matching} {Conjecture}},
	volume = {120},
	doi = {10.1016/j.jcta.2013.01.008},
	language = {en},
	number = {5},
	journal = {Journal of Combinatorial Theory, Series A},
	author = {Frankl, Peter},
	month = jul,
	year = {2013},
	pages = {1068--1072},
}

@incollection{Fur1991,
	edition = {1},
	title = {{Tur\'an} {Type} {Problems}},
	language = {en},
	booktitle = {Surveys in {Combinatorics}, 1991},
	publisher = {Cambridge University Press},
	author = {Füredi, Zoltán},
	editor = {Keedwell, A. D.},
	month = aug,
	year = {1991},
	doi = {10.1017/CBO9780511666216.010},
	pages = {253--300},
}

@misc{stoeckl_lecture_nodate,
	title = {Lecture notes on recent improvements for the sunflower lemma},
	url = {https://mstoeckl.com/notes/research/sunflower_notes.html},
	author = {Stoeckl, Manuel},
        urldate = {2024-04-07}
}

@misc{tao_sunflower_2020,
	title = {The sunflower lemma via {Shannon} entropy},
	url = {https://terrytao.wordpress.com/2020/07/20/the-sunflower-lemma-via-shannon-entropy/},
	language = {en},
	urldate = {2024-06-19},
	journal = {What`s new},
	author = {Tao, Terrence},
	month = jul,
}

@misc{rao_coding_2020,
	title = {Coding for {Sunflowers}},
	doi = {10.48550/arXiv.1909.04774},
	publisher = {arXiv},
	author = {Rao, Anup},
	year = {2020},
	note = {arXiv preprint},
}

@misc{hu_entropy_2021,
	title = {Entropy {Estimation} via {Two} {Chains}: {Streamlining} the {Proof} of the {Sunflower} {Lemma}},
	shorttitle = {Entropy {Estimation} via {Two} {Chains}},
	url = {https://theorydish.blog/2021/05/19/entropy-estimation-via-two-chains-streamlining-the-proof-of-the-sunflower-lemma/},
	language = {en},
	urldate = {2024-06-19},
	journal = {Theory Dish},
	author = {Hu, Lunjia},
	month = may,
	year = {2021},
}

@book{o2014analysis,
  title={Analysis of Boolean functions},
  author={O'Donnell, Ryan},
  year={2014},
  publisher={Cambridge University Press},
    doi = {10.1017/CBO9781139814782}
}

@book{bakry_analysis_2014,
	series = {Grundlehren der mathematischen {Wissenschaften}},
	title = {Analysis and {Geometry} of {Markov} {Diffusion} {Operators}},
	volume = {348},
	isbn = {978-3-319-00227-9},
	language = {en},
	publisher = {Springer International Publishing},
	author = {Bakry, Dominique and Gentil, Ivan and Ledoux, Michel},
	year = {2014},
	doi = {10.1007/978-3-319-00227-9},
}

@article{keller_t-intersecting_2024,
	title = {On \textit{t}-intersecting families of permutations},
	volume = {445},
	issn = {0001-8708},
	doi = {10.1016/j.aim.2024.109650},
	urldate = {2024-07-02},
	journal = {Advances in Mathematics},
	author = {Keller, Nathan and Lifshitz, Noam and Minzer, Dor and Sheinfeld, Ohad},
	month = may,
	year = {2024},
	pages = {109650},
}

@article{erdos1960intersection,
  title={Intersection theorems for systems of sets},
  author={Erd\H{o}s, Paul and Rado, Richard},
  journal={Journal of the London Mathematical Society},
  volume={1},
  number={1},
  pages={85--90},
  year={1960},
  publisher={Wiley Online Library},
    doi = {10.1112/jlms/s1-35.1.85}
}

@inproceedings{duke1977systems,
  title={Systems of finite sets having a common intersection},
  author={Duke, Richard A. and Erd{\H{o}}s, Paul},
  booktitle={Proceedings, 8th SE Conference on Combinatorics, Graph Theory and Computing},
  pages={247--252},
  url = {https://api.semanticscholar.org/CorpusID:601822},
  year={1977}
}

@article{park2024proof,
  title={A proof of the Kahn--Kalai conjecture},
  author={Park, Jinyoung and Pham, Huy},
  journal={Journal of the American Mathematical Society},
  volume={37},
  number={1},
  pages={235--243},
  year={2024},
    doi = {10.1090/jams/1028}
}

@article{keller2021junta,
  title={The junta method for hypergraphs and the Erd{\H{o}}s-Chv{\'a}tal simplex conjecture},
  author={Keller, Nathan and Lifshitz, Noam},
  journal={Advances in Mathematics},
  volume={392},
  pages={107991},
  year={2021},
  publisher={Elsevier},
    doi = {10.1016/j.aim.2021.107991}
}

@article{keller2023improved,
  title={Improved covering results for conjugacy classes of symmetric groups via hypercontractivity},
  author={Keller, Nathan and Lifshitz, Noam and Sheinfeld, Ohad},
  journal={Forum of Mathematics, Sigma},
  volume={12},
  pages={e85},
  year={2024},
    doi={10.1017/fms.2024.95},
  organization={Cambridge University Press}
}

@misc{kupavskii2023intersection,
  title={Intersection theorems for uniform subfamilies of hereditary families},
  author={Kupavskii, Andrey},
  note={arXiv preprint},
  doi={10.48550/arXiv.2311.02246},
  year={2023}
}

@misc{kupavskii2023erd,
  title={Erd\H{o}s--Ko--Rado type results for partitions via spread approximations},
  doi={10.48550/arXiv.2309.00097},
  author={Kupavskii, Andrey},
  note={arXiv preprint},
  year={2023}
}

@misc{kupavskii2024almost,
  title={An almost complete $ t $-intersection theorem for permutations},
  author={Kupavskii, Andrey},
  doi={10.48550/arXiv.2405.07843},
  note={arXiv preprint},
  year={2024}
}

@article{mubayi2016survey,
  title={A survey of Tur{\'a}n problems for expansions},
  author={Mubayi, Dhruv and Verstra{\"e}te, Jacques},
  journal={Recent trends in combinatorics},
  pages={117--143},
  year={2016},
  publisher={Springer},
    doi = {10.1007/978-3-319-24298-9_5}
}

@inproceedings{erdos1975problems,
  title={Problems and results in graph theory and combinatorial analysis},
  author={Erd{\H{o}}s, Paul},
  booktitle={Proceedings of the Fifth British Combinatorial Conference},
  pages={169--192},
  year={1975},
    url = {https://www.renyi.hu/~p_erdos/1976-36.pdf},
    urldate = {2024-04-07}
}

@article{frankl_forbidding_1985,
	title = {Forbidding just one intersection},
	volume = {39},
	issn = {0097-3165},
	number = {2},
	urldate = {2024-06-25},
	journal = {Journal of Combinatorial Theory, Series A},
	author = {Frankl, Peter and Füredi, Zoltán},
	month = jul,
	year = {1985},
	pages = {160--176},
        doi = {10.1016/0097-3165(85)90035-4}
}

@article{frankl1983extremal,
  title={An extremal set theoretical characterization of some Steiner systems},
  author={Frankl, Peter},
  journal={Combinatorica},
  volume={3},
  pages={193--199},
  year={1983},
  publisher={Springer}
}

@article{ellis2024stability,
  title={Stability for the complete intersection theorem, and the forbidden intersection problem of Erd{\H{o}}s and S{\'o}s},
  author={Ellis, David and Keller, Nathan and Lifshitz, Noam},
  journal={Journal of the European Mathematical Society},
  volume={26},
  number={5},
  pages={1611--1654},
  year={2024},
    doi = {10.4171/JEMS/144}
}

@article{frankl1990partition,
  title={A partition property of simplices in Euclidean space},
  author={Frankl, Peter and R{\"o}dl, Vojtech},
  journal={Journal of the American Mathematical Society},
  volume={3},
  number={1},
  pages={1--7},
  year={1990},
    doi = {10.2307/1990982}
}

@article{sgall1999bounds,
  title={Bounds on pairs of families with restricted intersections},
  author={Sgall,  Ji{\v{r}}í},
  journal={Combinatorica},
  volume={19},
  number={4},
  pages={555--566},
  year={1999},
  publisher={Springer},
    doi = {10.1007/s004939970007}
}

@inproceedings{buhrman1998quantum,
  title={Quantum vs. classical communication and computation},
  author={Buhrman, Harry and Cleve, Richard and Wigderson, Avi},
  booktitle={Proceedings of the thirtieth annual ACM symposium on Theory of computing},
  pages={63--68},
  year={1998},
    doi = {10.1145/276698.276713}
}

@article{erdos1965problem,
  title={A problem on independent r-tuples},
  author={Erd\H{o}s, Paul},
  journal={Annales Universitatis Scientiarium Budapestinensis de Rolando E\v{o}t\v{o}s Nominatae Sectio Mathematica},
  volume={8},
  number={93-95},
  pages={2},
  year={1965},
    url = {https://www.renyi.hu/~p_erdos/1965-01.pdf},
    urldate = {2024-04-07}
}

@article{bollob1976sets,
  title={Sets of independent edges of a hypergraph},
  author={Bollob\'as, B\'ela and Daykin, David E. and Erd\H{o}s, Paul},
  journal={The Quarterly Journal of Mathematics},
  volume={27},
  number={1},
  pages={25--32},
  year={1976},
  publisher={Oxford University Press},
    doi={10.1093/qmath/27.1.25}
}

@article{huang2012size,
  title={The size of a hypergraph and its matching number},
  author={Huang, Hao and Loh, Po-Shen and Sudakov, Benny},
  journal={Combinatorics, Probability and Computing},
  volume={21},
  number={3},
  pages={442--450},
  year={2012},
  publisher={Cambridge University Press},
    doi={10.1017/S096354831100068X}
}

@article{frankl2012matchings,
  title={On matchings in hypergraphs},
  author={Frankl, Peter and {\L}uczak, Tomasz and Mieczkowska, Katarzyna},
  journal={The Electronic Journal of Combinatorics},
  volume={19},
  number={2},
  pages={P42},
  year={2012},
    doi={10.37236/2176}
}

@article{frankl2022erdHos,
  title={The Erd{\H{o}}s matching conjecture and concentration inequalities},
  author={Frankl, Peter and Kupavskii, Andrey},
  journal={Journal of Combinatorial Theory, Series B},
  volume={157},
  pages={366--400},
  year={2022},
  publisher={Elsevier},
    doi={10.1016/j.jctb.2022.08.002}
}

@article{frankl1978extremal,
  title={An extremal problem for 3-graphs},
  author={Frankl, Peter},
  journal={Acta Mathematica Academiae Scientiarum Hungarica},
  volume={32},
  pages={157--160},
  year={1978},
  publisher={Springer},
    doi = {10.1007/BF01902209}
}

@article{chung1987maximum,
  title={The maximum number of edges in a 3-graph not containing a given star},
  author={Chung, Fan R. K. and Frankl, Peter},
  journal={Graphs and Combinatorics},
  volume={3},
  pages={111--126},
  year={1987},
  publisher={Springer},
    doi={10.1007/BF01788535}
}

@article{chung1983unavoidable,
  title={Unavoidable stars in 3-graphs},
  author={Chung, Fan R. K.},
  journal={Journal of Combinatorial Theory, Series A},
  volume={35},
  number={3},
  pages={252--262},
  year={1983},
  publisher={Elsevier},
    doi={10.1016/0097-3165(83)90011-0}
}

@article{bucic2021unavoidable,
  title={Unavoidable hypergraphs},
  author={Buci{\'c}, Matija and Dragani{\'c}, Nemanja and Sudakov, Benny and Tran, Tuan},
  journal={Journal of Combinatorial Theory, Series B},
  volume={151},
  pages={307--338},
  year={2021},
  publisher={Elsevier},
    doi = {10.1016/j.jctb.2021.06.010}
}

@article{bradavc2023turan,
  title={Tur{\'a}n numbers of sunflowers},
  author={Brada{\v{c}}, Domagoj and Buci{\'c}, Matija and Sudakov, Benny},
  journal={Proceedings of the American Mathematical Society},
  volume={151},
  number={03},
  pages={961--975},
  year={2023},
    doi = {DOI: https://doi.org/}
}

@article{erdos1961intersection,
  title={Intersection theorems for systems of finite sets},
  author={Erd\H{o}s, Paul and Ko, Chao and Rado, Richard},
  journal={Quarterly Journal of Mathematics},
  volume={12},
  number={2},
  pages={313--320},
  year={1961},
    doi = {10.1093/qmath/12.1.313}

}

@article{frankl1986erdos,
  title={The Erd\H{o}s-Ko-Rado theorem for vector spaces},
  author={Frankl, Peter and Wilson, Richard M},
  journal={Journal of Combinatorial Theory, Series A},
  volume={43},
  number={2},
  pages={228--236},
  year={1986},
  publisher={Elsevier},
    doi = {10.1016/0097-3165(86)90063-4}
}

@article{frankl1977maximum,
  title={On the maximum number of permutations with given maximal or minimal distance},
  author={Frankl, Peter and Deza, Mikhail},
  journal={Journal of Combinatorial Theory, Series A},
  volume={22},
  number={3},
  pages={352--360},
  year={1977},
  publisher={Elsevier},
    doi = {10.1016/0097-3165(77)90009-7}
}

@article{keevash2023forbidden,
  title={Forbidden intersections for codes},
  author={Keevash, Peter and Lifshitz, Noam and Long, Eoin and Minzer, Dor},
  journal={Journal of the London Mathematical Society},
  volume={108},
  number={5},
  pages={2037--2083},
  year={2023},
  publisher={Wiley Online Library},
    doi = {10.1112/jlms.12801}
}

@article{ellis2023forbidden,
  title={Forbidden intersection problems for families of linear maps},
  author={Ellis, David and Kindler, Guy and Lifshitz, Noam},
  journal={Discrete Analysis},
  volume={19},
  year={2023},
    doi = {10.19086/da.90718}
}

@misc{alweiss2020set,
  title={Set System Blowups},
  doi={10.48550/arXiv.2003.11202},
  author={Alweiss, Ryan},
  note={arXiv preprint},
  year={2020}
}

@article{furedi1983finite,
  title={On finite set-systems whose every intersection is a kernel of a star},
  author={F{\"u}redi, Zolt{\'a}n},
  journal={Discrete mathematics},
  volume={47},
  pages={129--132},
  year={1983},
  publisher={Elsevier},
    doi = {10.1016/0012-365X(83)90081-X}
}

@article{borg2009extremal,
  title={Extremal t-intersecting sub-families of hereditary families},
  author={Borg, Peter},
  journal={Journal of the London Mathematical Society},
  volume={79},
  number={1},
  pages={167--185},
  year={2009},
  publisher={Oxford University Press},
    doi = {10.1112/jlms/jdn062}
}

@article{dinur_intersecting_2009,
	title = {Intersecting {Families} are {Essentially} {Contained} in {Juntas}},
	volume = {18},
	issn = {1469-2163, 0963-5483},
	doi = {10.1017/S0963548308009309},
	language = {en},
	number = {1-2},
	urldate = {2024-10-03},
	journal = {Combinatorics, Probability and Computing},
	author = {Dinur, Irit and Friedgut, Ehud},
	month = mar,
	year = {2009},
	pages = {107--122},
}

@article{kolupaev2023erdHos,
  title={Erd{\H{o}}s matching conjecture for almost perfect matchings},
  author={Kolupaev, Dmitriy and Kupavskii, Andrey},
  journal={Discrete Mathematics},
  volume={346},
  number={4},
  pages={113304},
  year={2023},
  publisher={Elsevier},
    doi = {10.1016/j.disc.2022.113304}
}

@book{babai_linear_1992,
	title = {Linear {Algebra} {Methods} in {Combinatorics}, with {Applications} to {Geometry} and {Computer} {Science}},
	language = {en},
	publisher = {University of Chicago, Department of Computer Science},
	author = {Babai, Laszlo and Frankl, Peter},
	year = {1992},
        url = {https://people.cs.uchicago.edu/~laci/babai-frankl-book2022.pdf}
}

@article{deza_intersection_1978,
  title={Intersection properties of systems of finite sets},
  author={Deza, Michel and Erd\H{o}s, Paul and Frankl, P{\'e}ter},
  journal={Proceedings of the London Mathematical Society},
  volume={3},
  number={2},
  pages={369--384},
  year={1978},
  publisher={Oxford University Press},
    doi = {10.1112/plms/s3-36.2.369}
}

@article{frankston2021thresholds,
  title={Thresholds versus fractional expectation-thresholds},
  author={Frankston, Keith and Kahn, Jeff and Narayanan, Bhargav and Park, Jinyoung},
  journal={Annals of Mathematics},
  volume={194},
  number={2},
  pages={475--495},
  year={2021},
  publisher={Department of Mathematics, Princeton University Princeton, New Jersey, USA},
    doi = {10.4007/annals.2021.194.2.2}
}

@misc{kupavskii_noskov_followup,
    author = {Kupavskii, Andrey and Noskov, Fedor},
    title = {Turan-type problems under fully polynomial dependencies},
    note = {To appear},
    year = {2025}
}

@article{rao2023sunflowers,
  title={Sunflowers: from soil to oil},
  author={Rao, Anup},
  journal={Bulletin of the American Mathematical Society},
  volume={60},
  number={1},
  pages={29--38},
  year={2023},
    doi = {10.1090/bull/1777}
}

\end{document}